\documentclass[11pt]{article}
\usepackage[mathletters]{ucs}
\usepackage[utf8]{inputenc}

\usepackage{amsthm}
\usepackage{amssymb,amsmath}
\usepackage{todonotes} 
\usepackage{cleveref}

\def\bC{\mathbb{C}}

\def\bR{\mathbb{R}}

\def\bT{\mathbb{T}}
\def\bZ{\mathbb{Z}}
\def\cD{\mathcal{D}}
\def\cE{\mathcal{E}}

\def\cH{\mathcal{H}}
\def\cI{\mathcal{I}}
\def\cL{\mathcal{L}}
\def\cO{\mathcal{O}}
\def\cP{\mathcal{P}}
\def\cQ{\mathcal{Q}}

\def\cS{\mathcal{S}}
\def\cT{\mathcal{T}}
\def\cX{\mathcal{X}}
\def\gF{\mathfrak{F}}

\newtheorem{theorem}{Theorem}[section]
\newtheorem{remark}[theorem]{Remark}
\newtheorem{lemma}[theorem]{Lemma}
\newtheorem{proposition}[theorem]{Proposition}

\newtheorem{definition}[theorem]{Definition}
\newtheorem{assumption}{Assumption}[theorem]
\newtheorem{notation}[theorem]{Notation}

\DeclareMathOperator\Id{Id}
\DeclareMathOperator\Tr{Tr}
\def\<{\langle} \def\>{\rangle} 
\def\diff{\mathrm{d}}
\usepackage{bbm} \def\bbone{\mathbbm{1}}
\def\sm{\setminus}
\DeclareMathOperator*\argmin{argmin}
\def\ve{\varepsilon}

\def\Burgers{\mathrm{B}}
\def\Porous{\mathrm{P}}

\newtheorem*{assumption*}{Assumption}
\def\brho{\boldsymbol{\rho}}
\def\bm{\boldsymbol{m}}
\def\bu{\boldsymbol{u}}
\DeclareMathOperator\prox{prox}
\DeclareMathOperator\diver{div}

\def\odd{\mathrm{odd}}\def\even{\mathrm{even}}

\def\Re{\mathrm{Re}}
\DeclareMathOperator\erfc{erfc}
\DeclareMathOperator\supp{supp}

\def\sign{\epsilon}
\def\Diff{\mathrm{D}}
\def\vp{\varphi}

\usepackage{url}

\title{Discretization and convergence of the ballistic Benamou-Brenier
formulation of the porous medium and Burgers' equations.}

\author{
Jean-Marie Mirebeau\footnote{
University Paris-Saclay, ENS Paris-Saclay, CNRS, Centre Borelli, F-91190 Gif-sur-Yvette, France.
},
Erwan Stampfli\footnote{
University Paris-Saclay, Laboratoires de Mathématiques d'Orsay, 91400 Orsay, France.
}
}

\crefname{assumption}{Assumption}{Assumptions} \crefname{notation}{Notation}{Notations}
\crefname{theorem}{Theorem}{Theorems}
\crefname{remark}{Remark}{Remarks}
\crefname{lemma}{Lemma}{Lemmas}
\crefname{proposition}{Proposition}{Propositions}
\crefname{corollary}{Corollary}{Corollaries}
\crefname{definition}{Definition}{Definitions}

\begin{document}
\maketitle
\begin{abstract}
We study the discretization, convergence, and numerical implementation of recent reformulations of the quadratic porous medium equation (multidimensional and anisotropic) and Burgers' equation (one-dimensional, with optional viscosity), as forward in time variants of the Benamou-Brenier formulation of optimal transport. 
This approach turns those evolution problems into global optimization problems in time and space, of which we introduce a discretization, one of whose originalities lies in the harmonic interpolation of the densities involved.
We prove that the resulting schemes are 
unconditionally stable w.r.t.\ the space and time steps, 
and we establish a quadratic convergence rate for the dual PDE solution, 
under suitable assumptions. We also show that the schemes can be efficiently solved numerically using a proximal splitting method and a global space-time fast Fourier transform, and we illustrate our results with numerical experiments. 
\end{abstract}




\section{Introduction}

A number of non-linear partial differential equations (PDEs) can be reformulated as global convex optimization problems in time and space \cite{brenier2020examples}, revealing so-called \emph{hidden convexity} properties.
We follow in this paper an approach whose properties are particularly appealing for evolution problems featuring quadratic non-linearities \cite{vorotnikov2022partial}, already investigated numerically for Burgers' equation \cite{brenier2020examples,kouskiya2024inviscid} and for non-linear elasticity \cite{singh2024hidden}. 
In the continuous setting, its benefits include new notions of weak solutions, and related methods for proving their existence and/or uniqueness.
In the discrete setting, it leads to a radically new method for solving evolution problems, of which this paper provides the first numerical convergence analysis, with schemes free of any Courant-Friedrichs-Levy (CFL) condition, second-order accurate in time and space,  solved using proximal optimization algorithms and global space-time fast-Fourier transforms. 
As discussed below, this approach shares similarities with the Benamou-Brenier dynamic  formulation of optimal transport 
\cite{benamou2000fluidOT}
,  whose numerical analysis is studied in e.g.\

\cite{lavenant2021unconditional,natale2021computation} or \cite{liConvergenceDiscreteDynamic2024} in the matrix-valued setting. Our numerical implementation eventually relies on proximal splitting, inspired by 
\cite{papadakis2014optimal}
which follows that approach for the dynamic formulation of optimal transport, see \cite{lavenant2021unconditional} 
for a history and comparison of alternative numerical methods.
While \cite{vorotnikov2022partial,brenier2020examples} mostly focus on \emph{systems} of PDEs, often related to fluid mechanics, we shall limit in this paper our attention to two \emph{scalar} PDEs for simplicity, see \eqref{eq:QPME_Burgers}. Before that, let us further motivate the approach and replace it within the literature. 

The initial value problem for many non-linear evolution PDEs may admit infinitely many weak solutions for most initial conditions
\cite{diperna1979uniqueness,de2009euler,de2010admissibility}.
For the Euler equations of incompressible fluids and, also, for the class of hyperbolic systems of conservation laws with
a convex entropy, Brenier \cite{brenier2018initial} considered the following problem:
given an initial condition $u_0$ and a finite time interval $[0,T]$,
find among all weak solutions
starting from $u_0$ one that minimizes the time integral of the kinetic energy (for Euler) or the entropy (for systems of conservation laws).
This problem might have no solution, but it generates a dual variational problem which is automatically convex and usually admits a
solution, that we call \emph{dual solution}. Brenier obtained an explicit formula relating this dual solution to any tentative smooth solution of the
initial value problem, at least when $T$ is not too large, see \cref{remark:finalT}. It was also pointed out that dual problems of this kind
are strongly related to optimal transport problems (in their `Benamou-Brenier' formulation
\cite{benamou2000fluidOT}).
and can be interpreted as generalized (with matrix-valued densities) Mean Field Games 
\`a la Lasry-Lions. The method of \cite{brenier2018initial} was further developped by Vorotnikov 
\cite{vorotnikov2022partial} for a very large class of quadratic evolution PDEs and
a similar duality approach has been used by Acharya and collaborators \cite{acharya2023variational,acharya2024action,acharya2024variational}.
Let us finally mention \cite{brenier2022relaxed} where the dual technique is
applied to a multi-stream formulation of the Euler-Poisson system and \cite{brenier2020examples} where various parabolic PDEs are considered,
including the QPME, for which there is no restriction on the size of $T$ and which is the main subject of the present paper.\\

As announced, this paper is devoted to the discretization and convergence analysis of dual reformulations of two scalar evolution PDEs featuring quadratic non-linearities:
the quadratic porous medium equation (QPME, multidimensional and anisotropic), and Burgers' equation with optional viscosity (one dimensional).
\begin{align}
\label{eq:QPME_Burgers}
\text{QPME : }
	\partial_t u &= \tfrac 1 2 \diver(\cD \nabla u^2), &
	\text{Burgers' : }
	\partial_t u + \tfrac{1}{2} \partial_x u^2 = \nu \partial_{xx} u.
\end{align}
They are posed for simplicity on the periodic unit box $\bT^d$, with $d=1$ for Burgers' equation, where $\bT:=\bR/\bZ$ is the periodic unit interval. 
See \cite{kouskiya2024inviscid} for a closely related approach implementing non-periodic boundary conditions in Burgers' case.
The PDE parameters are a smooth positive definite tensor field $\cD \in C^\infty(\bT^d,\cS_d^{++})$, following \cite{li2018anisotropic}, and a non-negative diffusion coefficient $\nu \geq 0$ respectively. A smooth and positive initial condition $u_0 \in C^\infty(\bT^d,]0,\infty[)$ and a bounded time interval $[0,T]$ are also given. 
We make these strong assumptions for the sake of the numerical analysis, even though the PDEs of interest admit solutions under weaker assumptions \cite{vazquez2007porous}. Our numerical experiments \cref{sec:num} also show that the proposed schemes are robust and appear to work without these assumptions.  

The two PDEs \eqref{eq:QPME_Burgers} share a common structure, with a quadratic non-linearity:
\begin{align}
\label{eq:pdeAL}
	\partial_t u + \tfrac 1 2 L(u^2) + L A u &= 0
	\quad \text{on } [0,T], &
	u(0) &= u_0,
\end{align}
where $L = -\diver(\cD\nabla\cdot)$ and $A=0$ for the QPME, and $L = \partial_x$ and $A = -\nu \partial_x$ for Burgers' equation.
Following \cite{brenier2020examples,vorotnikov2022partial} we formally consider the (apparently trivial) problem \eqref{eq:minL2weak} of minimizing the $L^2$ norm among all weak solutions of \eqref{eq:pdeAL}; variants featuring a time-dependent weight or a base state can also be considered, see \cite{vorotnikov2025hidden,acharya2024variational} and \cref{remark:finalT}. Formally, again, we exchange in \eqref{eq:gap} the $\sup$ and $\inf$, and perform an integration by parts \begin{align}
\hspace{-5mm}
\label{eq:minL2weak}
	&\inf_{u(0) = u_0} \sup_{\phi(T)=0}\int_{[0,T]\times \bT^d} \tfrac 1 2 u^2 + (\partial_t u + \tfrac 1 2 L(u^2) + L A u) \phi \\
\label{eq:gap}
	&\overset{?}{=}\!\! 
	\sup_{\phi(T)=0} \inf_{u(0) = u_0} \int_{[0,T]\times \bT^d}
	\tfrac 1 2 u^2 (1+L^*\phi) - (u-u_0)\partial_t \phi + u A^*L^*\phi.
\end{align}
Here and below, given a bivariate mapping $u(t,x)$, we denote by $u(t) := u(t,\cdot)$ the univariate function of the space variable alone, where the time argument has been frozen to some given value $t$.
Note that One always has \eqref{eq:minL2weak} $\geq$ \eqref{eq:gap}, and the difference \eqref{eq:minL2weak} $-$ \eqref{eq:gap} is referred to as the duality gap.
We make at this point the critical assumption that $1+L^*\phi>0$ everywhere on $[0,T] \times \bT^d$, which is closely related with the lack of duality gap, see \cref{remark:finalT}.
Minimizing pointwise \eqref{eq:gap} w.r.t.\ the primal variable $u$, we obtain the non-trivial dual problem \cite{brenier2020examples,vorotnikov2022partial,kouskiya2024inviscid}:
\begin{align}	
\label{eq:energyPhi}
	&\inf_{\phi(T)=0}\int_{[0,T]\times \bT^d} \frac{(\partial_t \phi - A^*L^*\phi)^2}{2(1+L^*\phi)} - u_0 \partial_t \phi.
\end{align}
The primal and dual variables $u$ and $\phi$ obey the following optimality relation 
\begin{align}
\label{eq:pdePhi}
	\partial_t \phi - A^*L^*\phi &= (1+L^*\phi) u, 
	&\phi(T)&=0,
\end{align}
which can be regarded as a reversed in time linear evolution problem w.r.t.\ $\phi$.
(For large $T$, this may contradict our assumption $1+L^*\phi>0$, see \cref{remark:finalT}.)
Defining a pseudo-momentum and pseudo-density as 
 \begin{align}
\label{eq:mPhiRho}
	m &:= \partial_t \phi, &
	\rho &:= 1+L^*\phi,
\end{align}
we obtain as announced a reformulation of the  evolution PDE \eqref{eq:pdeAL} as an optimization problem which is global in time and space, and is close in spirit to the Benamou-Brenier formulation of optimal transport. Indeed, assuming $A^* 1 = 0$ which holds in the two cases of interest, \eqref{eq:energyPhi} is reformulated as 
\begin{align}
\label{eq:pdeMRho}
	&\inf_{m,\rho} \int_{[0,T]\times \bT^d} \frac{(m-A^*\rho)^2}{2\rho} - u_0 m &
	&\text{subject to } \partial_t \rho = L^*m \text{ and } \rho(T)=1.
\end{align}
The solution of the original problem \eqref{eq:QPME_Burgers} can be recovered as $u=(m-A^*\rho)/\rho$ at all points where $\rho>0$, in view of \eqref{eq:pdePhi} and \eqref{eq:mPhiRho}.
Specializing the above expressions to the models of interest, we obtain the following relations between the primal variable $u$ and dual pseudo-momentum $m$ and pseudo-density $\rho$, and the following dual optimization problems:
\begin{align}
\label{eq:EnergyQPME}
\text{QPME : } u &= \frac{m}{\rho}, &
	\inf_{m,\rho}& \int_{[0,T]\times \bT^d} \frac{m^2}{2\rho} - u_0 m
	\ \text{ s.t. } \partial_t \rho = \diver(\cD m) \ \text{ and } \rho(T)=1.\\
\nonumber
\text{Burgers' : } u&=\frac{ m-\nu \partial_x \rho}{\rho}, &
	\inf_{m,\rho} &\int_{[0,T]\times \bT} \frac{(m-\nu \partial_x \rho)^2}{2 \rho} - u_0 m 
	\ \text{ s.t. } \partial_t \rho = - \partial_x m \text{ and } \rho(T)=1.
\end{align}
We refer to \cite{brenier2020examples} and \cite[Appendix B]{vorotnikov2025hidden} for further discussion of the analogy with optimal transport, and the `ballistic' terminology.

As announced, this paper is devoted to a numerical convergence analysis of the energetic formulation \eqref{eq:pdeMRho} of the QPME and Burgers' equation \eqref{eq:QPME_Burgers}. 
For simplicity, we discretize the unknowns on Cartesian grids. For stability and accuracy, we use staggered finite differences in time and space, and for that purpose we define
$\cT_\tau, \cT'_\tau \subset [0,T]$ and $\bT_h,\bT'_h \subset \bT$ as 
\begin{align*}
	\cT_\tau &:= \{0,2\tau,\cdots,T\}, &
	\cT'_\tau &:= \{\tau,3\tau,\cdots,T-\tau\}, \\
	\bT_h &:= \{0,2h,\cdots,1-2h\},&
	\bT'_h &:= \{h,3h,\cdots,1-h\}.
\end{align*}
where $\tau>0$ and $h>0$ respectively denote the \emph{half} timestep, and \emph{half} gridscale. 
The grids $\cT_\tau$ and $\bT_h$ are \emph{centered}, whereas $\cT'_\tau$ and $\bT'_h$ are \emph{staggered}.
The following assumption is implicit throughout the paper. 
\begin{assumption}
\label{assum:step}
	The grid sizes $N_\tau := T/(2\tau)$ and $N_h := 1/(2h)$ are positive integers.
\end{assumption}
The staggered first order time finite difference, and the first order spatial finite difference in the direction $e \in \bZ^d$ (which may or may not be staggered depending on the parity of the components of $e$), are defined as 
\begin{align}
\label{eqdef:DtDhe}
	\partial_\tau u(t,x) &:= \frac{u(t+\tau,x)-u(t-\tau,x)}{2\tau}, &
	\partial_h^e u(t,x) &:= \frac{u(t,x+he)-u(t,x-he)}{2h}.
\end{align}
The one-dimensional spatial finite difference operators are denoted $\partial_h := \partial_h^{(1)}$ and $\partial_{hh} := \partial_h \partial_h$.
We also use the averaged $\ell^1(\bT_h^d)$ norm, and the \emph{perspective function} $\cP : \bR^2 \to [0,\infty]$ \cite{combettes2018perspective}:
\begin{align}
\label{eqdef:perspective}	
	\|f\|_{\ell^1(\bT_h^d)} &:= 
	(2h)^d \sum_{x \in \bT_h^d} |f(x)|,
	&
	\cP(m,\rho) &:= 
	\begin{cases}
		\frac{m^2}{2\rho} & \text{if } \rho >0,\\
		0 & \text{if } m=\rho=0,\\
		\infty & \text{else}.
	\end{cases} 
\end{align}
The ratios appearing in the energies \eqref{eqdef:porousEnergy} and \eqref{eqdef:BurgersEnergy} should be understood as instances of $\cP$, which is convex and lower semi-continuous in view of the identity $\cP(m,\rho) = \sup\{a \rho + b m\mid a+\tfrac 1 2 b^2 \leq 0\}$.

\paragraph{Discretization of the QPME.}
A preliminary step is the construction of a finite differences approximation of the anisotropic divergence-form Laplacian operator $-L := \diver(\cD \nabla \cdot)$.
In the isotropic case where $\cD = \Id$ identically, the usual finite differences Laplacian $-L_h := \sum_{1 \leq i \leq d} \, (\partial_h^{e_i})^2$ should be used, where $(e_1,\cdots,e_d)$ denotes the canonical basis, in expanded form $-L_h u(x) := (2h)^{-2} \sum_{1 \leq i \leq d} [u(x+2h e_i)-2 u(x) + u(x-2h e_i)]$. 
In the anisotropic case, we use an adaptive finite differences scheme with similar structural properties: for any $u : \bT_h^d \to \bR$
\begin{align}
\label{eqdef:Lh}	
	- L_h u &:= \sum_{e \in E} \partial^e_h(\lambda^e\partial_h^e u), &
	\text{where } \cD &= \sum_{e \in E} \lambda^e e e^\top.
\end{align}
We denoted by $E \subset \bZ^d$ a finite set referred to as the stencil, and by $\lambda^e \in C^\infty(\bT^d,[0,\infty[)$ some smooth non-negative coefficients obeying (\ref{eqdef:Lh}, right), see
\cref{rem:decompSDP}.
 Expanding (\ref{eqdef:Lh}, left) we obtain for all $x \in \bT^d_h$ (note that $x+2h e \in \bT^d_h$ for any $e \in E$) 
	\begin{equation}
	\label{eq:LhExpl}
		- L_h u(x) = (2h)^{-2} \sum_{e \in E} 
		\Big[\lambda^e(x+he) \big(u(x+2h e)-u(x)\big) 
		+ \lambda^e(x-h e) \big(u(x-2h e)-u(x)\big)\Big].
	\end{equation}
The next lemma provides a counterpart, for the QPME in the discrete setting, of the change of unknowns \eqref{eq:mPhiRho}. 
\begin{lemma}
\label{lem:domPorous}
	The set $\Phi^\Porous_{\tau h}:= \{\phi \in \bR^{\cT_\tau \times \bT^d_h} \mid \phi(T)=0\}$
	is in bijection with 
	$\{(m,\rho) \in \bR^{\cT'_\tau \times \bT^d_h} \times \bR^{\cT_\tau \times \bT^d_h} \mid \partial_\tau \rho = L_h m,\ \rho(T)=1\}$, 
	via $\phi \mapsto (\partial_\tau \phi, 1+L_h \phi)$. 
\end{lemma}
We omit the proof, which is a direct consequence of the time-independence of the coefficients $\lambda^e$, for all $e \in E$, and of the commutation of the time and space finite differences \eqref{eqdef:DtDhe}.
Using these notations, we discretize the dual energy \eqref{eq:EnergyQPME} as follows
\begin{align}
\label{eqdef:porousEnergy}
	\cE^\Porous_{\tau h}(m,\rho) := 2\tau (2h)^d \sum_{\substack{t \in \cT'_\tau\\x\in \bT_h^d}} \Big(\frac{1}{2} \sum_{\sigma =\pm} \frac{m(t,x)^2}{2\rho(t+\sigma \tau ,x)} - m(t,x) u_0(x)\Big).
\end{align}
The second sum symbol features two contributions, corresponding to the two possible signs $\sigma=+$ and $\sigma=-$. 
By convention we set $\cE^\Porous_{\tau h}(\phi) := \cE^\Porous_{\tau h}(\partial_\tau \phi, 1+L_h \phi)$, and we show in \cref{prop:porousStrictCvx} that this energy is a strictly convex function of $\phi \in \Phi_{\tau h}^\Porous$.
Let us emphasize that the energy \eqref{eqdef:porousEnergy} implicitly involves the \emph{harmonic mean} of the density field $\rho$, at the times $t + \tau$ and $t-\tau$ associated with the two possible signs $\sigma=\pm$. This choice is deliberate, and we show in \cref{sec:arith} that using the arithmetic mean instead 
as in \cite{papadakis2014optimal,natale2021computation,liConvergenceDiscreteDynamic2024} 
leads to lesser stability properties, and to a degeneracy of the optimization problem, which becomes local in time rather than global \eqref{eq:arithScheme}.

\begin{assumption}
\label{assum:porous}
We assume that $u \in C^\infty([0,T]\times \bT^d,]0,\infty[)$ is a smooth positive solution of the QPME (\ref{eq:QPME_Burgers}, left), associated with a smooth and symmetric positive definite diffusion tensor $\cD \in C^\infty(\bT^d,\cS_d^{++})$. We denote by $\phi \in C^\infty([0,T]\times \bT^d)$ the solution of the reversed in time parabolic PDE: $\partial_t \phi = (1-\diver(\cD \nabla \phi))u$ with terminal boundary condition: $\phi(T)=0$, and we assume that $1-\diver(\cD\nabla \phi)$ is \emph{positively bounded below}.
The numerical scheme stencil $E \subset \bZ^d\sm\{0\}$ is finite, and the coefficients $\lambda^e \in C^\infty(\bT^d,[0,\infty[)$ obey (\ref{eqdef:Lh}, right) and are smooth for all $e \in E$.
\end{assumption}

\begin{theorem}
\label{th:cvPorous}
Under \cref{assum:porous}. For all $0 < \tau \leq \tau_*$ and $0<h\leq h_*$, there exists a unique minimizer $\tilde \phi \in \Phi_{\tau h}^\Porous$ of the energy $\cE_{\tau h}^\Porous$. It satisfies 
\begin{equation*}
	\max_{t \in \cT_\tau} \|\tilde \phi(t)- \phi(t)\|_{\ell^1(\bT^d_h)} \leq C_* (\tau^2 + h^2).
\end{equation*}
The constants $\tau_*,h_*>0$ and $C_*$ only depend on $(\lambda,\phi)$. 
\end{theorem}

%
A uniformly converging approximation of the primal PDE solution $u = m/\rho = \partial_t \phi / (1-\diver(\cD \nabla \phi))$ can also be recovered from $\tilde \phi$, using a regularization by convolution to estimate its derivatives, see \cref{sec:conv}. In the numerical experiments \cref{sec:num}, a second order convergence rate is empirically observed for the primal variable $u$, without the need for such a convolution.

\paragraph{Discretization of Burgers' equation.} This equation is built around a first order, rather than second order, differentiation operator in space, and for this reason the problem density $\rho$ is discretized on an appropriately staggered grid.
\begin{lemma}
\label{lem:domBurgers}
	The set $\Phi^\Burgers_{\tau h}:= \{\phi \in \bR^{\cT_\tau \times \bT_h} \mid \phi(T)=0\}$
	is in bijection with 
	$\{(m,\rho) \in \bR^{\cT'_\tau \times \bT_h} \times \bR^{\cT_\tau \times \bT'_h} \mid \partial_\tau \rho + \partial_h m = 0,\ \rho(T)=1\}$, 
	via $\phi \mapsto (\partial_\tau \phi, 1-\partial_h \phi)$.  
\end{lemma}
We again omit the proof, which directly follows from the commutation of the finite difference operators $\partial_\tau$ and $\partial_h$, similarly to
\cref{lem:domPorous}. Using these notations we discretize (\ref{eq:EnergyQPME}, second line): 
\begin{equation}
\label{eqdef:BurgersEnergy}
	\cE_{\tau h}^\Burgers(m,\rho) := 4 \tau h\sum_{\substack{t \in \cT'_\tau\\x\in \bT_h}} 
	\Big(\frac{1}{4} \sum_{\substack{\sigma_t =\pm\\ \sigma_x=\pm}} 
	\frac{\big(m(t,x) - \nu \partial_h \rho(t+\sigma_t \tau, x)\big)^2}{2\rho(t+\sigma_t \tau ,x+\sigma_x h)} - m(t,x) u_0(x)\Big).
\end{equation}
By convention we set $\cE_{\tau h}^\Burgers(\phi) := \cE_{\tau h}^\Burgers(\partial_\tau \phi, 1-\partial_h \phi)$, and we show in \cref{prop:BurgersStrictCvx} that this energy is a strictly convex function of $\phi \in \Phi_{\tau h}^\Burgers$. The energy \eqref{eqdef:BurgersEnergy} implicitly involves the harmonic mean of four values of the density field $\rho$, at all combinations of the times $t \pm \sigma_t \tau$ and positions $x \pm \sigma_x h$, and this is again the key to the proof of stability and convergence.
In the discretization of the Benamou-Brenier formulation of optimal transport, a closely related energy is considered, with $\nu=0$ and $u_0=0$, and a Cauchy rather than ballistic problem in time. Somewhat curiously, the arithmetic mean is usually chosen in this context \cite{papadakis2014optimal,natale2021computation,liConvergenceDiscreteDynamic2024} for the interpolation of the density w.r.t.\ time. 
Let us mention that \cite{natale2021computation} observes improved stability and less oscillations in the solution when using harmonic (as opposed to arithmetic) interpolation in \emph{space}, yet only arithmetic interpolation is considered in time.

\begin{assumption}
\label{assum:Burgers}
We assume that $u \in C^\infty([0,T]\times \bT,]0,\infty[)$ is a smooth positive solution of Burgers' equation (\ref{eq:QPME_Burgers}, right), with non-negative viscosity $\nu \geq 0$ (possibly zero). We denote by $\phi \in C^\infty([0,T]\times \bT)$ the solution of the reversed in time first order linear PDE: $\partial_t \phi + \nu \partial_{xx} \phi = (1-\partial_x \phi)u$ with b.c. $\phi(T)=0$, and we assume that $1-\partial_x \phi$ is \emph{positively bounded below}. 
\end{assumption}

The assumption of a smooth solution is not too restrictive in the case of the QPME, see \cref{assum:porous}, since this PDE has a regularizing effect on positive solutions \cite{vazquez2007porous}, and likewise for Burgers' equation with positive viscosity.
In contrast, it appears highly questionable for the inviscid Burgers' equation ($\nu=0$), which develops a shock in finite time $T_{\mathrm{shock}}>0$, even if the initial condition is smooth.
However, strikingly and unfortunately, it turns out that the BBB approach on the interval $[0,T]$ does \emph{not} characterize the (viscosity) solution $u$ to Burgers' equation if $T> T_{\mathrm{shock}}$, but rather a modified solution $u^T$ depending on the final time $T$.
More precisely,
\cite[Theorem 5.5.2]{brenier2020examples}
characterizes $u^T$ as the unique shock free solution of Burgers' equation on $[0,T]$ such that $u^T(T,\cdot) = u(T,\cdot)$, see 
\cref{fig:inviscid}
page \pageref{fig:inviscid} for a numerical illustration and additional discussion.
The correct entropy solution is therefore still obtained at the \emph{final} time $t=T$, although it is not for $t<T$.
The critical assumption $1-\partial_x \phi > 0$ is also violated if $T>T_{\mathrm{shock}}$.)
A numerical method for Burgers' equation proposed in \cite{kouskiya2024inviscid} adresses this issue by concatenating the numerical solutions of the BBB formulation discretized over small sub-intervals of time, see \cref{remark:finalT}. Proving convergence in this setting is however outside the scope of this paper, and we thus stick to \cref{assum:Burgers}, which implicitly implies $T< T_{\mathrm{shock}}$ in the inviscid case.

\begin{theorem}
\label{th:cvBurgers} Under \cref{assum:Burgers}. For all $0 < \tau \leq \tau_*$ and $0<h\leq h_*$, there exists a unique minimizer $\tilde \phi \in \Phi_{\tau h}^\Burgers$ of the energy $\cE_{\tau h}^\Burgers$. It satisfies 
\begin{align*}
	\max_{t \in \cT_\tau}\, \| \tilde \phi(t)-\phi(t)\|_{\ell^1(\bT_h)} 
	\leq C_* (\tau^2 + h^2).
\end{align*}
The constants $\tau_*,h_*>0$ and $C_*$ only depend on $\phi$. 
\end{theorem}

\paragraph{Minimization algorithm and implementation.}
We numerically minimize the energy \eqref{eqdef:porousEnergy} (resp.\ \eqref{eqdef:BurgersEnergy}), w.r.t.\ the momentum $m$ and density $\rho$ variables subject to the continuity equation discretized as in \cref{lem:domPorous} (resp.\ \cref{lem:domBurgers}). For that purpose, we use the classical Chambolle-Pock \cite{chambolle2011primaldual} proximal splitting primal-dual optimization algorithm, similarly to \cite{papadakis2014optimal}. This requires a reformulation of the objective functional using carefully chosen additional variables. 
(A Newton solver can also be used for small problem sizes, following \cite{kouskiya2024inviscid}.)

For instance we introduce an anti-symmetric (resp.\ symmetric) extension of the optimization unknown $m$ (resp.\ $\rho$) to negative times, which has two benefits. First, it decouples the perspective functions appearing in \eqref{eqdef:porousEnergy}, so that the corresponding proximal operator is computable independently for each point in time and space. 
Second, it allows using a global FFT in time and space, in order to compute the orthogonal projection onto the linear subspace associated with the discretized continuity equation, see \cref{lem:domPorous,lem:domBurgers}. In a similar spirit, additional unknowns $m_e = \partial_h^e m$, $e\in E$, containing the directional derivatives of the momentum in the stencil directions, are required for the anisotropic QPME.

The proposed numerical method, being embarrassingly parallel in \emph{space and time}, is easily ported to GPUs. Experiments confirm second order accuracy, and illustrate the use of large time steps in contrast with e.g.\ the explicit scheme, which is subject to a parabolic type  Courant-Friedrichs-Levy (CFL) stability condition, or the semi-implicit scheme. See \cref{sec:implem,sec:num}. 


\paragraph{Contributions and comparisons with related work.}
The ballistic Benamou-Brenier (BBB) formulation of non-linear PDEs is a subject of active research in the continuous setting \cite{brenier2020examples,vorotnikov2022partial,vorotnikov2025hidden,acharya2023variational,acharya2024action,acharya2024variational}.
The discretization of this approach has so far only been considered for one-dimensional problems and without proof of convergence under mesh refinement: for Burgers' equation \cite{brenier2020examples,kouskiya2024inviscid}, and for non-linear elasticity \cite{singh2024hidden}.
This paper extends the numerical study to the QPME which is a multidimensional and possibly anisotropic problem, in addition to Burgers' equation, and establishes second order convergence rates when the solution is smooth, see \cref{th:cvPorous,th:cvBurgers}. 

Our results can also be related to the discretization \cite{papadakis2014optimal} and convergence analysis of the Benamou-Brenier formulation of optimal transport 
\cite{benamou2000fluidOT}.
In contrast with some of these works, we assume the existence of a smooth solution of the PDE, which is justified in view of the regularizing effect of the QPME on positive solutions (numerical experiments \cref{sec:num} nevertheless show that the scheme also works on non-smooth compactly supported solutions), and likewise for the viscous Burgers' equation (the inviscid case being subject to an obstruction discussed below \cref{assum:Burgers}). 
Our numerical scheme involves an harmonic (rather than arithmetic) local interpolation of the density w.r.t.\ time, see \cref{sec:arith}, in contrast with earlier works. 
The implementation of the proximal optimization method is modified accordingly, and also to address the ballistic time-boundary conditions, and the possibly anisotropic spatial finite difference operators. 

The discretization of the anisotropy \eqref{eq:LhExpl}, via adaptive decompositions of the anisotropic diffusion tensor field, is in the vein of
\cite{fehrenbach2013diffusion,bonnans2022randers,bonnans2023monotone,mirebeau2025acoustic}.

\paragraph{Summary.}
Our main convergence result \cref{th:cvPorous} for the QPME is established in \cref{sec:cv}, the proximal optimization method for the energies \eqref{eqdef:porousEnergy} and \eqref{eqdef:BurgersEnergy} is described in \cref{sec:implem}, and numerical experiments are presented in \cref{sec:num}. \Cref{sec:Burgers} is devoted to the convergence result \cref{th:cvBurgers} for the discretization of Burgers' equation, whose proof is a small adaptation of the porous
medium case.
\Cref{sec:arith} discusses a modification of the energy \eqref{eqdef:porousEnergy} involving the arithmetic interpolation of the density field, and \cref{sec:conv} deduces from \cref{th:cvPorous,th:cvBurgers} and a mollification argument, a uniform convergence result towards the solution $u$ of the non-linear PDE of interest \eqref{eq:QPME_Burgers}.
Finally, an explicit solution to the BBB formulation of the QPME, corresponding to the Barenblatt profile, is presented in \cref{sec:Barenblatt}, and a well-posedness result for the BBB formulation of Burgers' viscous equation is given in \cref{sec:nogapBV}.

\begin{remark}[Positivity of the density and final evolution time] 
\label{remark:finalT}
The BBB formulation requires that the dual potential $\phi$ obeying the time-reversed evolution PDE \eqref{eq:pdePhi} satisfies (at the very least) $1 + L^*\phi \geq 0$ almost everywhere in time and space, see \eqref{eq:energyPhi}. Depending on the PDE of interest, this condition is typically satisfied for small $T$ in view of the terminal condition $\Phi(T,\cdot) = 0$, but may fail for larger time intervals. 
See \cite{brenier2020examples,vorotnikov2022partial} for more discussion on this topic, which is directly related with the absence of duality gap in \eqref{eq:gap}.

In the case of the QPME, a non-trivial argument involving Aronson-Bénilan estimates \cite[Theorem 5.1.1]{brenier2020examples} fortunately shows that an arbitrarily large time $T$ is admissible. 
Likewise for Burgers' equation with positive viscosity, see \cite[Proposition 5.2.1]{brenier2020examples} and \cref{sec:nogapBV}. 
However, in the case of the inviscid Burgers' equation, the largest admissible time $T$ is typically finite and corresponds to the time of formation of the first shock, as discussed below \cref{assum:Burgers}.

Two approaches aimed at extending the time $T$ of applicability of the BBB formulation, for a selection of PDEs related with fluid mechanics and which are not considered in this paper, have been proposed. Vorotnikov \cite[Remark 4.2]{vorotnikov2025hidden} minimizes
$\int_{[0,T] \times \bT^d} u(t,x)^2 \allowbreak \exp(-\gamma t) \allowbreak \diff x \diff t$ among all weak solutions of the PDE of interest, 
where $\gamma>0$ is a sufficiently large discount factor,
rather than the unweighted kinetic energy \eqref{eq:minL2weak} corresponding to $\gamma=0$. 
Acharya et al \cite{acharya2024variational} minimize $\int_{[0,T] \times \bT^d} (u - \overline u)^2$, where $\overline u$ is a well chosen \emph{base state}, which can be regarded as an initial.
These modifications ensure that there is no duality gap in the corresponding optimization problems, similar to \eqref{eq:gap}, for a large time $T$. 
In numerical applications, 
\cite{kouskiya2024inviscid}
 also replace the full interval $[0,T]$ with a succession of overlapping small sub-intervals.

\end{remark}

 \begin{remark}[Decompositions of symmetric positive definite matrices, with integer offsets]
\label{rem:decompSDP}
For each $D \in \cS_d^{++}$ there exists non-negative weights $\lambda_D : \bZ^d \to \bR$ with finite support $\# \supp(\lambda_D) \leq d (d+1)/2$ such that $D = \sum_{e \in \bZ^d} \lambda_D(e) e e^\top$.
Selling 
\cite{Selling:1874Algorithm}
proposed an explicit construction of such weights in dimension $d \leq 3$, which can be extended to higher dimension using the framework of Voronoi's first reduction 
\cite{voronoi1908parfaites}
, see the discussion in
\cite{bonnans2023monotone}.
One defect of these approaches is that the weights $\lambda_D(e)$ are piecewise linear w.r.t.\ $D$ in Selling's case, and are not uniquely determined in Voronoi's case, hence are non-smooth which can possibly degrade the convergence rate of the numerical schemes in which they are used.
An alternative decomposition $\tilde \lambda_D$ with $C^\infty$ smooth coefficients w.r.t.\ $D \in \cS_d^{++}$ is presented in
\cite{bonnans2023monotone}
in dimension $d=2$ satisfying $\# \supp(\tilde \lambda_D) \leq 4$, and in
\cite{mirebeau2025acoustic}
in dimension $d\geq 3$ satisfying $\# \supp(\tilde \lambda_D) \leq C(d)$ (numerical experiments suggest that $C(3)=13$).
In addition to the present paper, we refer to
\cite{fehrenbach2013diffusion,krylov2005rate,bonnans2023monotone,mirebeau2025acoustic}
and references therein for the construction and the convergence analysis of numerical schemes based on such decompositions of positive definite matrices.
\end{remark}

\section{Convergence analysis, quadratic porous medium equation}
\label{sec:cv}
This section is devoted to the proof of our main result: the unconditional second order convergence rate \cref{th:cvPorous} for the proposed discretization of the quadratic porous medium equation.
We begin by recalling an exact expansion satisfied by the perspective function  \cite{combettes2018perspective}.

\begin{proposition}
\label{prop:perspective}
	Let $m,\hat m, \rho,\hat \rho \in \bR$ be such that $\cP(m,\rho)<\infty$. Then 
	\vspace{-3mm}
\begin{align*}	
	\cP(m+\hat m, \rho+\hat \rho) &= \cP(m,\rho) + 
	(u\hat m  -\tfrac{1}{2} u^2 \hat \rho )
	+ \cP(\hat m -\hat \rho u, \rho + \hat \rho), &
	\text{with } u &:= \begin{cases}
		m/\rho\ \ \ \text{ if } \rho>0,\\
		0\  \text{ if } m=\rho=0.
	\end{cases}
\end{align*}
\end{proposition}
We omit the proof, which follows from direct substitution.
The perspective function is convex and lower-semi-continuous, as already observed below \eqref{eqdef:perspective}, and the vector $(u,-\frac{1}{2} u^2)$ is its gradient if $\rho>0$, and a subgradient element if $\rho=m=0$. 


\begin{notation}
\label{not:porous}
We fix in the following the half timestep $\tau>0$ and half gridscale $h>0$. For readability, and in order to avoid multiple sub- or super-scripts, we denote by $\cE := \cE_{\tau h}^\Porous$ the energy of interest \eqref{eqdef:porousEnergy} and by $\Phi := \Phi_{\tau h}^\Porous$ its domain, see \cref{lem:domPorous}. Likewise, we denote by $\cT := \cT_\tau$ and $\cT':=\cT'_\tau$ the centered and staggered grids in time, and by $\cX := \bT_h^d$ the centered grid in space. Finally, we denote by $\|\cdot\|_1 := \|\cdot\|_{\ell^1(\cX)}$ the normalized $\ell^1$ norm, see \eqref{eqdef:perspective}. 
\end{notation}

Using \cref{prop:perspective} we obtain
\begin{equation}
\label{eqdef:LQ}
	\cE(\phi+\hat \phi) = \cE(\phi) + \cL(\hat \phi; \phi) + \cQ(\hat \phi; \phi),
\end{equation}
for any $\phi,\hat \phi \in \Phi$ such that $\cE(\phi) < \infty$. 
The linear and second order (non-quadratic) terms are
\begin{align}
\nonumber
	\cL(\hat \phi; \phi)
	&:= 
	\tau (2h)^d \! \sum_{x \in \cX} \sum_{t \in \cT'}
	 \sum_{\sigma=\pm} 
	\Big[u^\sigma(t,x) \hat m(t,x) - \tfrac 1 2 u^\sigma(t,x)^2 \hat \rho(t+\sigma \tau,x) \Big]  
	+ (2h)^d \! \sum_{x\in \cX} u_0(x) \hat \phi(0,x),\\
\nonumber
	\cQ(\hat \phi; \phi) &:= \tau (2h)^d
	\sum_{x \in \cX} \sum_{t \in \cT'} \sum_{\sigma=\pm} 
	\cP\big(\hat m(t,x)-\hat \rho(t+\sigma \tau,x) u^\sigma(t,x),\, \rho(t+\sigma \tau,x)+\hat \rho(t+\sigma \tau,x)\big),\\
\label{eqdef:mhm}
	\text{where} & \
	m := \partial_\tau \phi,\ 
	\rho := 1+L_h \phi,\ 
	\hat m := \partial_\tau \hat \phi,\ 
	\hat \rho := L_h \hat \phi, 
	\text{ and } u^\sigma(t,x) := \frac{m(t,x)}{\rho(t+\sigma \tau,x)}.
\end{align}
Note that $\phi,\hat \phi,\rho,\hat \rho$ are defined on $\cT \times \cX$, whereas $m,\hat m, u^\sigma$ are defined on $\cT' \times \cX$ for any sign $\sigma=\pm$. 
The ratio defining $u^\sigma$ should be understood in the sense of 
\cref{prop:perspective}
in general; in the proof however, under
\cref{assum:porous} 
and thanks to 
\cref{prop:consistency}
below, the denominator $\rho$ is positive.
We used the telescopic sum $2\tau \sum_{t \in \cT'} m(t,x) = -\phi(0,x)$ to obtain $\cL$. We next deduce from this expression the strict convexity of the energy.

\begin{proposition}
\label{prop:porousStrictCvx}
The energy $\cE$ is strictly convex on the domain $\Phi$. 	
\end{proposition}

\begin{proof}
The convexity of $\cE$ follows from the convexity of the perspective function \eqref{eqdef:perspective}, but proving strict convexity requires additional arguments.

 Consider $\phi,\hat \phi\in \Phi$ such that $\cE(\phi+s\hat \phi) = \cE(\phi) + s \cL(\hat \phi; \phi) + \cQ(s \hat \phi; \phi)$ is finitely valued and affine w.r.t.\ $s \in [0,1]$, i.e.\ $\cQ(s \hat \phi; \phi)$ is affine. We prove by induction on $n = 0, \cdots, N_\tau := T/(2\tau)$ that $\hat \phi(T-2\tau\, n,x) = 0$ for all $x \in \cX$. 
	This property holds by definition of the domain $\Phi$ for $n=0$. Assuming $\hat \phi(T-2\tau n,x)=0$, for all $x \in \cX$, and denoting $t := T-2\tau n-\tau$ we obtain $\hat m(t,x) = \hat \phi(t-\tau,x)/(2\tau)$ and $\hat \rho(t+\tau,x) = 0$ with the notation \eqref{eqdef:mhm}. Since $\cQ(s\hat \phi;\phi)$  is a sum of convex terms , the following term 
	\begin{equation*}
		\cP\big(s\hat m(t,x)-s\hat \rho(t+\tau,x)u^+(t,x), \rho(t+\tau,x)+s\hat \rho(t+\tau,x)\big) = 
		s^2\cP(\hat \phi(t-\tau,x)/(2\tau),\rho(t+\tau,x))
	\end{equation*}
	(corresponding to $\sigma=+$)
	must be an affine function of $s\in [0,1]$. This implies $\hat \phi(t-\tau,x)=0$, which proves the induction. Thus $\hat \phi=0$, and therefore $\cE$ is strictly convex as announced. 
\end{proof}

 \paragraph{Outline of the proof of
\cref{th:cvPorous}
.}
We study in 
\cref{subsec:LPorous}
the linear term in the expansion \eqref{eqdef:LQ}, and under
\cref{assum:porous} 
we establish that 
\begin{equation}
\label{eq:linL1}
	\cL(\hat \phi; \phi) \leq C_\cL (\tau^2+h^2) \Big(\|\hat \phi(0)\|_1 + \tau \sum_{t \in \mathring \cT} \|\hat \phi(t)\|_1\Big),
\end{equation}
where $\phi$ denotes the solution to the continuous BBB formulation, and $\hat \phi$ is an arbitrary perturbation. The constant $C_\cL = C_\cL(\lambda,\phi)$ only depends on the scheme coefficients $\lambda$ and on $\phi$. The estimate \eqref{eq:linL1} is obtained by rewriting, using discrete integration by parts, the quantity $\cL(\hat \phi; \phi)$ as a discrete version of the QPME applied to the variable conjugate to $\phi$, tested against $\hat \phi$, see \eqref{eq:porousIPP0}, \eqref{eq:porousIPP1} and \eqref{eq:pourousIPP0taylor}. 
The smoothness of $\phi$ is crucial in this estimate, and the second order quadratic rate is achieved thanks to the scheme consistency and evenness w.r.t.\ $\tau$ and $h$.

We next consider in 
\cref{subsec:QPorous}
an implicit discretization of a backwards in time diffusion equation, with unknown $\hat \phi$, diffusivity $u^-$, and r.h.s.\ $r^-$:
\begin{align}
\label{eq:implDiff}
 	\forall t \in \cT',\ \partial_\tau \hat \phi(t) - u^-(t) L_h \hat \phi(t-\tau) &= r^-(t),&
 	\hat \phi(T)&=0.
\end{align}
Using the degenerate ellipticity of the finite differences operator $L_h$, the non-negativity of $u^-$ and the boundedness of $L_h u^-$, and a Duhamel type formula for the solution $\hat \phi$, we obtain 

\begin{align}
\label{eq:phiR}
	\forall t \in \cT, \ \|\hat \phi(t)\|_1 &\leq C_\cQ R, &
	\text{where } R &= 2 \tau \sum_{t\in \cT'} \|r^-(t)\|_1,
\end{align}
for $0<\tau \leq \tau_*$ and $0<h\leq h_*$, where the positive constants $C_\cQ, \tau_*, h_*$ again only depend on $\lambda$ and $\phi$. 
We then estimate $\cL$ and $\cQ$ in terms of $R$, which is the averaged r.h.s.\ of \eqref{eq:implDiff}: denoting $C'_\cL := (1+T) C_\cL C_\cQ$ we have 
\begin{align}
\label{eq:EstimLR}
	\cL(\hat \phi; \phi) &\geq -C'_\cL (\tau^2+h^2) R, &
	R^2 &\leq 2T \cQ(\hat \phi; \phi).
\end{align}
The first estimate uses \eqref{eq:linL1} applied to $\cL(-\hat \phi; \phi)$, and \eqref{eq:phiR}. 
The second estimate is obtained by recognizing $r^-$ as the first argument of $\cP$ in the sum \eqref{eqdef:mhm} defining $\cQ(\hat \phi; \phi)$ (choosing the term $\sigma=-$), 
and using a simple argument based on the Cauchy-Schwartz inequality, see \eqref{eq:R2Q} below.

Finally, let us assume that the perturbation $\hat \phi$ of the continuous BBB solution $\phi$ is chosen in such way that $\phi+\hat \phi$ improves the discrete BBB energy $\cE$. Then 
\begin{equation*}
	0 
	\geq \cE(\phi+ \hat \phi) - \cE(\phi) 
	= \cL(\hat \phi; \phi) + \cQ(\hat \phi; \phi)
	\geq - C'_\cL (\tau^2+h^2) R + R^2/(2T).
\end{equation*}
It follows that $R \leq 2T C'_\cL (\tau^2+h^2)$, and therefore $\|\hat \phi(t)\|_1 \leq C_\cQ R \leq C_*(\tau^2+h^2)$ with $C_* := 2T C'_\cL C_\cQ$.
Let us conclude the proof of 
\cref{th:cvPorous}
: we have by the above
\begin{align}
\nonumber
	\inf\{ \cE(\tilde \phi) \mid \tilde \phi \in \Phi\} 
	&= \inf\{\cE(\phi+\hat \phi) \mid \hat \phi \in \Phi,\, \cE(\phi+\hat \phi) \leq \cE(\phi)\} \\
	&= \inf\{\cE(\phi+\hat \phi)\mid \hat \phi \in \Phi, \forall t \in \cT, \|\hat \phi(t)\|_1 \leq C_*(\tau^2+h^2)\}.
	\label{eq:minPclose}
\end{align}
The latter infimum is attained, since $\cE$ is lower-semi-continuous (as a sum of perspective functions) and the optimization domain is compact. The corresponding minimizer is global,  is unique since $\cE$ is strictly convex by
\cref{prop:porousStrictCvx}
, and it obeys the  estimate announced in
\cref{th:cvPorous}.

 \subsection{Upper estimate of the linear term}
\label{subsec:LPorous}
This section is devoted to the proof of \eqref{eq:linL1}, which estimates the linear term $\cL(\hat \phi; \phi)$ in terms of the $\ell^1(\cX)$ norm of the perturbation $\hat \phi$, when $\phi$ is a smooth solution to the continuous problem following \cref{assum:porous}.
For that purpose, we first rewrite $\cL(\hat \phi; \phi)$ as finite differences of $\phi$ tested against $\hat \phi$, using discrete summations by parts in time and space. 
\begin{lemma}
\label{lem:timeIPP}
For any $f : \cT \to \bR$ and $g : \cT' \to \bR$, one has denoting $\mathring\cT := \cT \sm \{0,T\}$
\begin{align}
\label{eq:IPP}
	2\tau \sum_{t\in \cT'} g(t)\, \partial_\tau f(t)  
	&= f(T)g(T-\tau) - f(0) g(\tau)
	- 2\tau\sum_{t \in \mathring \cT} f(t)\, \partial_\tau g(t).
\end{align}
\end{lemma}

The proof of this basic and classical identity is omitted, and we turn to the spatial counterpart \eqref{eq:LhAdj}, which incidentally shows that $L_h$ is positive semi-definite. This property is used in \cite[Appendix C]{bonnans2023monotone} to establish convergence rates for the discretization of an elliptic equation, but it is not used here since we instead rely on degenerate ellipticity, see \cref{prop:implLap}.
\begin{lemma}
\label{lem:LhAdj}
The operator $\partial_h^e$ is skew-adjoint, for all $e \in E$, in the sense that 
\begin{equation}
\label{eq:DheAdj}
\sum_{x \in \cX} u(x) \partial_h^e w(x) = \frac{1}{2 \tau} \sum_{x \in \cX} u(x) w(x+h e)- \frac{1}{2\tau} \sum_{x \in \cX} u(x) w(x-he) = -\sum_{y \in \cX+h e} w(y) \partial_h^e u(y)	
\end{equation}
for all $u \in \bR^\cX$, $w\in \bR^{\cX+h e}$. 
The operator $L_h$ is self-adjoint: for all $u,v \in \bR^\cX$
	\begin{align}
	\label{eq:LhAdj}
		\sum_{x \in \cX} u(x) L_h v(x) &= \sum_{e \in E} \sum_{y \in \cX+h e} \lambda^e(y)\, \partial_h^eu(y)\, \partial_h^ev(y) = \sum_{x \in \cX} v(x) L_h u(x). 
	\end{align}
\end{lemma}
\begin{proof}
Equation \eqref{eq:DheAdj} is obtained using the change of variables $y=x+h e$ (resp.\ $y=x-he$) in the intermediate sums terms, which maps $\cX\to \cX+h e$ (resp.\ likewise, since $\cX -h e=\cX+he$ for any $e\in E \subset\bZ^d$). Equation \eqref{eq:LhAdj} follows from the defining formula  $- L_h u := \sum_{e \in E} \partial_h^e(\lambda^e\partial_h^e u)$.
\end{proof}
As announced, we rewrite $\cL(\hat \phi; \phi)$ using successively \cref{lem:timeIPP,lem:LhAdj}, and recalling \eqref{eqdef:mhm}:
\begin{align}
\nonumber
	- \cL(\hat \phi;\phi) &= 
	(2h)^d \sum_{x \in \cX}
	\Big( \frac{u^-(\tau,x) + u^+(\tau,x)}{2} \hat \phi(0,x) + \frac{\tau}{2} u^-(\tau,x)^2 \hat \rho(0,x) - 
	u_0(x) \hat \phi(0,x) \Big)\\
\nonumber
	&+ \tau (2h)^d \sum_{x \in \cX} \sum_{t \in \mathring \cT} \sum_{\sigma=\pm} \Big(
	\partial_\tau u^\sigma(t,x) \hat \phi(t,x) + \frac{1}{2} u^\sigma(t-\sigma \tau,x)^2 \hat \rho(t,x)\Big)\\
\label{eq:porousIPP0}
	&=(2h)^d \sum_{x \in \cX}
	\Big(\frac{u^-(\tau,x) + u^+(\tau,x)}{2} + \frac{\tau}{2} L_h \big[(u^-)^2\big](\tau,x) - u_0(x)\Big) \hat \phi(0,x)\\
\label{eq:porousIPP1}
	&+ \tau (2h)^d \sum_{x \in \cX} \sum_{t \in \mathring \cT} \sum_{\sigma=\pm}
	\Big(\partial_\tau u^\sigma(t,x) + \frac{1}{2} L_h\big[(u^\sigma)^2\big](t-\sigma \tau,x) \Big) \hat \phi(t,x)
\end{align}
 In the rest of this subsection, we work under \cref{assum:porous}, and thus assume that $\phi\in C^\infty([0,T] \times \bT^d,\bR)$ is a smooth function which solves the continuous problem.
We recognize in \eqref{eq:porousIPP0} and \eqref{eq:porousIPP1} a discrete version of the primal PDE we started from, namely the QPME (\ref{eq:QPME_Burgers}, left). The next step, achieved in \mbox{
\cref{prop:consistency}}\hskip0pt
, is to prove that our discretization choice implies that $u^\sigma$, $\sigma=\pm$, solves this discrete PDE up to a residue quadratically small in $\tau$ and $h$.
For that purpose we distinguish the variables 
\begin{align*}
	m &= \partial_\tau \phi, &\rho &:= 1+L_h \phi,& u^\sigma(t,x) &:= \frac{m(t,x)}{\rho(t+\sigma \tau,x)}, &
	\bm &:= \partial_t \phi, &\brho &:= 1-\diver(\cD\nabla \phi),& \bu &:= \frac{\bm}{\brho},
\end{align*}
and note that $m$ and $u^\sigma$ are defined on the time-restricted domain $[\tau,T-\tau] \times \bT^d$, whereas the other quantities are defined over $[0,T] \times \bT^d$. By \cref{assum:porous}, one has $\brho > 0$ uniformly, and $\bu$ obeys the QPME: $\partial_t \bu = \frac{1}{2} \diver(\cD\nabla \bu^2)$ and $\bu(0,\cdot) = u_0$.

The following two basic lemmas, on the integral representation of the finite difference operator, and on the Taylor expansion of a smooth function which is symmetric w.r.t.\ some variables, should be clear by themselves and are thus stated without proof.

\begin{lemma}
\label{lem:diffAsInt}
	Let $F \in C^k (\bT^d \times [-1,1]^n)$, where $k \geq 1$, and let $e \in \bR^d$. Define $G\in C^{k-1}(\bT^d \times [-1,1]^{n+1})$ 
	as follows: for all $x \in \bT^d$ and all $h_1,\cdots,h_n,h \in [-1,1]$ 
\begin{align*}
	G(x; h_1,\cdots,h_n,h) &:= \frac{1}{2} \int_{-1}^1 e^\top 
	\nabla
	F(x+s h e; h_1,\cdots,h_n) \diff s, 
\end{align*}
where the gradient is w.r.t.\ the spatial variable $x$.
Then $G(x; h_1,\cdots,h_n,h) = \partial_h^e F(x; h_1,\cdots,h_n)$ for all $h\neq 0$.
In addition $g(x) = e^\top \nabla f(x)$, where $f(x) := F(x;0,\cdots)$ and $g(x):=G(x; 0,\cdots)$.
\end{lemma}

\begin{lemma}
\label{lem:TaylorSym}
	Let $F\in C^2(\bT^d \times [-1,1]^2)$. 
	Define $a=1$ (resp.\ $a=2$ if $F(x;\tau,h) = F(x; -\tau,h)$ for all $x \in \bT^d$ and $\tau,h \in [-1,1]$), and likewise $b=1$ (resp.\ $b=2$ if $F(x; \tau, h) = F(x; \tau,-h)$ for all $x,\tau,h$).
	Then $\max_{x \in \bT^d} |F(x; \tau,h) - f(x)| = \cO(|\tau|^a+|h|^b)$, where $f(x) := F(x; 0,0)$. 
\end{lemma}

\begin{proposition} 
\label{prop:consistency}
Under \cref{assum:porous} and with the notations $m,\rho,u^\sigma,\bm,\brho,\bu$ above one has 
\begin{align}
\label{eq:TaylorPorous}
	m &= \bm + \cO(\tau^2),&
	\rho &= \brho + \cO(h^2), &
	u^\sigma &= \bu + \cO(\tau + h^2), &
	u^+ + u^- &= 2\bu + \cO(\tau^2 + h^2),
\end{align} \\[-7ex]
\begin{align}
\label{eq:TaylorPorousH}
	-L_h u^- &= \diver(\cD\nabla \bu) + \cO(\tau+h^2), &
	-L_h [(u^-)^2] &= \diver(\cD\nabla \bu^2) + \cO(\tau+h^2),
\end{align}\\[-7.5ex]
\begin{align*}
	\partial_\tau \sum_{\sigma=\pm} u^\sigma &= 2\partial_t \bu + \cO(\tau^2 + h^2),
	&
	\sum_{\sigma=\pm} L_h [(u^\sigma)^2](t-\sigma\tau,x) &= -2\diver(\cD\nabla \bu^2)(t,x) + \cO(\tau^2+h^2), 
\end{align*}
where the $\cO$ notation is understood in the $L^\infty$ norm over the common domain of definition. 
\end{proposition}

\begin{proof}
It is sufficient in this proof to assume that $\phi \in C^6$, and that the coefficients $\lambda^e \in C^5$ for all $e \in E$, rather than $C^\infty$, which allows a slight weakening of  \mbox{
\cref{assum:porous}}\hskip0pt
We consider an arbitrary extension of $\phi$ into a periodic function of the time variable $t \in \bR/(T+1)\bZ$, in addition to the space variable $x \in \bT^d$, with the same smoothness, so as to fit the assumptions of \cref{lem:diffAsInt,lem:TaylorSym}. We remove here \cref{assum:step} on the timestep and gridscale, and let $\tau$ and $h$ be possibly zero, positive or negative.
The announced estimates can be reformulated as $F(t,x;\tau,h) = f(t,x)+ \cO(|\tau|^a+|h|^b)$, for some exponents $a,b\in \{1,2\}$.
Expressing  successive finite differences as integral operators as in \cref{lem:diffAsInt}, we find that $F(t,x;\tau,h) = \gF(t,x; \tau,\cdots,\tau,h,\cdots,h)$ where $\gF(t,x; \tau_1,\cdots,\tau_m,h_1,\cdots,h_n)$ is at least $C^2$, and that $f(t,x) = \gF(t,x; 0,\cdots,0)$ holds by design. \Cref{lem:TaylorSym} thus applies, and the announced estimates follow from the symmetry properties of $F$ w.r.t.\ the arguments $\tau$ and $h$.

The staggered time finite difference $\partial_\tau$ defined in \eqref{eqdef:DtDhe} is indeed symmetric w.r.t.\ $\tau$, thus $\partial_\tau \phi = \partial_t \phi + \cO(\tau^2)$ which is (\ref{eq:TaylorPorous}, i). The scheme $L_h$ expanded in \eqref{eq:LhExpl} is symmetric w.r.t.\ $h$, hence is second order accurate: recalling the decomposition (\ref{eqdef:Lh}, right) of $\cD$ we obtain
\begin{equation*}
	\diver(\cD \nabla \phi) 
	= \sum_{e \in E} \diver( \lambda^e e e^\top \nabla \phi) 
	= \sum_{e \in E} e^\top \nabla (\lambda^e e^\top \nabla\phi) 
	= -L_h \phi + \cO(h^2)
\end{equation*}
which is (\ref{eq:TaylorPorous}, ii), see also \cite[Proposition C.15]{bonnans2023monotone} for a similar discussion.
In particular, $\rho$ is like $\brho$ positively bounded below for sufficiently small $h$, and thus $u^\sigma$ is well defined. 
The time shift in the denominator of $u^\sigma$ breaks the symmetry w.r.t.\ $\tau$, but not w.r.t.\ $h$, hence the $\cO(\tau+h^2)$ error in (\ref{eq:TaylorPorous}, iii) and likewise \eqref{eq:TaylorPorousH}. The sum $u^+ + u^-$ restores the symmetry w.r.t.\ $\tau$, hence (\ref{eq:TaylorPorous}, iv) and the remaining estimates. 
\end{proof}

Inserting the expansions of \cref{prop:consistency} in \eqref{eq:porousIPP1} we obtain 
\begin{align}
\label{eq:pourousIPP0taylor}
-\cL(\hat \phi; \phi)
	&=(2h)^d \sum_{x \in \cX}
	\Big(\bu(\tau,x) -\frac{\tau}{2} \diver(\cD\nabla \bu^2)(\tau,x)- u_0(x) +\cO(\tau^2+h^2)\Big) \hat \phi(0,x)\\
\nonumber
	&+ 2\tau (2h)^d \sum_{x \in \cX} \sum_{t \in \mathring \cT} 
	\Big(\partial_t \bu(t,x)-\frac{1}{2} \diver(\cD \nabla \bu^2)(t,x) + \cO(\tau^2+h^2)\Big) \hat \phi(t,x).
\end{align}
 Recalling that $\bu$ is a smooth solution to the QPME, we obtain \eqref{eq:linL1} as announced.

\subsection{Lower estimate of the second order term}
\label{subsec:QPorous}

The first part of this subsection establishes the announced estimate \eqref{eq:phiR} of the solution of a backwards in time diffusion equation \eqref{eq:implLin} discretized in an implicit manner.
We then establish the announced lower bound \eqref{eq:EstimLR} on the second order term $\cQ(\hat \phi; \phi)$ in \eqref{eqdef:LQ}. Together, these results fill the remaining gaps in the proof outline, which concludes the proof of \mbox{
\cref{th:cvPorous}}\hskip0pt



The term $\cQ(\hat \phi; \phi)$ defined in \eqref{eqdef:LQ} is a sum of perspective functions, whose first argument reads
\begin{equation}
\label{eq:residuePorous}
	\!r^\sigma(t) := \hat m(t) - \hat \rho(t+\sigma \tau)u^\sigma(t) = 
	\frac{\hat \phi(t+\tau) - \hat\phi(t-\tau)}{2\tau} 
	- u^\sigma(t) L_h\hat\phi(t+\sigma \tau) ,
\end{equation}
for all times $t \in \cT'$ and signs $\sigma=\pm$. 
Here and below, given a bivariate function $u(t,x)$, we denote by $u(t) := u(t,\cdot)$ the univariate function where the time argument has been frozen to some given value $t$.
The term \eqref{eq:residuePorous} can be regarded as the r.h.s.\ of a timestep for the discretization of a time-reversed linear PDE closely related to \eqref{eq:pdePhi}, with the terminal boundary condition $\hat \phi(T)=0$. This timestep is explicit (resp.\ implicit) if $\sigma = +$ (resp.\ $\sigma=-$, see \eqref{eq:implDiff}). 

\paragraph{Proof of the estimate \eqref{eq:phiR} of  $\hat \phi$ in terms of $r^-$.}
We begin with a stability lemma, based on the discrete degenerate ellipticity of the adaptive finite difference scheme $L_h$ for the anisotropic Laplacian $-\diver(\cD \nabla\cdot )$, see \cite{bonnans2023monotone}. Our first ingredient is an elementary result due to Minkowski, dating back to 1900, and whose proof is recalled for completeness.


\begin{lemma}[Minkowski, see {\cite[Note 6.1]{berman1994nonnegative}}]
\label{lem:Minkowski}
If a matrix has non-negative off diagonal entries, and positive row sums, then it is invertible and the inverse has non-negative entries. 
\end{lemma}

\begin{proof}
	By assumption, the matrix of interest can be written $A = D(\Id - Z)$, where $D$ is a diagonal matrix with positive diagonal entries, and $Z$ is a non-negative matrix with zeros on the diagonal and whose maximum row sum $s := \max_i \sum_j Z_{ij}$ satisfies $s <1$.
	Noting that $s = \|Z\|_{l^\infty\to l^\infty}$ is the operator norm of $Z$ in the discrete $l^\infty$ norm, we obtain that the series $A^{-1} = (\Id-Z)^{-1} D^{-1} = \sum_{k \geq 0} Z^k D^{-1}$ is convergent and componentwise non-negative.
\end{proof}

 We say that a linear operator on $\bR^\cX$ is inverse positive, if it is invertible and the inverse matrix has non-negative entries in the canonical basis.

\begin{proposition}
\label{prop:implLap}
	For any $u \in [0,\infty[^\cX$ the operator $\Id + u L_h$ on $\bR^\cX$ is inverse positive. 
	If in addition $\|L_h u\|_\infty < 1$, then $\|(\Id + u L_h)^{-1} f\|_1 \leq (1-\|L_h u\|_\infty)^{-1} \|f\|_1$, for all $f \in \bR^\cX$.
\end{proposition}

\begin{proof}
The matrix of the operator $L_h$ has non-positive off-diagonal entries, and its rows (or columns since it is symmetric) sum to zero, as is clear from the expanded form \eqref{eq:LhExpl}. 
Therefore, and since $u$ is non-negative, the matrix of $\Id + u L_h$ has non-positive off-diagonal entries, and the sum of each row is one, hence this operator is inverse positive by \mbox{
\cref{lem:Minkowski}}\hskip0pt
 , as announced.



Now consider $f \geq 0$, and let $g := (\Id+u L_h)^{-1} f$, which is thus non-negative. Then 
\begin{equation*}
	\frac{\|f\|_1}{(2h)^d} = \sum_{x \in \cX} f(x) = \sum_{x\in \cX} \big[g(x) + u(x) L_h g(x)\big] 
	= \sum_{x \in \cX} g(x)(1+ L_h u(x)) \geq \frac{\|g\|_1}{(2h)^d} (1-\|L_h u\|_\infty),
\end{equation*}
using the self-adjointness of $L_h$, see \cref{lem:LhAdj}. 
Observing that any $f \in \bR^\cX$ can be expressed as the difference $f = f_+ - f_-$ of its positive and negative part, with $\|f\|_1 = \|f_+\|_1 + \|f_-\|_1$, we conclude the proof by linearity of $(\Id+u L_h)^{-1}$. 
\end{proof}

In the following we choose the sign $\sigma=-$ corresponding to the \emph{implicit} scheme in \eqref{eq:residuePorous}, for stability considerations in view of \cref{prop:implLap}. 
In addition to the explicit and implicit schemes, we discuss in \cref{sec:arith} the semi-implicit scheme, which arises when considering a slightly modified energy functional.
Rearranging terms in \eqref{eq:residuePorous} we obtain 
\begin{align}
\label{eq:implLin}
	(1+2 \tau\, u^-(t) L_h) \hat \phi(t-\tau) &= \hat \phi(t+\tau) - 2 \tau\,  r^-(t), &
	\hat \phi(T)&=0,
\end{align}
where the spatial variable $x \in \cX$ is omitted for readability. 
We next assume that $u^-\geq 0$, in such way that $1+2 \tau\, u^-(t) L_h$ is invertible by \mbox{
\cref{prop:implLap}}\hskip0pt
. The solution $\hat \phi$ of \eqref{eq:implLin} can therefore be expressed in terms of the r.h.s., using the discrete Duhamel formula or simply an induction on $N=0,\cdots,N_\tau$ where $N_\tau := T/(2\tau)$ is the number of timesteps:
\begin{equation}
\label{eq:Duhamel}
	\phi(T-2\tau\, N) = -2\tau \sum_{0\leq n <N} 
	\Big[\prod_{n\leq k < N} \big(1+2\tau\, u^-(T-(2k+1)\tau)\, L_h\big)\Big]^{-1} 
	r^-(T-(2n+1) \tau)
\end{equation}
The announced estimate \eqref{eq:phiR} follows from \mbox{
\cref{prop:implLap}}\hskip0pt
, with the constant 
\begin{align*}
	\prod_{t \in \cT'} (1-2\tau \|L_h u^-(t)\|_\infty)^{-1}
	&\leq (1-\tau K)^{-N_\tau}
	\leq C_\cQ := \exp(T K),&
	\text{with } K &:= 2 \max_{t \in \cT'} \|L_h u^-(t)\|_\infty,
\end{align*}
where we assumed that $\tau K \leq 1/2$, hence $(1-\tau K)^{-1} \leq \exp(2 \tau K)$, in addition to $u^- \geq 0$. 

Let us now show that the latter two conditions hold under \cref{assum:porous}, which states  that $\phi \in C^\infty([0,T] \times \bT^d,\bR)$ is a smooth function, defined in terms of the QPME solution $\bu > 0$ itself assumed to be smooth and positive.

By \cref{prop:consistency} one has $u^- = \bu + \cO(\tau+h^2)$, and therefore $u^- \geq 0$ as desired on $[\tau,T-\tau]\times \bT^d$, for any half timestep $0<\tau \leq \tau_*$ and half gridscale $0<h\leq h_*$, where $\tau_*,h_*>0$ depend only on $(\lambda,\phi)$.

In addition $-L_h u^- = \diver(\cD \nabla \bu) + \cO(\tau + h^2)$ by \mbox{
\cref{prop:consistency}}\hskip0pt
, showing that $K$ is bounded independently of $\tau\in ]0,\tau_*]$ and $h\in ]0,h_*]$, hence we may assume that $\tau_* K \leq 1/2$ as desired, up to reducing the maximum timestep $\tau_*>0$.

\paragraph{Proof of  lower bound {\normalfont (\ref{eq:EstimLR} , right)} .}
Since the estimated quantity $\cQ(\hat \phi; \phi)$  is a sum of perspective functions  $\cP$ , we begin with a convexity lemma for this functional.

\begin{lemma}
\label{lem:scalarPersp}
	One has $\|\cP(m,\rho)\|_1 \geq \cP(\|m\|_1,\|\rho\|_1)$, for all $m, \rho \in \bR^\cX$. 
\end{lemma}

\begin{proof}
	Follows from Jensen's inequality applied to the perspective function, which is convex.
\end{proof}

In the case of positive densities, one could also directly use the Cauchy-Schwartz inequality, which yields the equivalent statement: $(\sum_{x\in \cX} |m(x)|)^2 \leq \sum_{x \in \cX} \frac{m(x)^2}{\rho(x)} \sum_{x \in \cX} \rho(x)$.

\begin{lemma}
\label{lem:int1}
	Let $\phi \in \bR^\cX$ and $\rho := 1+L_h \phi$. If $\rho \geq 0$ then $\|\rho\|_1 = 1$.
\end{lemma}

\begin{proof}
We have $\sum_{x \in \cX} L_h \phi(x) = 0$, since $L_h$ is self-adjoint, see \cref{lem:LhAdj}, and since it vanishes on constants, see \eqref{eq:LhExpl}. Thus $\|\rho\|_1 = (2h)^d \sum_{x \in \cX} (1+L_h \phi(x)) = 1$, since $\rho\geq 0$ and $\#(\cX) = \#(\bT_h)^d = (2h)^{-d}$.
\end{proof}
We obtain
\begin{equation}
\label{eq:R2Q}
	\frac{1}{T} R^2 
	\leq 2\tau \sum_{t \in \cT'} \|r(t)\|^2_1 
	= 2 \tau \sum_{t \in \cT'} \cP(\|r(t)\|_1, \|\rho(t-\tau)+\hat \rho(t-\tau)\|_1) \leq 2 \cQ(\hat \phi;\phi),
\end{equation}
using successively (i) Cauchy-Schwartz's inequality and $T=2 \tau N_\tau$, (ii) the fact that $\rho(t-\tau)+\hat \rho(t-\tau) \geq 0$, otherwise $\cQ(\hat \phi; \phi)=\infty$ in which case the announced inequality is clear, and thus $\|\rho(t-\tau)+\hat \rho(t-\tau)\|_1 = \|2+L_h (\phi+\hat \phi)\|_1 = 2$ by \cref{lem:int1}, and (iii) \cref{lem:scalarPersp} together with the non-negativity of the perspective function (hence of the  contributions to $\cQ(\phi; \hat \phi)$ associated with the other sign choice $\sigma=+$). \mbox{
\Cref{th:cvPorous} }\hskip0pt
follows as described in the proof outline \eqref{eq:minPclose}.

\section{Efficient implementation}
\label{sec:implem}
\def\per{\mathrm{per}}
\def\bi{\boldsymbol{i}}
\def\dt{dt}

We present in this section an implementation of the primal-dual proximal algorithm \cite[Algorithm 2]{chambolle2011primaldual} for the minimization of the energies \eqref{eqdef:porousEnergy} and \eqref{eqdef:BurgersEnergy} associated with the BBB formulations of the QPME and Burgers's equation, subject to the relevant constraints. (For problem instances of small size, a damped Newton method can also be used, following \cite{kouskiya2024inviscid}. See the discussion in  \cref{sec:num}.)
Our approach is closely related to the proximal implementation of the Benamou-Brenier formulation of optimal transport \cite{papadakis2014optimal}, with mild differences related to: the change in boundary conditions w.r.t.\ the time variable, the additional linear term in the Burgers viscous equation,
 the use of the harmonic rather than the arithmetic mean of densities, and the anisotropic tensor field $\cD$ in the QPME. 

Our implementation relies for simplicity on a global Fourier transform in time and space, as opposed to space only in \cite{papadakis2014optimal}.
For that purpose, we introduce the extended time grids
\begin{align*}
	\cT_{\per} &:= 
	\{-T+2\tau,\cdots,T-2\tau,T\}, &
	\cT'_\per &:= 
	\{-T+\tau, T+ 3 \tau, \cdots, T-\tau\},
\end{align*}
where the half timestep $\tau>0$ is fixed throughout this section. 
The mapping $\rho$ defined on $\cT_\per$ is implicitly extended by $2T$-periodicity, in such way that $\partial_\tau \rho$ is well defined on $\cT'_\per$ in (\ref{eq:cstrPer}, iii), (\ref{eq:mene}, iii) and (\ref{eq:BurgersCstr}, left) below.
In this section only, we denote by $\sign(t)\in \{-1,1\}$ the sign of $t \in \cT'_\per$.
We fix likewise the half gridscale $h>0$ and denote by $\cX := \bT_h^d$ the discrete domain (not to be confused with the characteristic function $\chi$ below). 
Focusing as a first step on the QPME, we consider the minimization of 
\begin{equation}
\label{eqdef:EperQPME}
\cE_\per(m,\rho) := 
	\sum_{x \in \cX} \sum_{t \in \cT'_\per} \Big[\cP\big(m(t,x), \rho(t+\tau,x)\big) - \sign(t) m(t,x) u_0(x) \Big]+ \sum_{x \in \cX} \chi(\rho(T,x)-1).
	\vspace{-1mm}
\end{equation}
The characteristic function is defined as $\chi(s) =0$ if $s=0$, $\chi(s) = \infty$ otherwise, and it implements the terminal boundary condition $\rho(T)=1$. 
The objective function should be minimized among all $m \in \bR^{\cT'_\per\times \cX}$ and $\rho\in \bR^{\cT_\per \times \cX}$ subject to the linear constraints of oddness of $m$, evenness of $\rho$, and continuity equation:
\begin{align}
\label{eq:cstrPer}
	m(t) &= -m(-t),\ t \in \cT'_\per, &
	\rho(t) &= \rho(-t),\ t \in \cT_\per, &
	\partial_\tau \rho(t) &= L_h m(t),\ t \in \cT'_\per.
\end{align}
Let us emphasize that the time periodization, the parity constraints, and the time shift $\rho(t+\tau,x)$ of the density in $\cE_\per$, together, are equivalent to the sum over the two signs $\sigma = \pm$ in the original energy $\cE^\Porous_{\tau h}$, see \eqref{eqdef:porousEnergy}.
Thus minimizing $\cE_\per$ subject to \eqref{eq:cstrPer} is equivalent to minimizing $\cE^\Porous_{\tau h}$ on the domain of \cref{lem:domPorous}.

The chosen minimization algorithm is the iterative numerical method \cite[Algorithm 2]{chambolle2011primaldual}, which requires computing the \emph{proximal operator} of the energy $\cE_\per$, and the orthogonal projection onto the linear space defined by \eqref{eq:cstrPer}. 
As described below, these two steps are made simple by the fully separable structure of \eqref{eqdef:EperQPME}, and the Fourier-compatible structure of \eqref{eq:cstrPer}.
Recall that the proximal operator of a proper convex lower-semi-continuous $f:\bR^n \to ]-\infty,\infty]$ is defined for any $\lambda>0$ as
\begin{equation*}
	\prox_{\lambda f} (x) := \argmin_{y \in \bR^n} \frac{1}{2} \|x-y\|^2 + \lambda f(y).
\end{equation*}
The proximal operator of the perspective function is obtained by solving a univariate third degree polynomial equation, see \cite[Example 3.8]{combettes2018perspective}. 

\begin{proposition}
\label{prop:proxSep}
The proximal value $(m',\rho') = \prox_{\lambda \cE_\per}(m,\rho)$ is expressed independently for each point and time in terms of the proximal operator of the perspective function: 
 \begin{equation*}
\big(m'(t,x),\rho'(t+\tau,x)\big) = \prox_{\lambda \cP}\big(m(t,x)+\lambda\sign(t) u_0(x),\, \rho(t+\tau,x)\big)
\end{equation*}
 for all $x \in \cX$ and $t \in \cT_\per \sm \{T-\tau\}$, except for the special case $m'(T-\tau) = \big(m(T-\tau) + \lambda u_0\big)/(1+\lambda)$ and $\rho'(T)=1$.
\end{proposition}

\begin{proof}
The energy $\cE_\per$ is a sum of contributions depending on disjoint sets of variables, namely the pairs $(m(t,x), \rho(t+\tau,x))$ for each $x \in \cX$ and $t \in \cT'_\per$. Thanks to this separable structure, the corresponding proximal values can be computed independently. 
The additional linear term $u_0(x) m(t,x)$ is equivalent to a translation of the projected point.
The case $t=T-\tau$ is specific due to the characteristic function enforcing the constraint $\rho(T,x)=1$, in such way that the objective is quadratic w.r.t.\ $m(t,x)$. 
\end{proof}


We compute the orthogonal projection onto the linear space defined by the constraints \eqref{eq:cstrPer} in two steps: we first project onto the subspace $E$ defined by the parity constraints (\ref{eq:cstrPer}, i and ii) using the explicit expressions $m_\odd(t) = (m(t)-m(-t))/2$ and $\rho_\even(t) := (\rho(t)+\rho(-t))/2$, and then project onto the subspace $F$ defined by the continuity equation (\ref{eq:cstrPer}, iii). 
Indeed if the pair  $(m,\rho)$ obeys the continuity equation, then so does $(-m(-\cdot), \rho(-\cdot))$, hence also $(m_\odd,\rho_\even) \in E$ and $(m_\even,\rho_\odd) \in E^\perp$, where $m_\even := m-m_\odd$ and $\rho_\odd := \rho-\rho_\even$, in such way that the next lemma applies.

\begin{lemma}
\label{lem:succProj}
Let $E,F$ be closed vector subspaces of a Hilbert space $H$, such that $F = (F\cap E)\oplus(F\cap E^\perp)$. Then the orthogonal projections on 
$E$, $F$ and $E \cap F$,
obey $P_{E \cap F} = P_F P_E$.
\end{lemma}

\begin{proof}
If $x\in E^\perp$, then $P_F P_E(x) = P_F(0)=0$ and $P_{E \cap F}(x) = 0$ since $E \cap F \subset E$. Let us now consider $x \in E$, and $y:=P_{E \cap F}(x)$. Consider $z \in F$, which by assumption can be expressed as $z = z_1+z_2$ where $z_1 \in E \cap F$ and $z_2 \in E^\perp \cap F$. 
	Then $\<x-y, z_1\> = 0$ by definition of the orthogonal projection on $E \cap F$, and $\<x-y,z_2\>=0$ since $x,y\in E$ and $z_2 \in E^\perp$. Thus $\<x-y,z\>=0$, $\forall z \in F$, and therefore $y = P_F(x)$ as announced, which concludes the proof. 
\end{proof}

The projection onto the discretized continuity equation constraint (\ref{eq:cstrPer}, iii) is computed using a global space-time Fourier transform. More precisely, denote for all $\zeta \in \cT_\per$ and $\xi \in \cX$
\begin{align}
\label{eq:fft_QPME}
	M(\zeta,\xi) &:= \! \sum_
	{\substack{t \in \cT_\per\\ x \in \cX}}
	m(t+\tau,x)
	e^{-\bi\pi (\frac{\zeta t}{\tau} +\frac{\<\xi,x\>}{h})},
	&
	P(\zeta,\xi) &:= \sum_
	{\substack{t \in \cT_\per\\ x \in \cX}}
	\rho(t,x) 
	e^{-\bi\pi (\frac{\zeta t}{\tau} +\frac{\<\xi,x\>}{h})},
\end{align}
 assuming for simplicity that $2T=1$. A half time step $m(t+\tau,x)$ is introduced in (\ref{eq:fft_QPME}, left) to account for the fact that the momentum $m$ is defined on a staggered grid in time; similarly, the discretized Burgers' equation below involves a density defined on a staggered grid in space, whose Fourier transform is expressed in terms of $\rho(t,x+h)$.

In the case of the standard isotropic Laplacian $-L_h u(x) := \Delta_h u(x) = (2h)^{-2} \sum_{1 \leq i \leq d} [u(x+2h e_i)-2 u(x) + u(x-2h e_i)]$, where $(e_1,\cdots,e_d)$ denotes the canonical basis, the discretized continuity equation (\ref{eq:cstrPer}, iii) reads as follows in Fourier coordinates
\begin{equation*}
	P(\zeta,\xi) \frac{\exp(2\pi \bi \zeta) - 1}{2\tau}= M(\zeta,\xi) \sum_{1 \leq i \leq d}\frac{\sin^2(\pi \xi_i)}{h^2} 
\end{equation*}
The projection can be computed independently for each pair of Fourier coordinates $(\zeta,\xi)$
using the following basic lemma in dimension $2$. This elementary result, stated without proof, also addresses all projections required in the numerical implementation of the variants below.
\begin{lemma}
\label{lem:projCplx}
	Let $x\in \bC^d$ and $v\in \bC^d\sm \{0\}$. The orthogonal projection of $x$ onto the line $\bC v$ (resp.\ orthogonal space $(\bC v)^\perp$) is obtained as $(v^*x) x/\|v\|^2$ (resp.\ $x-(v^*x) x/\|v\|^2$), with $v^* := \overline v^\top$.
\end{lemma}

On the other hand, if the coefficients $\lambda^e$ vary over the domain, then we introduce additional variables $m_e,n_e \in \bR^{\cT'_\per\times (\cX+h e)}$ for each offset $e\in E$ (unfortunately increasing memory usage as a byproduct). 
The constraint (\ref{eq:cstrPer}, iii) is removed, and replaced with 
\begin{align}
\label{eq:mene}
	m_e &= \partial_h^e m, &
	n_e &= \lambda^e m_e, &
	\partial_\tau \rho = \sum_{e \in E} \partial_h^e n_e,
\end{align}
as well as the oddness constraints $m_e(-t)=-m_e(t)$ and $n_e(-t)=-n_e(t)$, $t \in \cT'_\per$. 
The constraints (\ref{eq:mene}, i and iii) are added to \eqref{eq:cstrPer} and treated in space-time Fourier coordinates using \cref{lem:projCplx} in dimension $\#(E)+1$ (projecting on a line for (\ref{eq:mene}, i), and on a hyperplane for (\ref{eq:mene}, iii)), after a preliminary projection onto the oddness and evenness constraints. On the other hand, the constraint (\ref{eq:mene}, ii), for all $e \in E$, is added to the objective functional $\cE_\per(m,\rho)$ in the form of the characteristic function $\chi(n_e-\lambda^e m_e)$, and treated using \cref{lem:projCplx} in dimension $2$. 

\paragraph{Burgers' equation.} The energy \eqref{eqdef:BurgersEnergy} is reformulated as follows: \begin{equation*}
	\sum_{x \in \cX'} \sum_{\sigma=\pm} \sum_{t \in \cT'_\per} \Big[
	\cP\big(
	\sign(t)m_\sigma(t,x)+a_\sigma(t+\tau,x),
	\rho(t+\tau,x+\sigma h)\big)
	-\sign(t)m_\sigma(t,x) u_0(x)\Big] + \chi(\rho(T,x)-1),
\end{equation*}
up to the multiplicative constant $\tau h$.
We duplicated in this expression the momentum variable: $m_+ = m_-$, and we introduced an auxiliary variable which is also duplicated: $a_+=a_-$, see (\ref{eq:BurgersCstr}, right) below. As a result, the energy has a separable structure: it is a sum of contributions depending on the disjoint $5$-tuples of variables $(\rho(t+\tau,x+h), m_+(t,x), a_+(t+\tau,x), m_-(t,x+2h), a_-(t+\tau,x+2h))$ for all $x \in \cX'$ and $t \in \cT'_\per$. 
The corresponding proximal values are thus computable independently, using the explicit expression of the proximal operator of the two-dimensional perspective function $(m,\rho) \in \bR^2\times \bR \mapsto \cP_2(|m|,\rho)$ \cite[Example 3.8]{combettes2018perspective} (indeed $\cP(m_0,\rho) + \cP(m_1,\rho) = \cP_2( (m_0,m_1),\rho)$ by construction), and the following lemma, which is stated without proof since it immediately follows from the co-variance of the proximal operator under the isometric change of coordinates
$(m,\rho,a) \mapsto (\frac{m+a}{\sqrt 2}, \rho, \frac{m-a}{\sqrt 2})$.

\begin{lemma}
\label{lem:proxSum}
	Let $f : \bR^n\times \bR \to ]-\infty,\infty]$ be proper, convex, lower-semi-continuous, and positively $2$-homogeneous w.r.t.\ the first variable, and let $F(m,\rho,a) := f(m+a,\rho)$. 
	Then 
	\begin{align*}
		\prox_{\tau F}(m,\rho,a) &= \big(
		\frac{u + v}{\sqrt 2}, \tilde \rho, \frac{u-v}{\sqrt 2}\big),&
	\text{where } (u,\tilde \rho) &:= \prox_{2\tau f} (\frac{m+a}{\sqrt 2},\rho), &
	\text{and } v &:= \frac{m-a}{\sqrt 2}.
	\end{align*}	
\end{lemma}
The optimization variables obey the discretized continuity equation, and defining constraint,
\begin{align}
\label{eq:BurgersCstr}
	\partial_\tau \rho + \partial_h m_\sigma &= 0, &
	a_\sigma &= -\nu \partial_h \rho,
\end{align}
for any sign $\sigma=\pm$ (recall that we use duplicated variables $m_+=m_-$ and $a_+=a_-$),
as well as the parity constraints $m_\sigma(-t)=-m_\sigma(t)$, $t\in \cT'_\per$, and $\rho(-t)=\rho(t)$ and $a_\sigma(-t)=a_\sigma(t)$, $t\in \cT_\per$.
For the implementation we successively project thanks to \cref{lem:succProj} on the duplication constraints, then the parity constraints, and finally the PDE constraints \eqref{eq:BurgersCstr}  using space-time Fourier coordinates and \cref{lem:projCplx}. 

Finally, let us mention that the auxiliary variables $a_\sigma \in \bR^{\cT_\per \times \cX}$, $\sigma = \pm$, and the rotation trick \cref{lem:proxSum}, can be discarded for simplicity and efficiency in the non-viscous case $\nu=0$.


\section{Numerical results}
\label{sec:num}

We present in this subsection numerical results for the QPME, one- and two-dimensional, isotropic or anisotropic,
and for 
Burgers' viscous equation, 
obtained with the proposed numerical methods, see \cref{sec:implem}.
Let us immediately acknowledge that, unfortunately and consistently with \cite[Theorem 5.2.2]{brenier2020examples} as discussed below \cref{assum:Burgers}, the proposed method appears unsuited for Burgers' inviscid equation after shock formation, unless we are only interested in the final value at $t=T$ where the entropy solution is still recovered in the continuous setting, while it is not as $t<T$.
An extensive literature has been devoted to the numerical solution of the QPME and Burgers' equation, see \cite{bonkile2018systematic,li2018anisotropic} and references therein; we not compare our results with or even review the state of the art methods for these PDEs in this paper, which should instead be regarded as an early investigation of the radically new BBB numerical approach.
The numerical implementation is available on demand and relies on the Python\textsuperscript{\textregistered} language and the Taichi\textsuperscript{\textregistered} library for just-in-time compilation of CPU and GPU code.

The computational cost and the numerical accuracy of the numerical solutions are dictated by three parameters: the half timestep $\tau>0$, the half gridscale $h>0$, as well as the number of iterations $N_{\prox}$ of the primal dual proximal method.
Indeed, proximal algorithms applied to constrained and/or non-smooth problems are subject to the rather slow $\cO(1/N_{\prox})$ convergence rate, see the discussion of numerical cost below; in practice we often choose $N_{\prox}$ between $2\,000$ and $12\, 000$. 
For comparison purposes, a damped Newton method (whose step is adjusted using a line search) was also implemented, following \cite{kouskiya2024inviscid}. It minimizes the same energy functional $\cE_{\tau h}^\Porous$ or $\cE_{\tau h}^\Burgers$ but expressed in terms of the dual variable $\phi$ rather than $(m,\rho)$, and augmented with a logarithmic barrier penalty enforcing the positivity of the problem density $1+L_h \phi$ or $1-\partial_h \phi$, see \cref{lem:domPorous,lem:domBurgers}.
When applicable, this second approach in practice minimizes the energy functional to machine precision in $6$ to $270$ iterations depending on the experiment; however, it does not scale to large problem instances
since it relies on the inversion of ill-conditioned linear systems with $N_\tau N_h^d$ unknowns, which quickly becomes intractable as $N_\tau$ and $N_h$ increase.

We use periodic boundary conditions for simplicity, and consistently with the theoretical analysis.
The extension to Neumann boundary conditions for the QPME is straightforward, by doubling the physical domain size and extending the solution by reflection (in practice, this amounts to replacing the FFT with a cosine transform in \cref{sec:implem}). 
Dirichlet-type boundary conditions for Burgers' equation are discussed in \cite{kouskiya2024inviscid}.

For reference, let us recall the Barenblatt \cite{vazquez2007porous} explicit solution $u^\Porous$ to the QPME, which is non-smooth and compactly supported, and a classical explicit solution $u^\Burgers$ to Burgers viscous equation obtained via the Hopf-Cole transform 
\cite{whitham2011linear}:
\begin{align}
\label{eqdef:Barenblatt}
	u^\Porous(t,x) &:= \frac{2}{t^\alpha} \max \big\{0, \gamma-\frac{\beta}{4} \frac{ \|x\|^2}{t^{2\beta}}\big\}, &
	u^\Burgers(t,x) &:= \sqrt{\frac{\nu}{\pi t}}\, \frac{\delta \exp(\frac{-x^2}{4 \nu t})}{1+\frac{\delta} 2 \erfc(\frac{x} {\sqrt{4\nu t}})}, 
\end{align}
where 
$\alpha:=d/(d+2)$, $\beta:=1/(d+2)$, and $\gamma,\delta>0$ are arbitrary positive constants. 
We denoted by $\erfc(x) := \frac{2}{\sqrt \pi} \int_x^\infty \exp(-s^2) \diff s$ the complementary error function. In the process of running these experiments, we conjectured, and then proved, a closed form expression for the dual potential associated with the Barenblatt solution, in the BBB formulation of the QPME, see \cref{sec:Barenblatt}.

\begin{figure}
	\def\w{5.2cm}
 \includegraphics[width=\w]{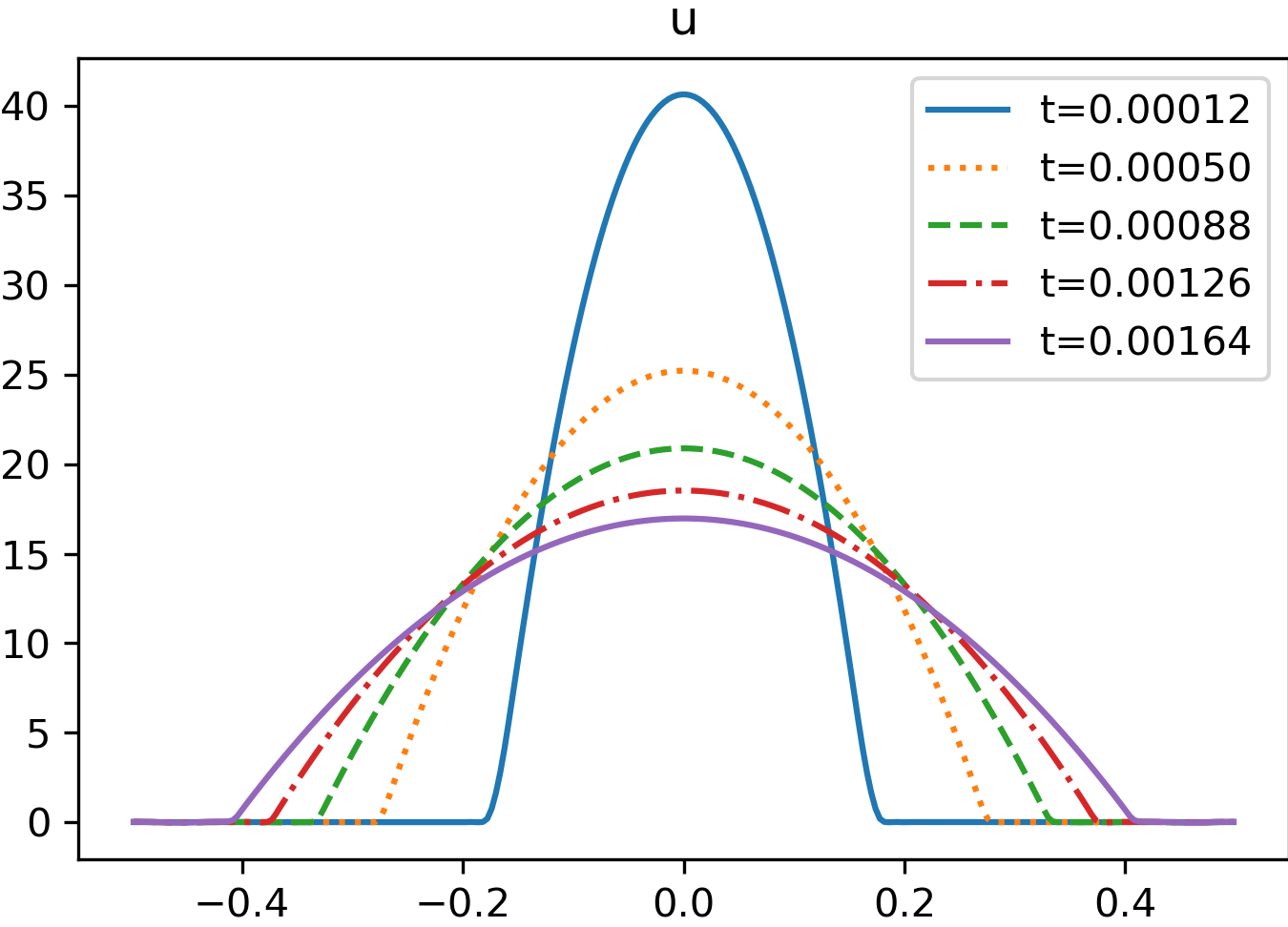}	
\includegraphics[width=\w]{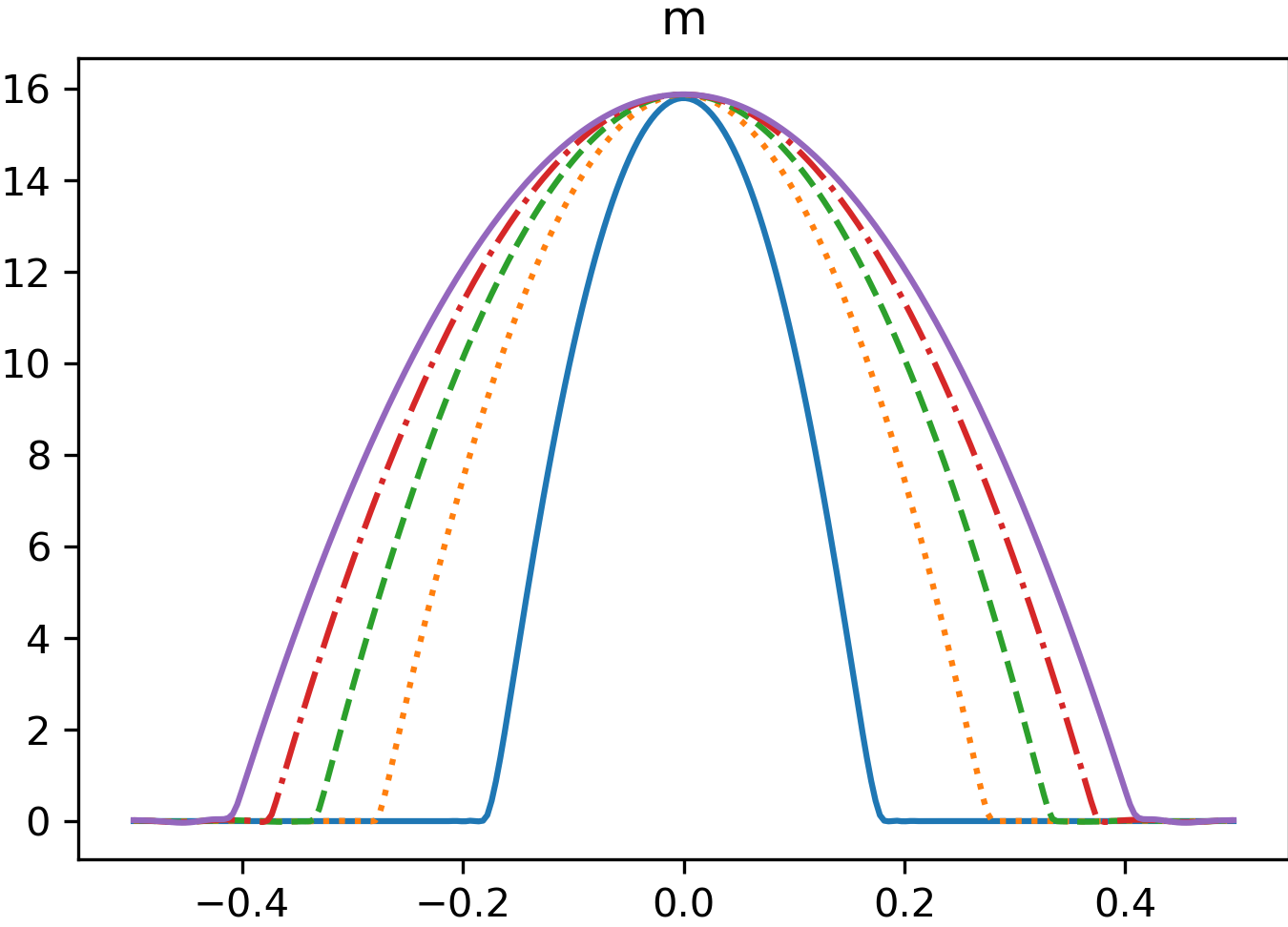}	
\includegraphics[width=\w]{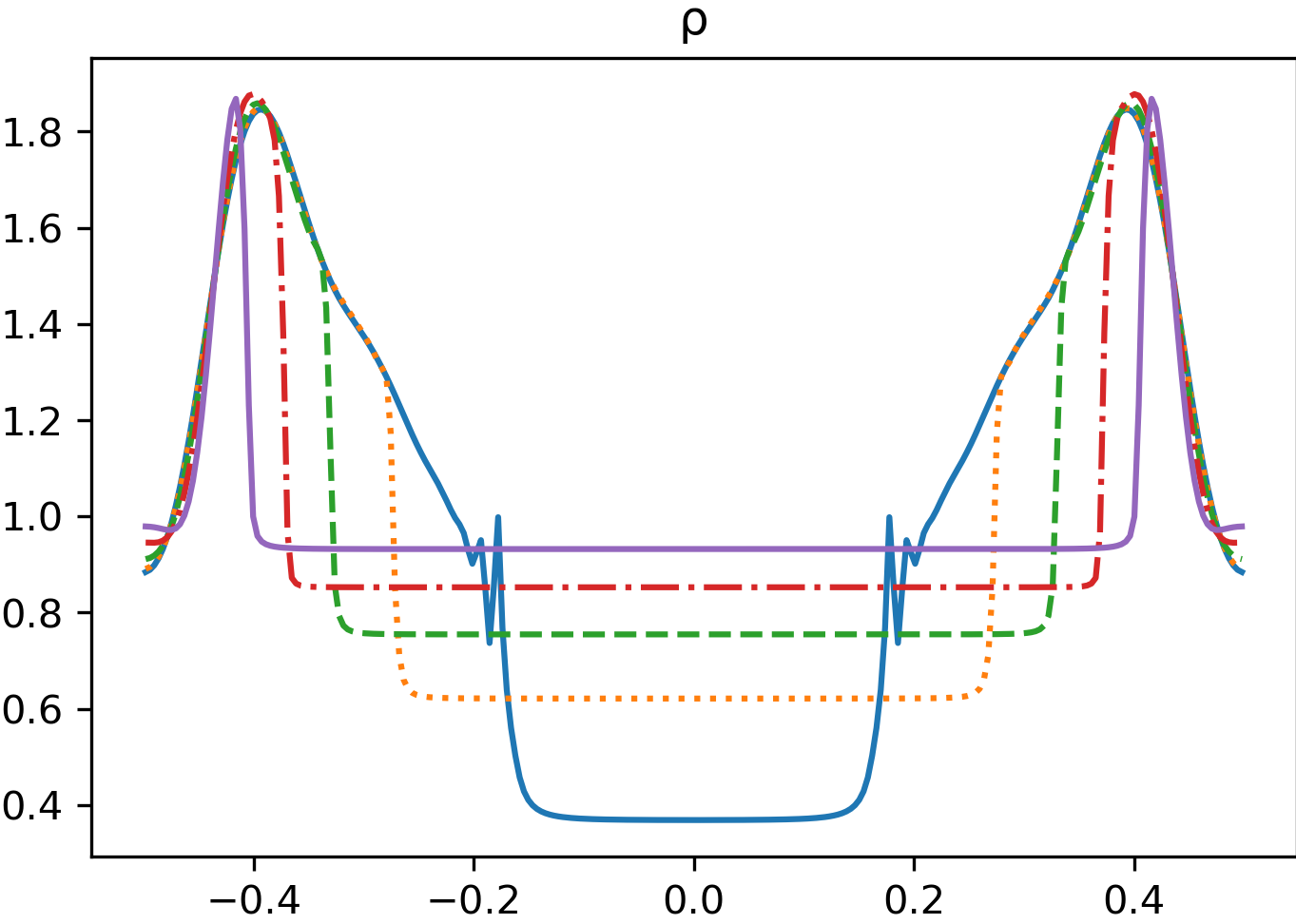}	
 \includegraphics[width=\w]{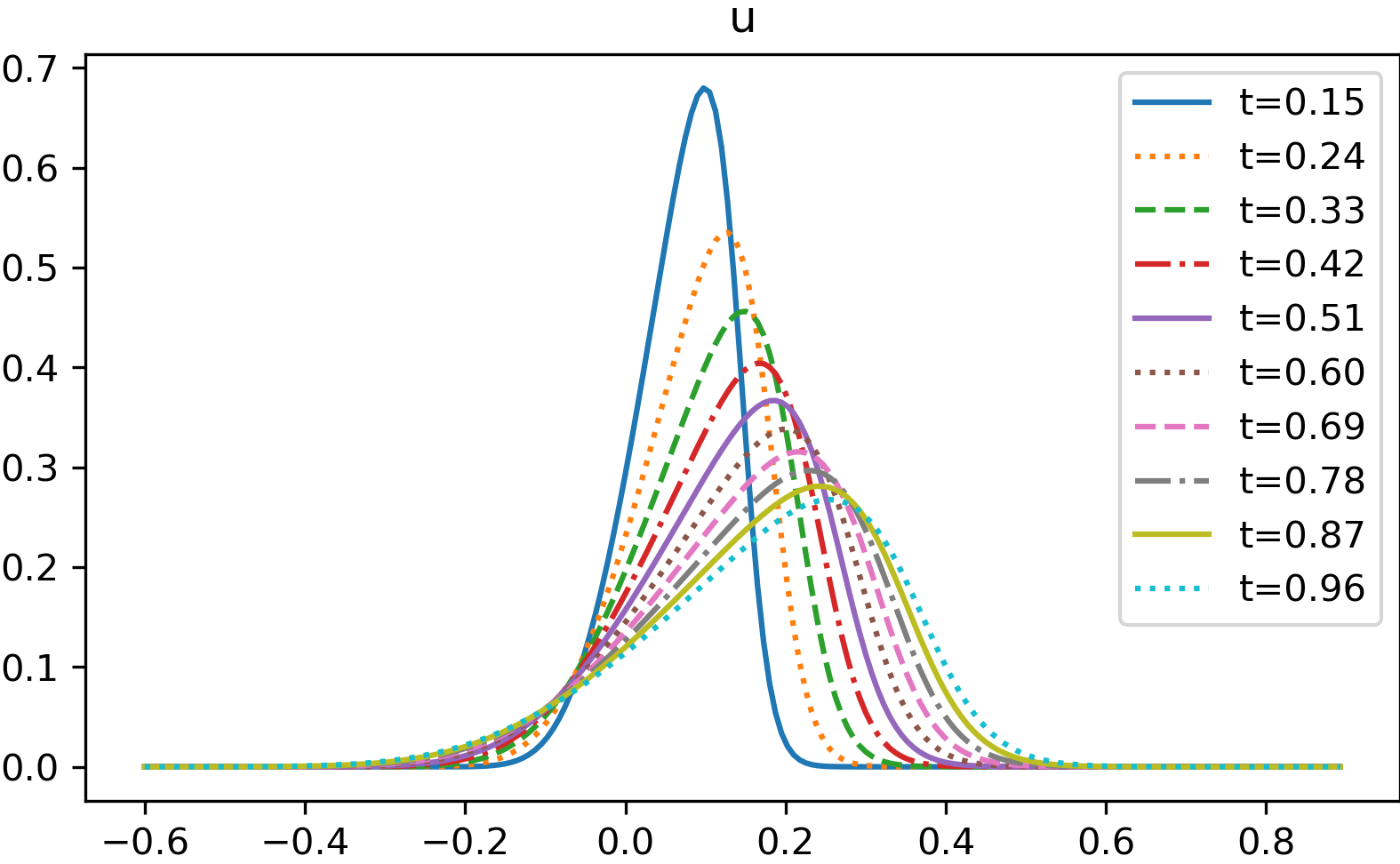}
\includegraphics[width=\w]{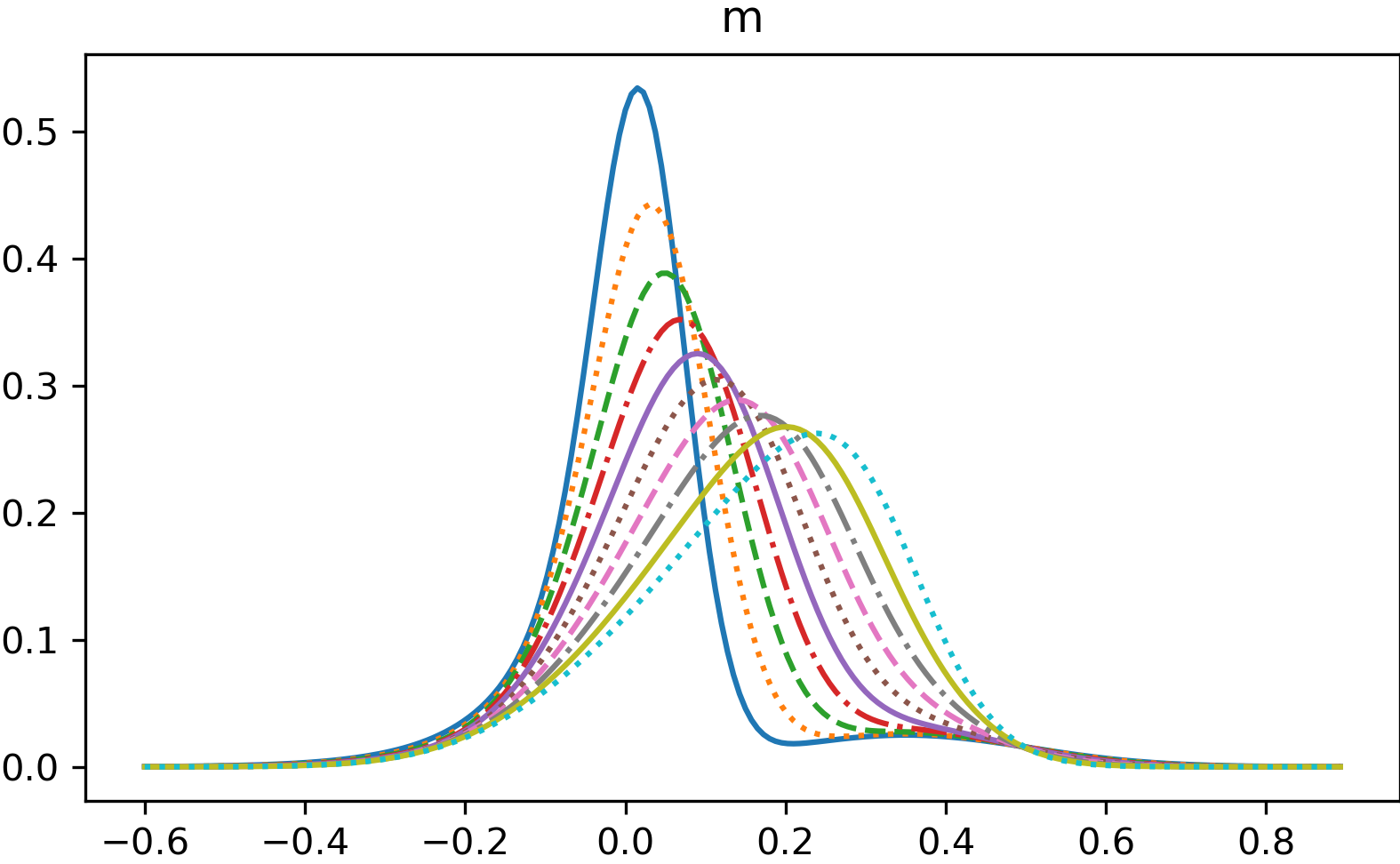}
\includegraphics[width=\w]{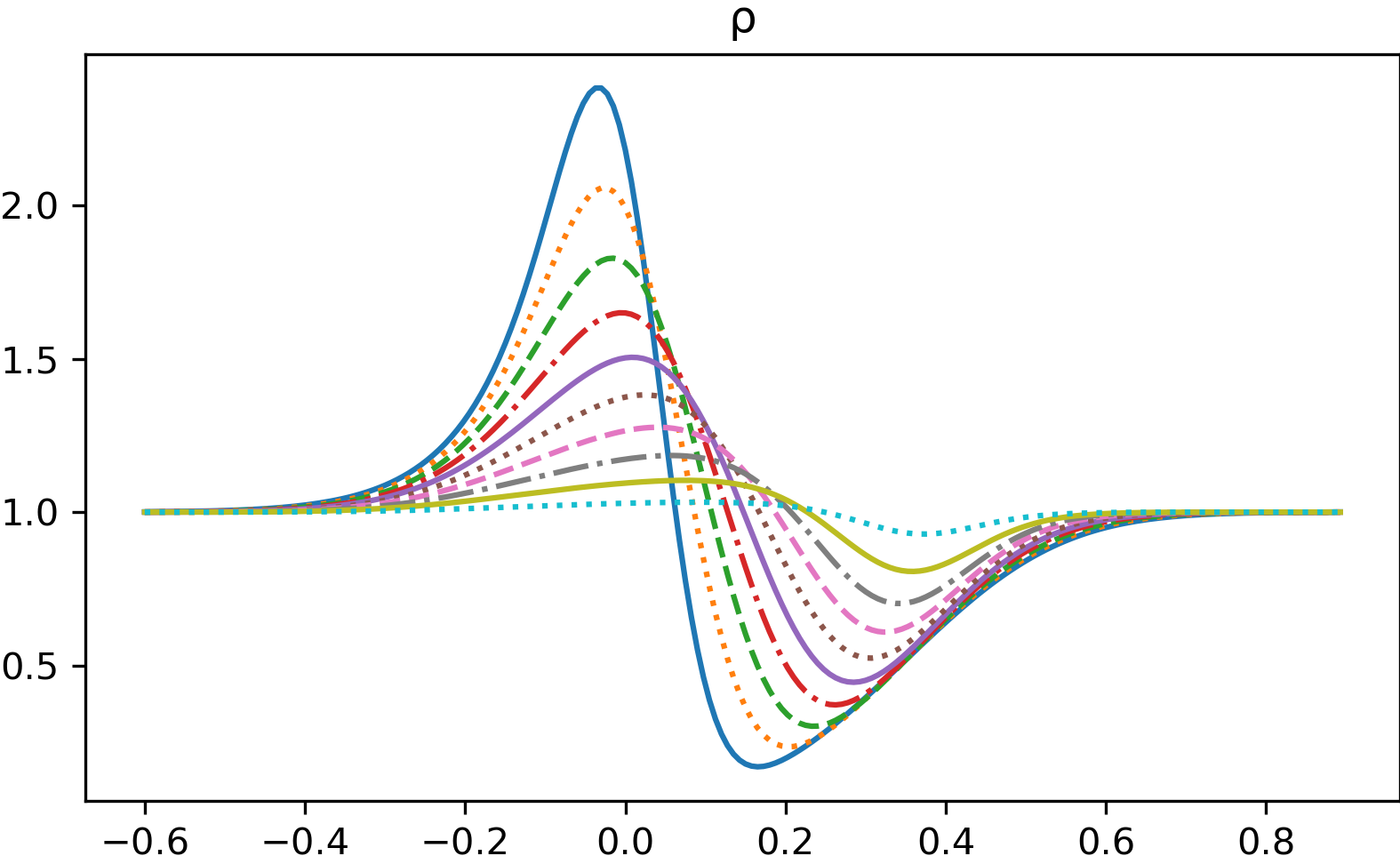}
\caption{
Top left : Barenblatt profile (\ref{eqdef:Barenblatt}, left), with parameter $\gamma=1$, on the time interval $[10^{-4},2\times 10^{-3}]$. Bottom left : explicit solution to Burgers' equation (\ref{eqdef:Barenblatt}, right), with $\nu=10^{-2}$ and $\delta = \exp(\Re)-1$ with $\Re=5$ the Reynolds number. 
Top and bottom, center and right: numerically computed auxiliary variables $m$ and $\rho$, for the corresponding BBB formulation (\ref{eq:pdeMRho}).  
}
\label{fig:exactSol}
\end{figure}


\paragraph{Discussion of numerical cost.}
The proposed BBB numerical method for the QPME has computational cost $\tilde \cO(N_{\prox} N_\tau N_h^d)$ where the ``tilde'' ignores the logarithmic factors associated with the fast Fourier transform. Memory usage is $\cO(N_\tau N_h^d)$, and typical accuracy $\cO(N_{\prox}^{-1} + N_\tau^{-2} + N_h^{-2})$.
For comparison, the simple explicit scheme (\ref{eq:arithScheme},  with $\theta=1$) has computational cost $\cO(N_\tau N_h^d)$, memory usage $\cO(N_h^d)$, and typical accuracy $\cO(N_\tau^{-1} + N_h^{-2})$, but is subject to the CFL condition $N_\tau \gtrsim N_h^2$. Let us acknowledge that the slow convergence of the primal-dual solver, and the higher memory usage, make the BBB approach the costlier one usually, except in specific edge cases such as the use of very large timesteps discussed in the second paragraph of \cref{subsec:numQPME}.
Some optimizations of the BBB numerical implementation are proposed in \cite{kouskiya2024inviscid} (for Burgers' equation), such as concatenating the numerical solutions obtained over small time sub-intervals, and using a well chosen initial guess and ``base state'', see \cref{remark:finalT}.

\subsection{Anisotropic Quadratic Porous medium equation}
\label{subsec:numQPME}

\begin{figure}
\def\h{3.7cm}
  \includegraphics[height=\h]{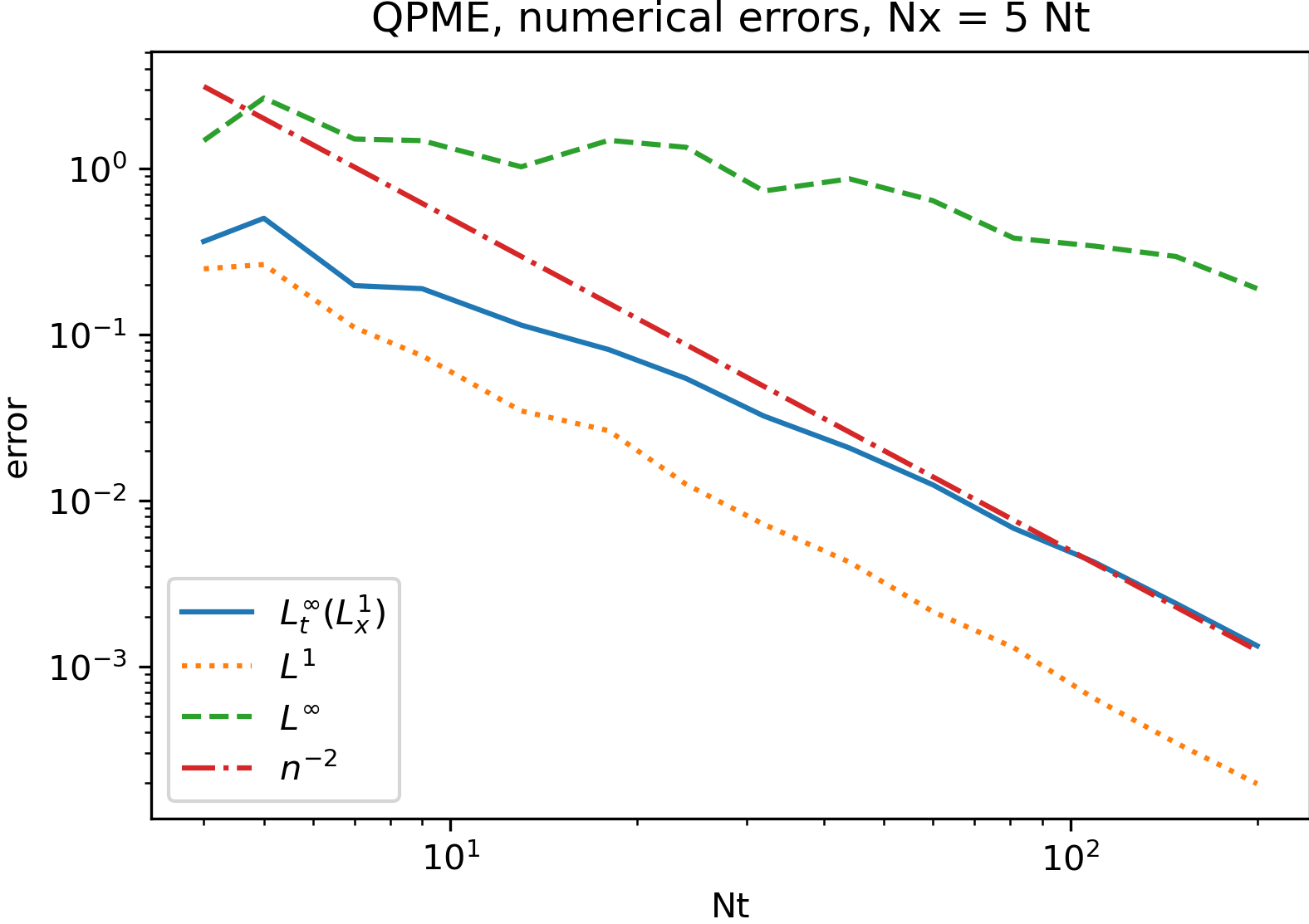}
	\includegraphics[height=\h]{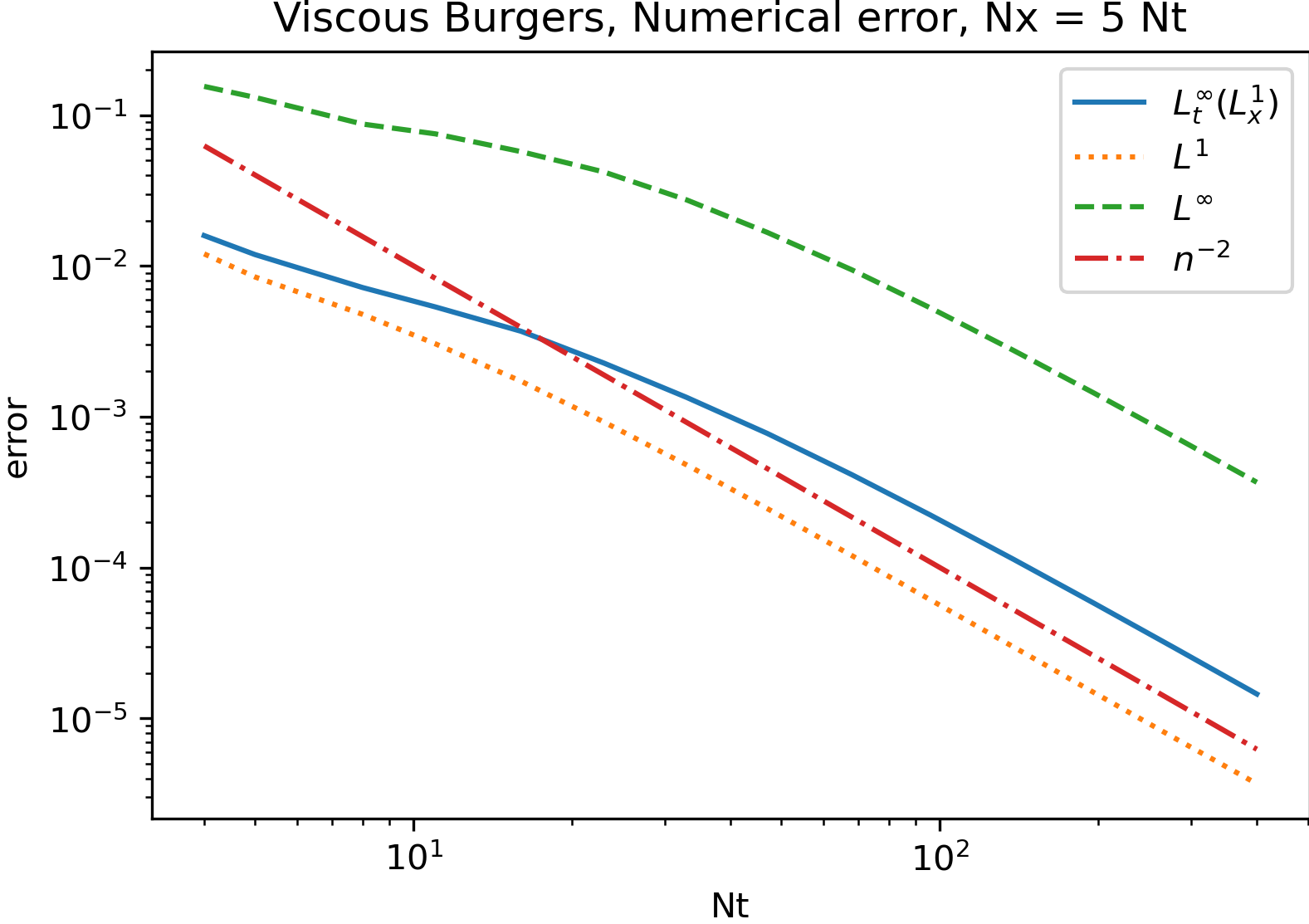}
	\includegraphics[height=\h]{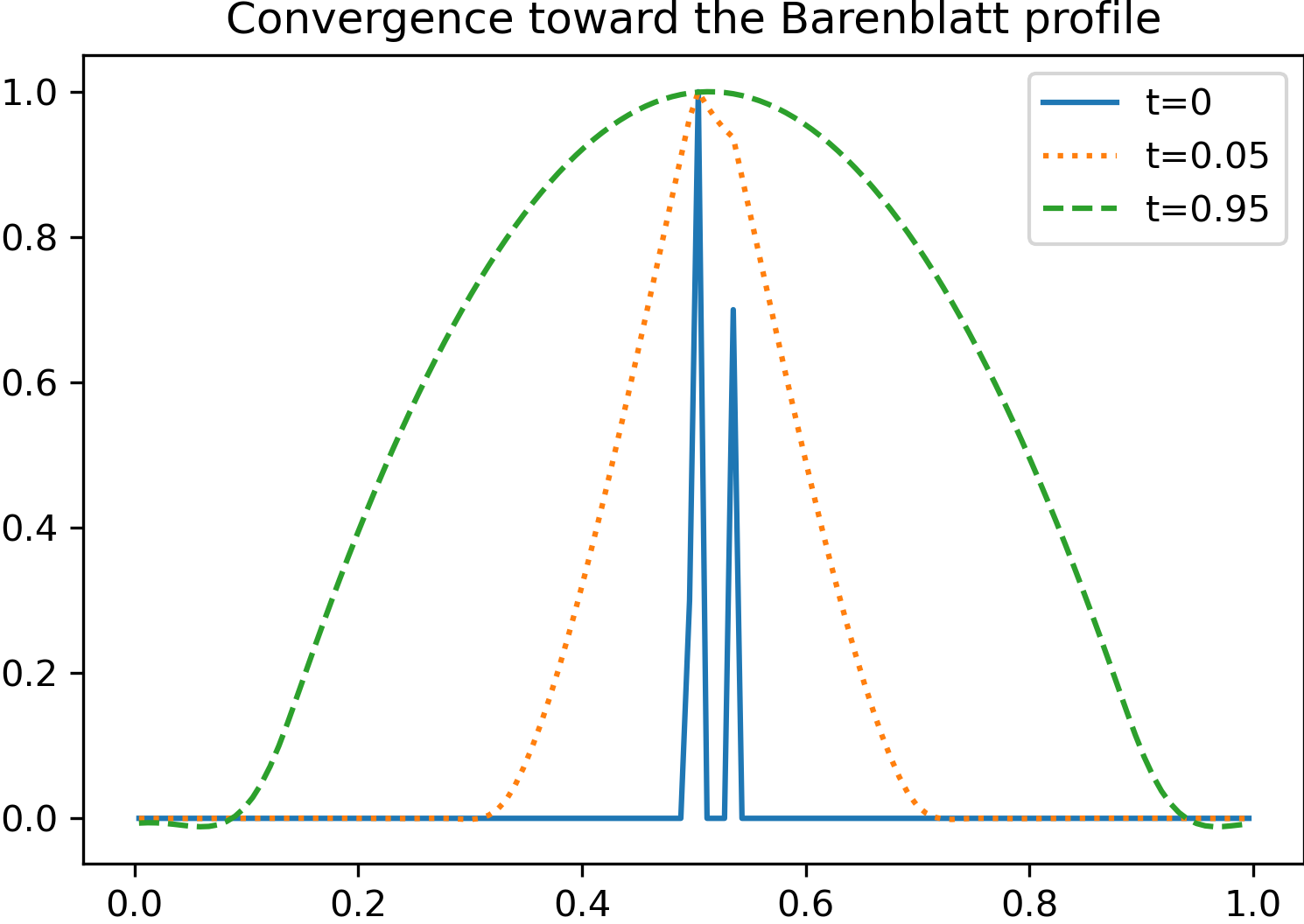}
	\caption{
	Left: log-log plot of numerical error, between the Barenblatt profile with the parameters of \cref{fig:exactSol}, and the numerical solution of the discretized QPME, using $4 \leq N_\tau \leq 200$ timesteps and $N_h = 5 N_\tau$ discretization points; 
	note the second order convergence rate in the $L^1$ norm.
	Center: likewise for Burgers viscous equation, using $4 \leq N_\tau \leq 200$ timesteps and $N_h = 5 N_\tau$ discretization points; note the second order convergence in the $L^\infty$ norm.
	Right: numerical solution of the QPME with an irregular initial condition,  obtained with $N_h = 128$ discretization points, and very few $N_\tau=10$ timesteps, see the discussion of large timesteps in \cref{subsec:numQPME}. Note that a Barenblatt-like profile is obtained at the final time, as can be expected since it is an asymptotic attractor \mbox{
\cite{vazquez2007porous}}\hskip0pt
	}
	\label{fig:cvPorous}
\end{figure}

\paragraph{Accuracy validation.} 
In order to assess the second order accuracy of the proposed numerical method, we reproduce the known Barenblatt solution (\ref{eqdef:Barenblatt}, left), with parameter $\gamma=1$ and from the time $T_0=10^{-4}$ to $T=10^{-3}$, whose spatial support is contained in $[-1/2,1/2]$.
We numerically solve a sequence of small problem instances, where the number $N_\tau$ of timesteps ranges from $4$ to $200$, and the number of spatial discretization points is chosen as $N_h = 5 N_\tau$; the energy functional of interest \eqref{eqdef:porousEnergy} is minimized to machine precision using a damped Newton method ($60$ to $270$ iterations). 
In contrast with our convergence result \cref{th:cvPorous}, the problem solution is \emph{non-smooth}, \emph{non-positive} since it has compact support, and we evaluate the numerical error on the \emph{primal variable} $u$ rather than the dual potential $\phi$. 
We nevertheless observe a second order convergence rate in the $L^1$ norm, which is consistent with the scheme order; 
convergence in the $L^\infty$ norm is slower, as expected in view of the non-smoothness of the solution, see \cref{fig:cvPorous}.
The exact density $\rho(t,x)$ features two discontinuities w.r.t.\ $x \in [0,\infty[$ for each $t \in [T_0,T[$: it is constant on $[0,R(t)[$, then affine on $]R(t),R_T[$, then again constant on $]R_T,\infty[$, with jumps at the radii $R(t)$ and $R_T$, whose closed form expression is obtained in \cref{sec:Barenblatt}. Convergence for the variable $\rho$ is slow as can be expected, and we see some spurious numerical oscillations near the first discontinuity, see \cref{fig:exactSol}, but this does not affect the ratio $u = m/\rho$ whose support ends at $R(t)$. 

\paragraph{Large time steps.}
The solution of the QPME, with any compactly supported initial data $u_0\in C^0_c(\bR)$,  converges (when suitably rescaled) as $t \to \infty$ to the Barenblatt profile \cite{vazquez2007porous}.
We illustrate this phenomenon on \cref{fig:cvPorous} (right), where the initial data array $u_0 : \bT_h \to \bR$ vanishes except for three arbitrary positive values ($\cdots,3/10,1,0,0,0,7/10,\cdots$) close to the domain center. Setting the final time to $T=1$, using only $N_\tau = 10$ timesteps and (arbitrarily) $N_h = 128$ discretization points, the final time solution resembles as expected the bell-shaped Barenblatt profile.

For comparison, the explicit scheme (\ref{eq:arithScheme}, $\theta=1$) is subject to the CFL condition  $\max(u) \tau \leq h^2$ (equivalently $N_\tau \geq 2 \max(u) N_h^2$), and thus diverges unless $N_\tau \gtrsim 33000$ timesteps are used in this example.
The semi-implicit scheme (\ref{eq:arithScheme}, $\theta=1/2$), which can be solved one timestep at a time but also corresponds to the choice of the arithmetic rather than harmonic mean of the density in the discretized energy functional, compare \eqref{eqdef:porousEnergy} with \eqref{eq:arithEnergy} in \mbox{
\cref{sec:arith}}\hskip0pt
 , diverges unless $N_\tau \geq 5000$ timesteps are used in this example (it seems that larger timesteps can be used with smooth initial conditions). Finally, the implicit scheme (\ref{eq:arithScheme}, $\theta=0$) is unconditionally stable, hence $N_\tau=10$ timesteps can be used as well, but it is only first order accurate w.r.t.\ time.

\begin{figure}
\def\h{3.6cm}
	\includegraphics[height=\h]{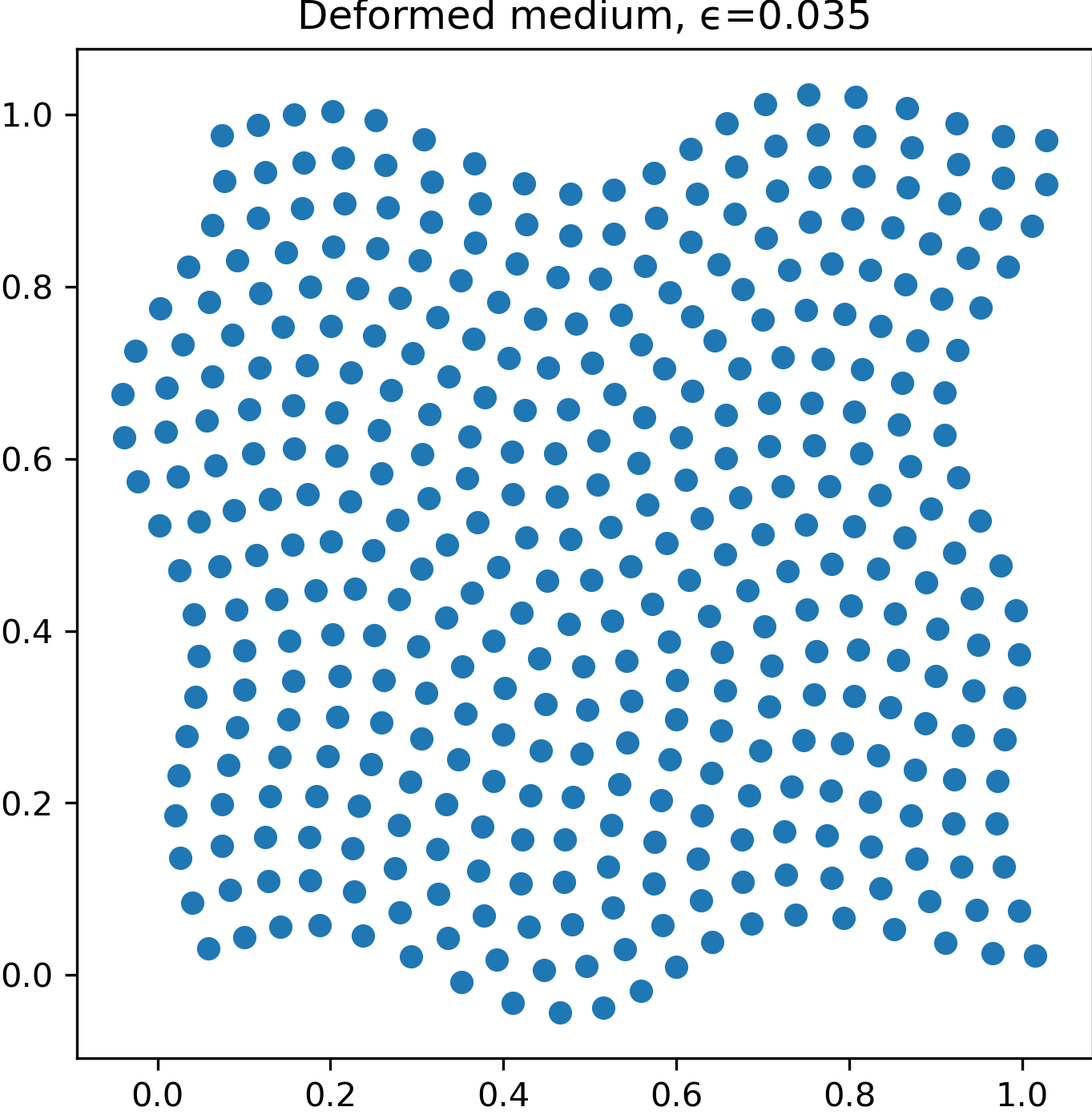}
	\includegraphics[height=\h]{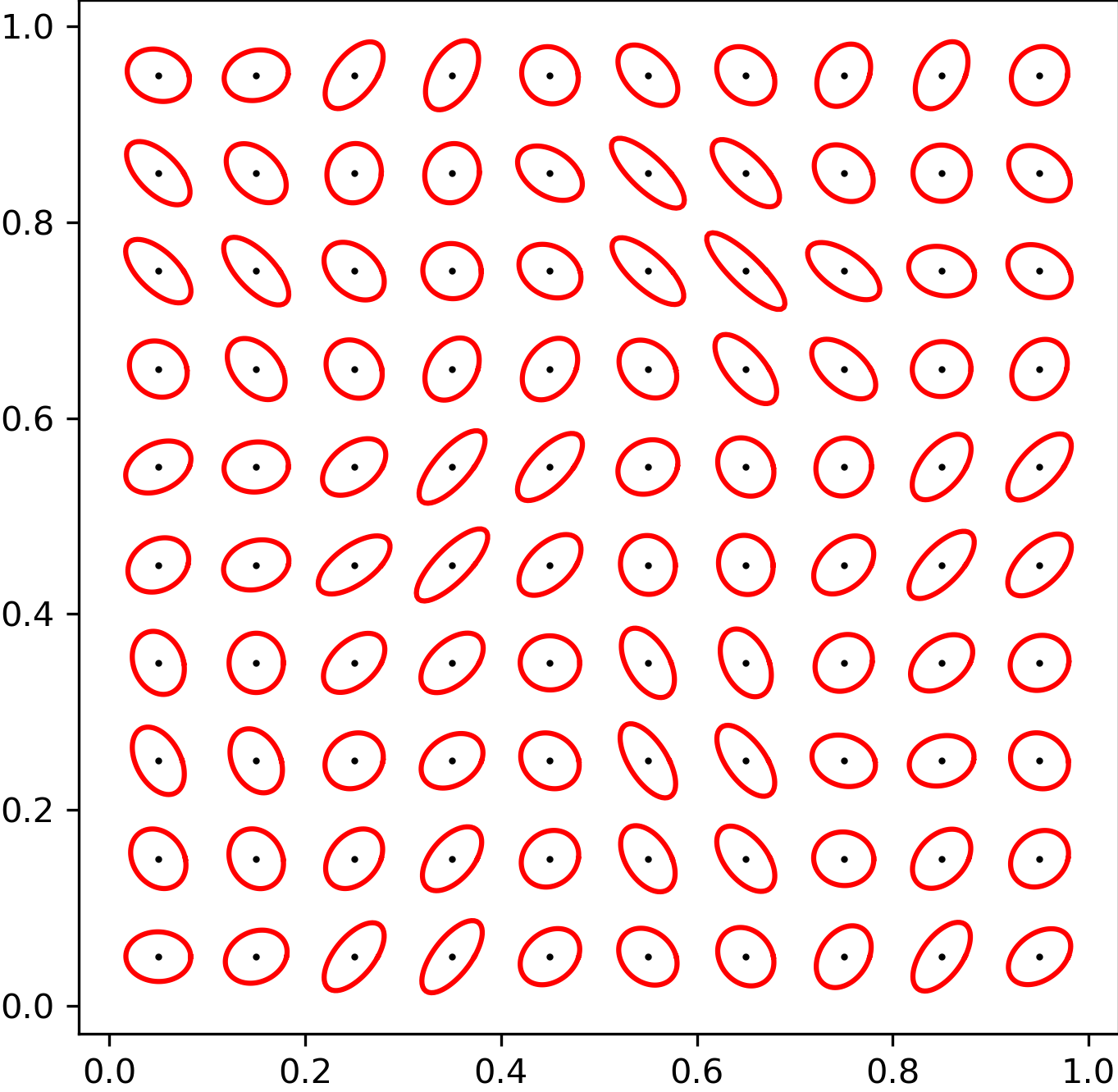}
	\includegraphics[height=\h]{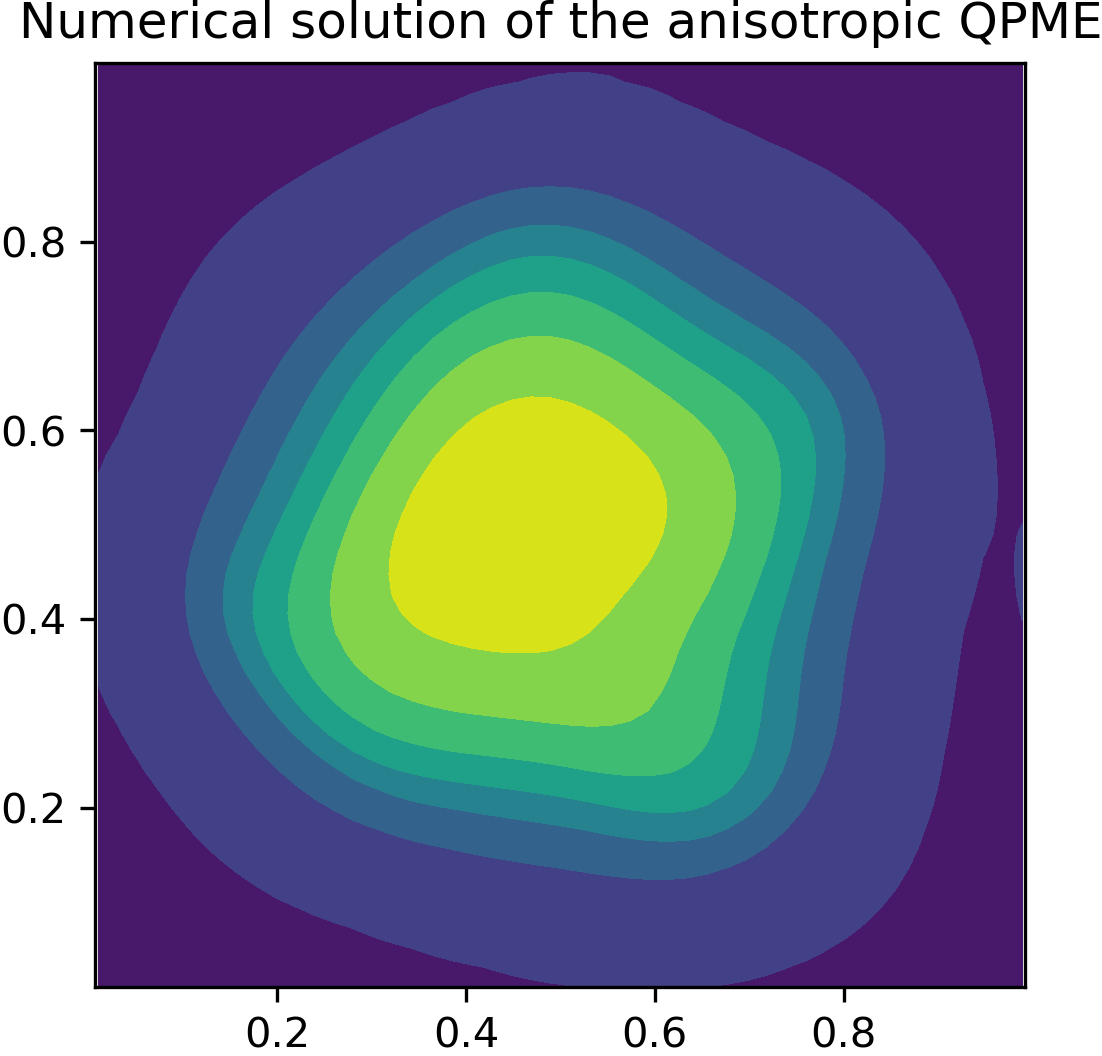}
	\includegraphics[height=\h]{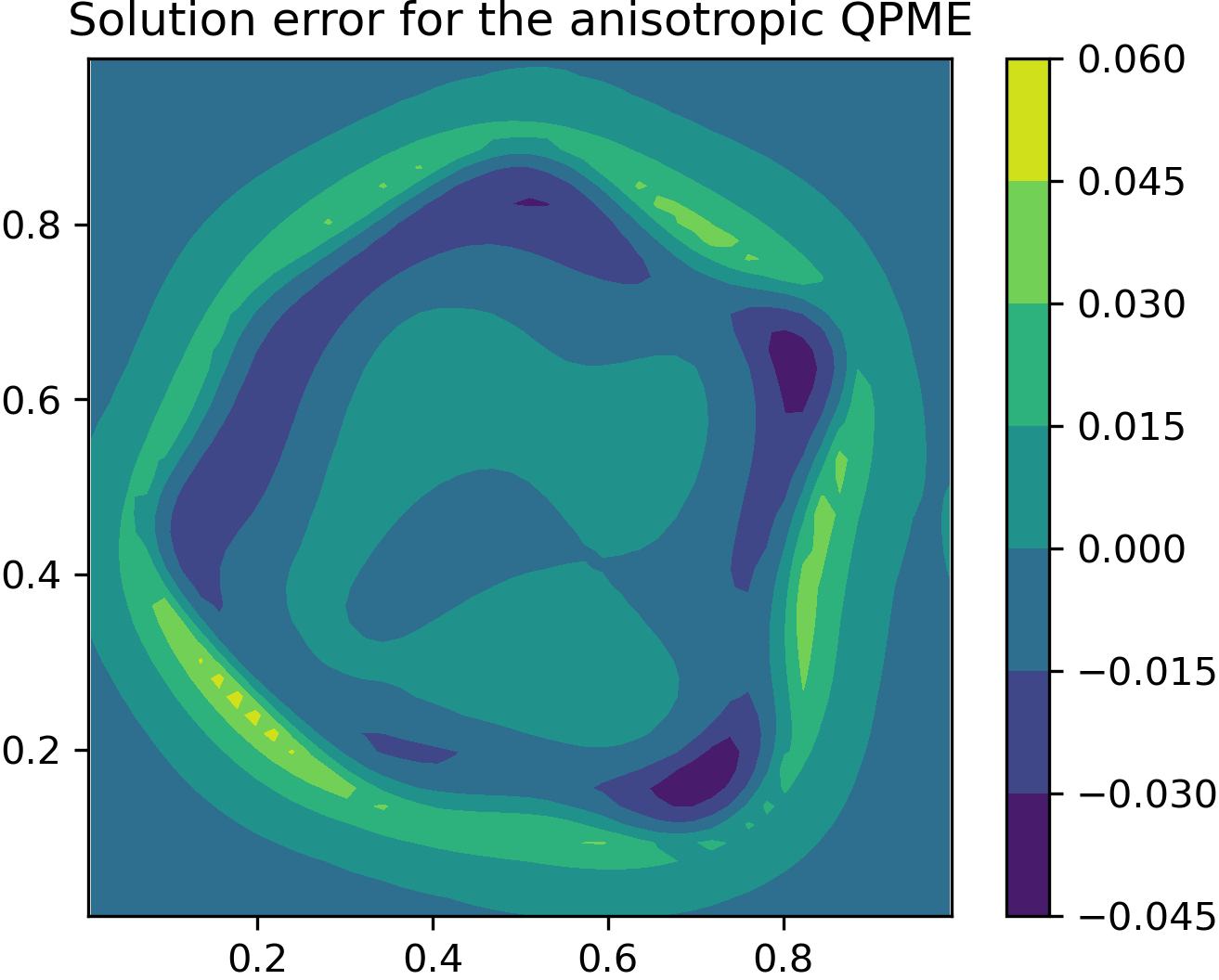}
	\caption{
	I: Image of the Cartesian grid by the chosen diffeomorphism $\phi$. 
	II: Visualization of the anisotropy via the ellipses $\cE(x) := \{v \in \bR^2 \mid \|\Diff \phi(x) v\| = r\}$, for some $r>0$. 
	III: Numerical solution of the anisotropic QPME, on a small grid with $N_\tau=12$ and $N_h=48$.
	IV: Difference between the numerical and the exact solution, which is obtained as the composition of the Barenblatt profile $u^\Porous$ with the diffeomorphism $\phi$.
	}
	\label{fig:QPME_Aniso}
\end{figure}

\paragraph{Two-dimensional anisotropic.}
Let $\psi \in C^2(\bR^d,\bR^d)$ be a diffeomorphism whose Jacobian matrix obeys $|\det(\Diff \psi)| = 1$ identically, i.e.\ $\psi$ preserves the Lebesgue measure, and let $u(t,x)$ obey the isotropic QPME. 
Then $(t,x)\mapsto u(t,\psi(x))$ obeys the anisotropic QPME associated with the diffusion tensor field $\cD(x) := (\Diff \psi(x)^\top \Diff \psi(x))^{-1}$, by the usual change of variable formulas for the gradient and divergence.
In this experiment, we choose $u$ as 
the Barenblatt profile with $\gamma=1$ on the time interval $[10^{-5}, 10^{-4}]$, and we define $\psi(x_0,y_0) = (x_2,y_2)$ via $x_1 = x_0+\epsilon \sin(2\pi y_0+1)$, $y_1 = y_0+\epsilon \sin(4\pi x_1+5)$, $x_2 = x_1+\epsilon \sin(4\pi_1 y+3)$, $y_2=x_2+\epsilon \sin(2\pi x_2+2)$, with $\epsilon = 0.035$. 
By construction, $\psi$ is a $1$-periodic measure preserving diffeomorphism, illustrated on \cref{fig:QPME_Aniso}, with maximum distortion ratio $\approx 3.5$ (ratio of the largest to the smallest singular value of $\Diff \phi$).
Selling's decomposition of this tensor field $\cD$ yields the collection $E := \{(1,0),(0,1),(1,1),(1,-1)\}$  of offsets, and a \emph{piecewise smooth} \cite[Proposition B.5]{bonnans2022randers} weight $\lambda^e$ for each $e\in E$ obeying (\ref{eqdef:Lh}, right), as required for the construction of the proposed scheme \eqref{eq:LhExpl}. 
A modified decomposition introduced in \cite[Theorem 1.8]{bonnans2023monotone} yields \emph{smooth} weights $\tilde \lambda^e$, fitting \cref{assum:porous}, but uses slightly more offsets $e \in \tilde E = E \cup \{(2,1),(2,-1),(1,2),(-1,2)\}$, which increases the memory usage and the numerical cost for no observable benefit in this experiment. 
A numerical solution obtained on a small grid $N_\tau = 12$, $N_h = 48$, is shown on \cref{fig:QPME_Aniso}.

\paragraph{Scalability.}
The proposed BBB discretization of the QPME works in arbitrary space dimension $d$, but its numerical cost is of course subject to the curse of dimensionality, especially so since the  numerical method manipulates space-time $(d+1)$-dimensional arrays.
The primal-dual proximal optimization algorithm is embarrassingly parallel in time and space, assuming a suitable implementation of the FFT, and thanks to the Taichi\textsuperscript{\textregistered} library our numerical implementation can be executed on the CPU or the GPU without modification.
In order to investigate the scalability of the proposed method, we ran the above two-dimensional anisotropic QPME experiment on a larger grid: with $N_\tau=64$ and $N_h=128$, thus $N_\tau N_h^2\approx 10^6$ space-time discretization points, computations take $7.9s$ per proximal iteration on the CPU, and $0.069s$ on the GPU, which is a $\approx 110 \times$ performance improvement.
Similarly, the isotropic two-dimensional QPME with $N_\tau=128$ and $N_h = 512$, thus $N_\tau N_h^2 \approx 33\times 10^6$ space-time discretization points, takes $10.5s$ per proximal iteration on the CPU, and $0.085s$ on the GPU, which is a $\approx 120\times$ performance improvement.
Approximately $2000$ proximal iterations are needed to achieve good accuracy.
Laptop equipped with an Intel Core i7, 8 core, 2.3Ghz CPU, and an Nvidia 2060 Max-Q GPU. The chosen dimensions $N_\tau$ and $N_h$ are the largest powers of two such that the computations fit in the GPU 8GB memory. The anisotropic implementation has higher memory usage, for an identical gridsize $N_\tau N_h^2$, due to the additional unknowns \eqref{eq:mene}.

\subsection{Burgers' equation}
\label{subsec:numBurgers}

\begin{figure} 
\def\h{3.25cm}
	 \includegraphics[height=\h]{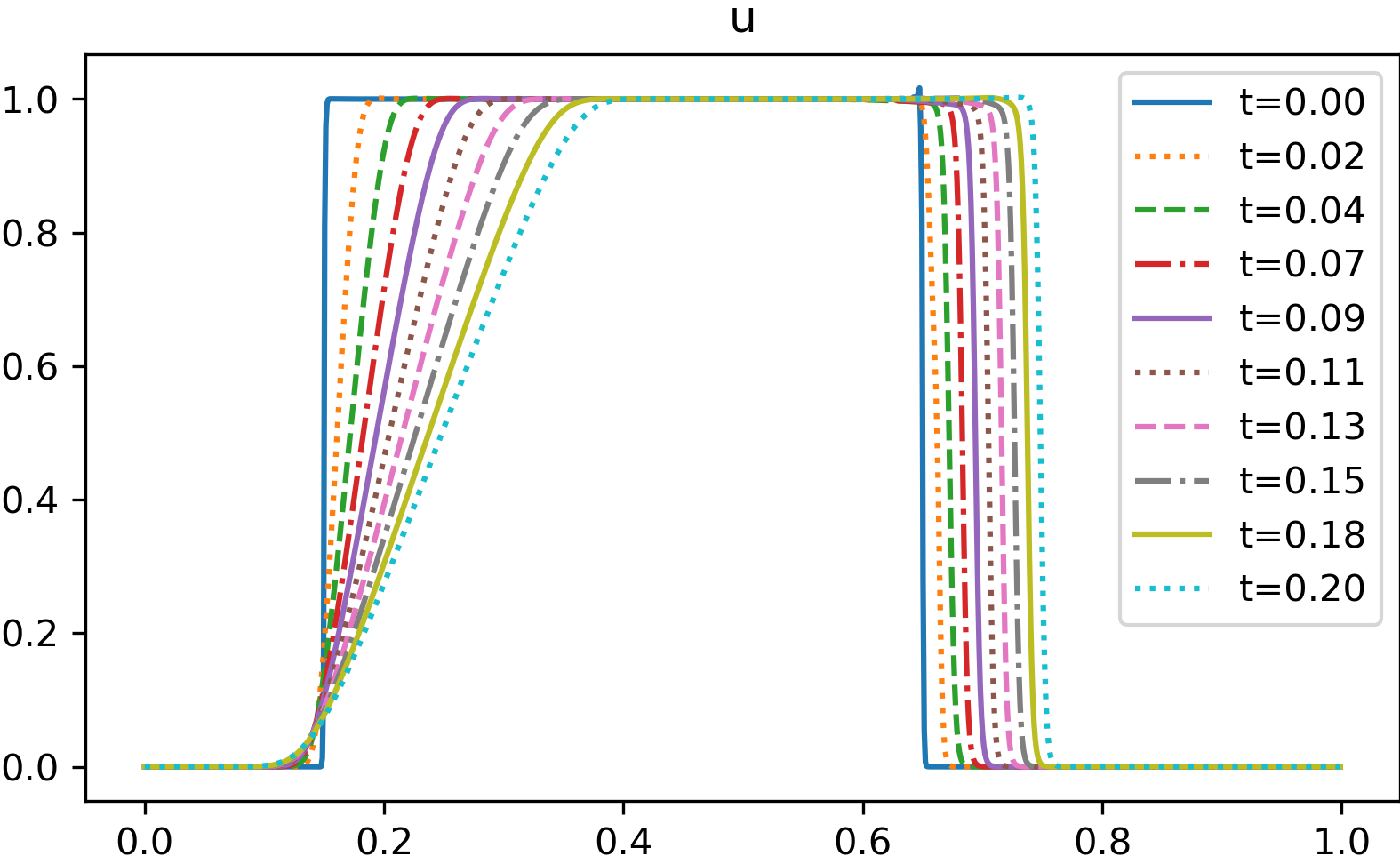}
	\includegraphics[height=\h]{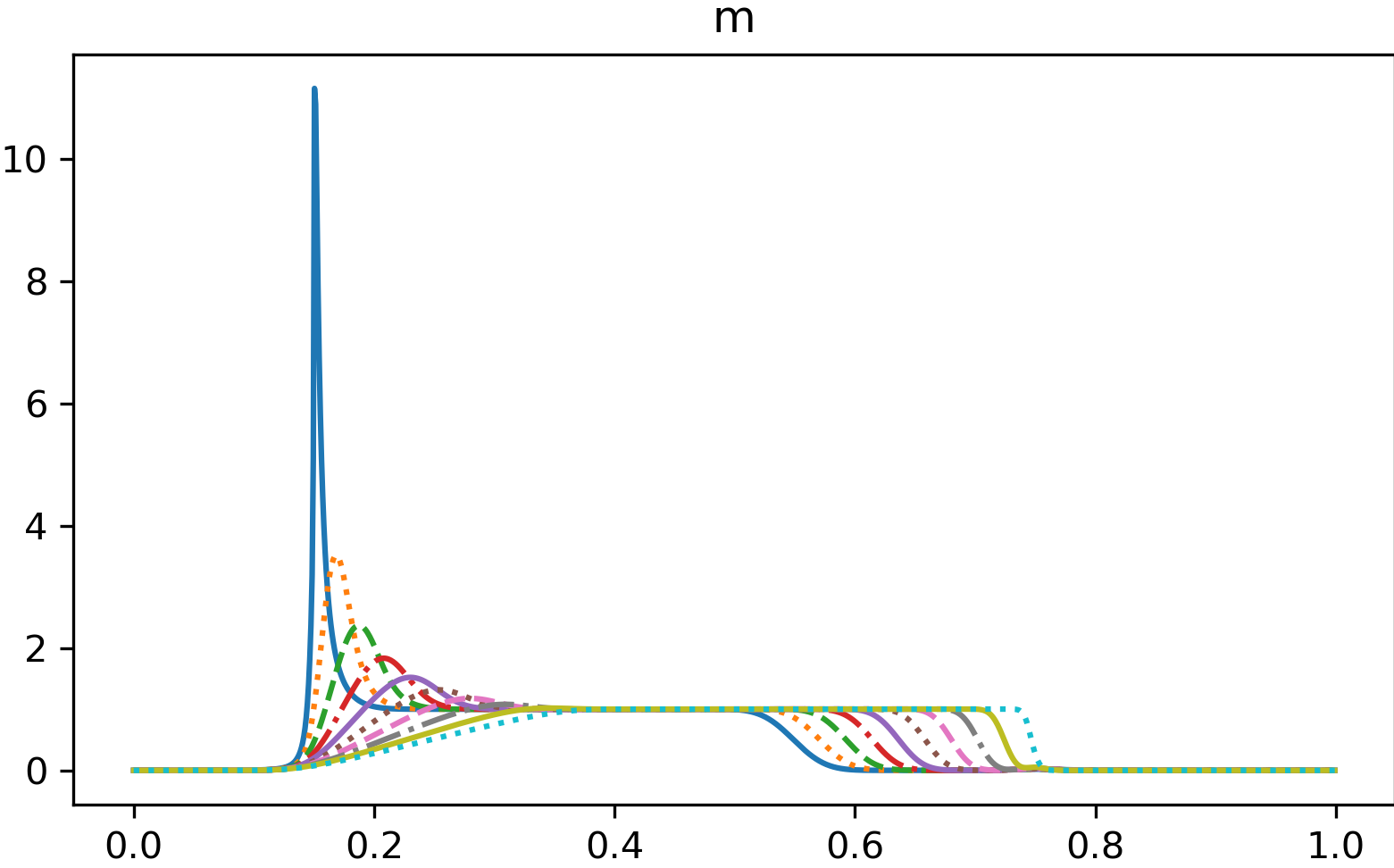}
	\includegraphics[height=\h]{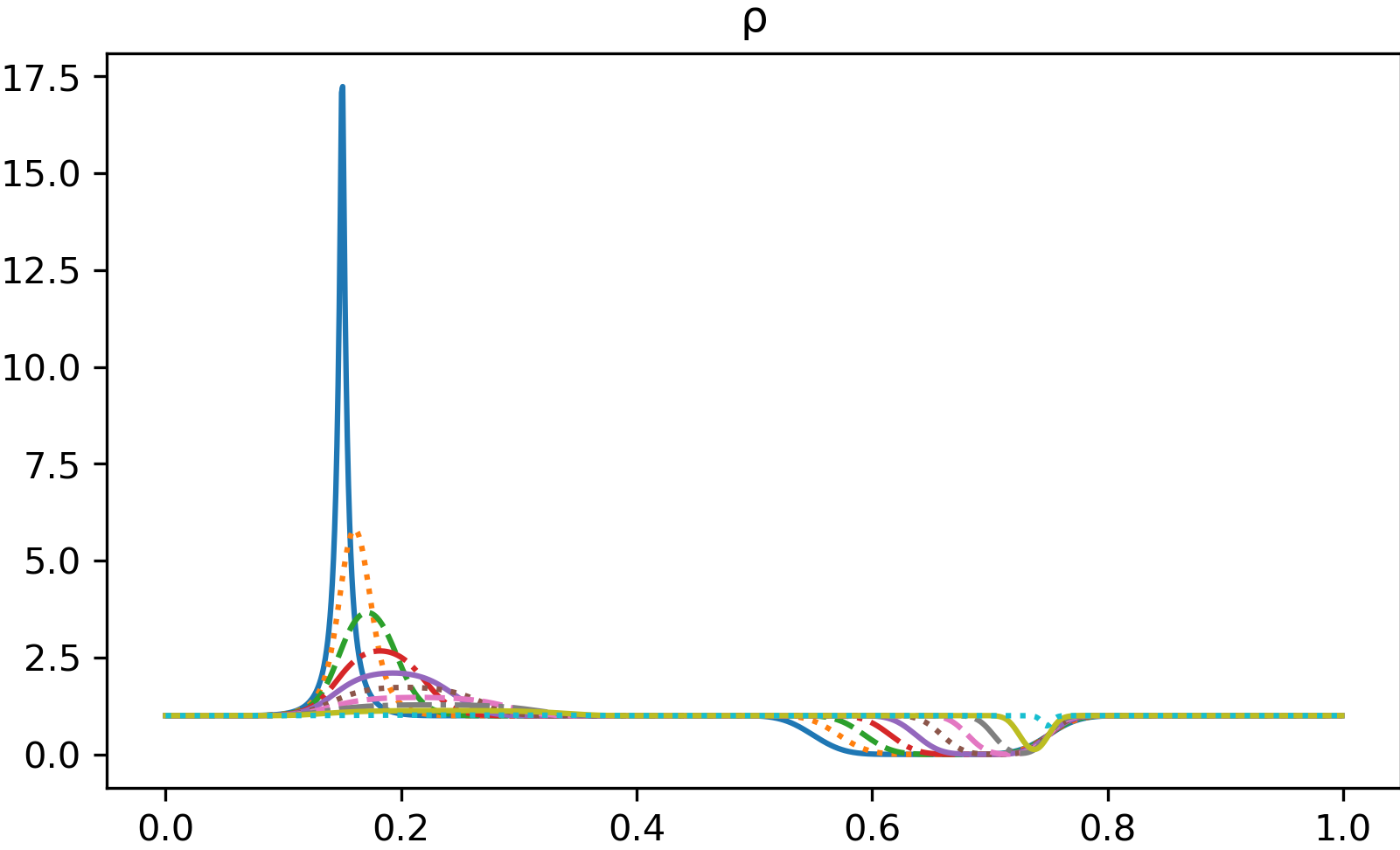} \\ \includegraphics[height=\h]{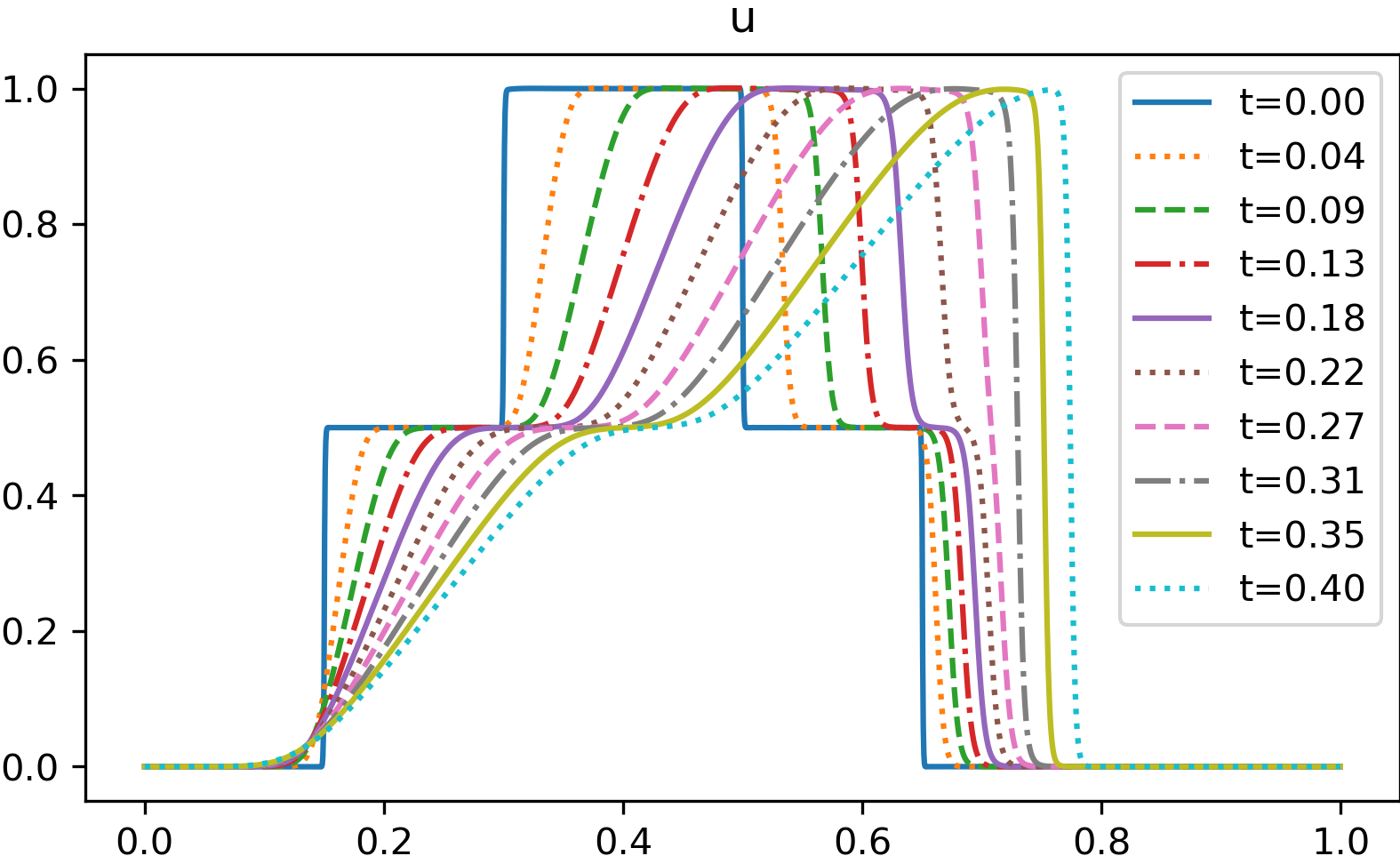}
	\includegraphics[height=\h]{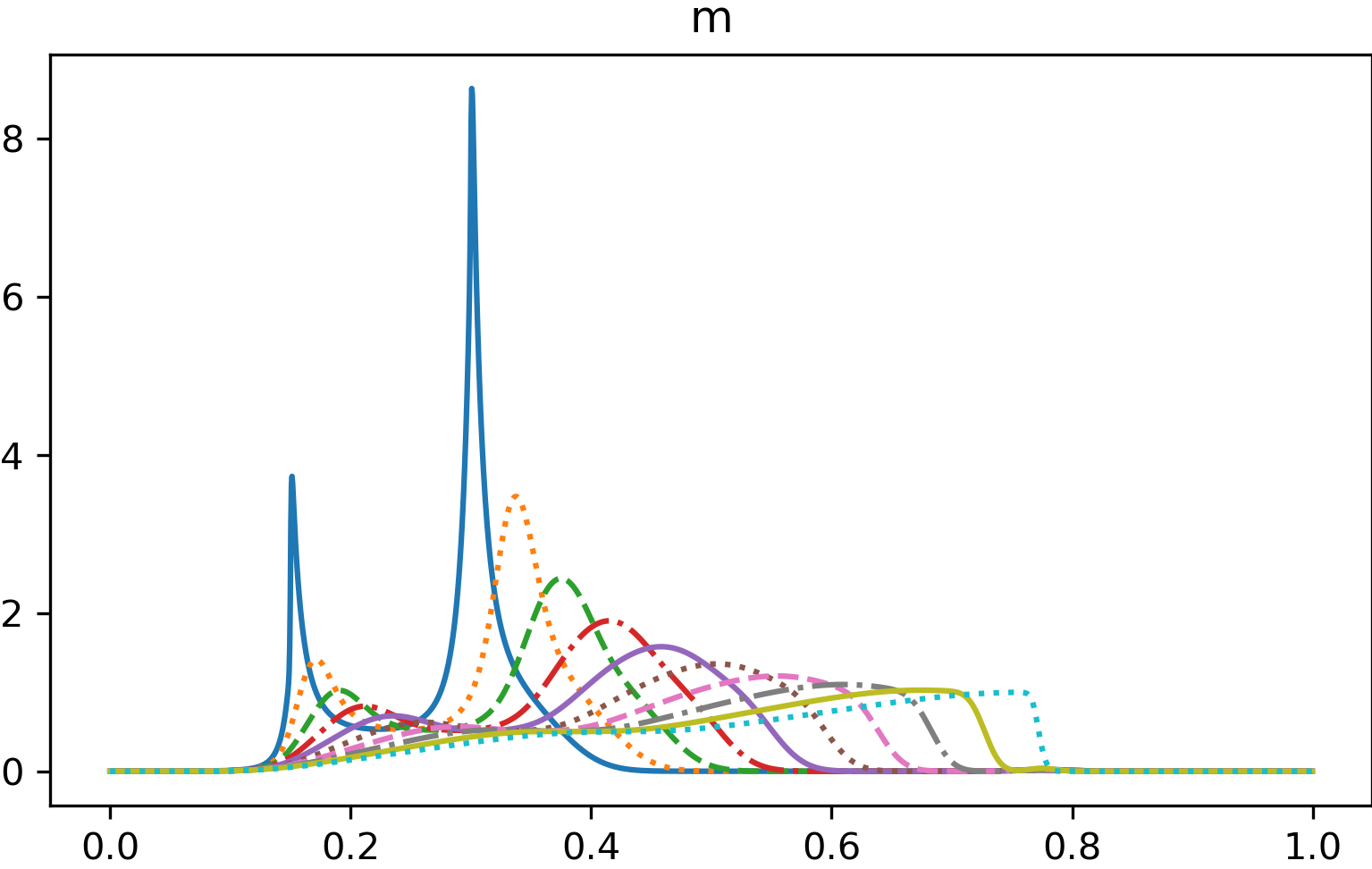}
	\includegraphics[height=\h]{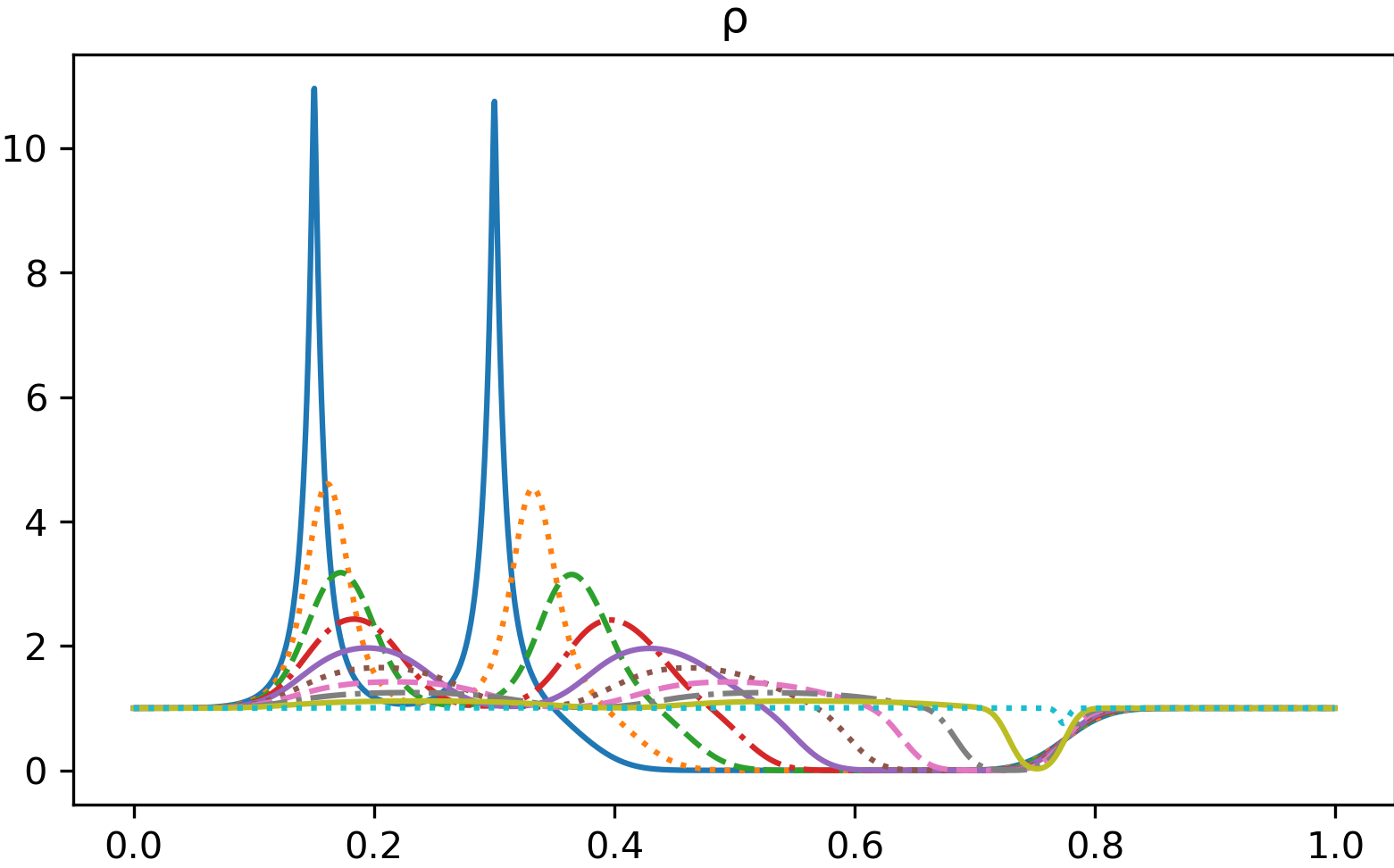}
	\caption{
	Top: Burgers' equation with viscosity $\nu = 10^{-3}$, initial condition $u_0 = \bbone_{0.15\leq x \leq 0.65}$, evolution time $T=0.2$, $N_\tau = 512$, $N_h=1024$, $N_{\prox} = 6000$.
	Bottom: Burgers' equation with viscosity $\nu = 10^{-3}$, initial condition $u_0 = \frac{1}{2} (\bbone_{0.15\leq x \leq 0.65} + \bbone_{0.3\leq x \leq 0.5})$, evolution time $T=0.4$,
	$N_\tau = 512$, $N_h=1024$, $N_{\prox} = 6000$.
	}
	\label{fig:burgersInterval}
\end{figure}

\begin{figure} 
\def\h{3.25cm}
	 \includegraphics[height=\h]{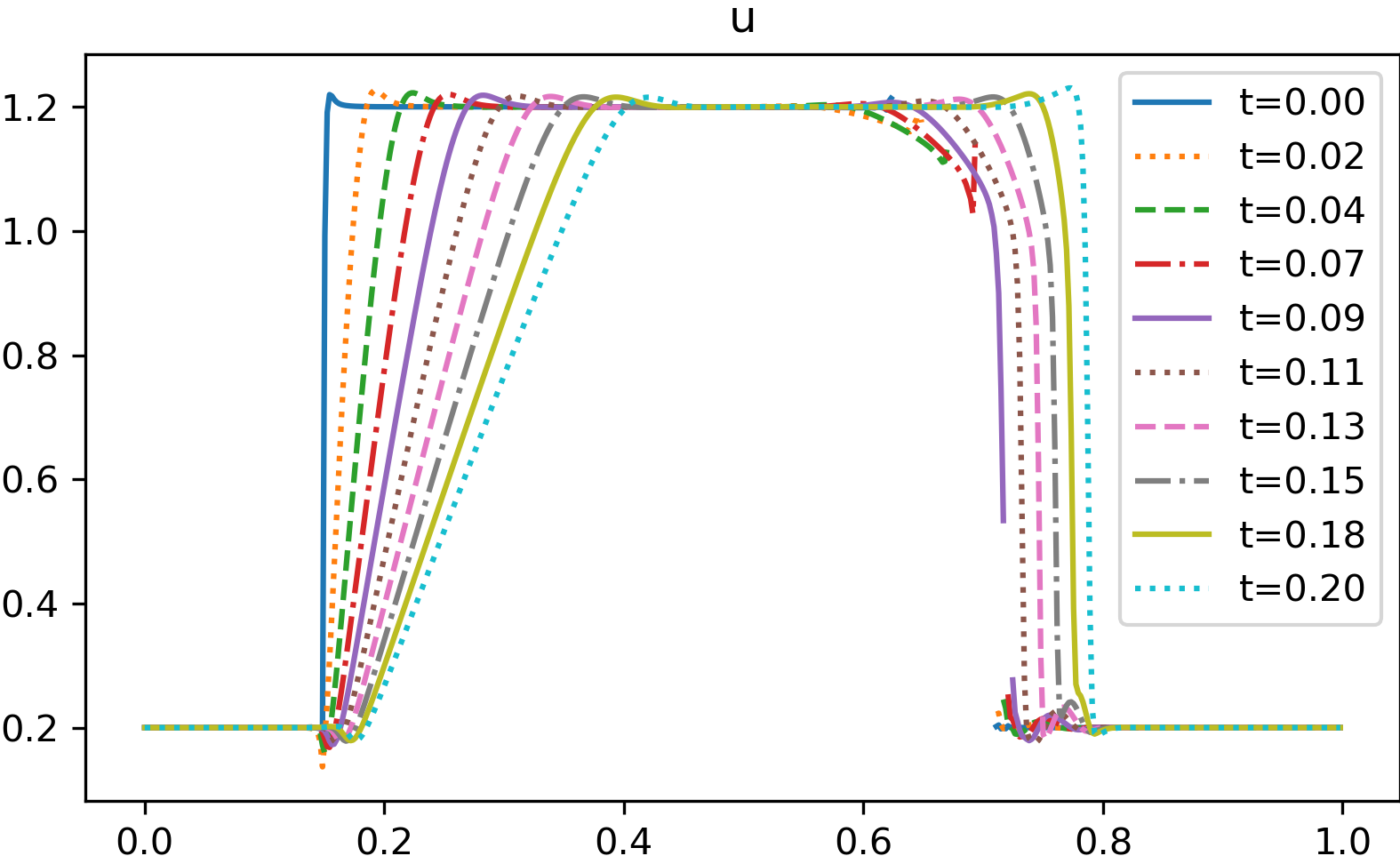}
	\includegraphics[height=\h]{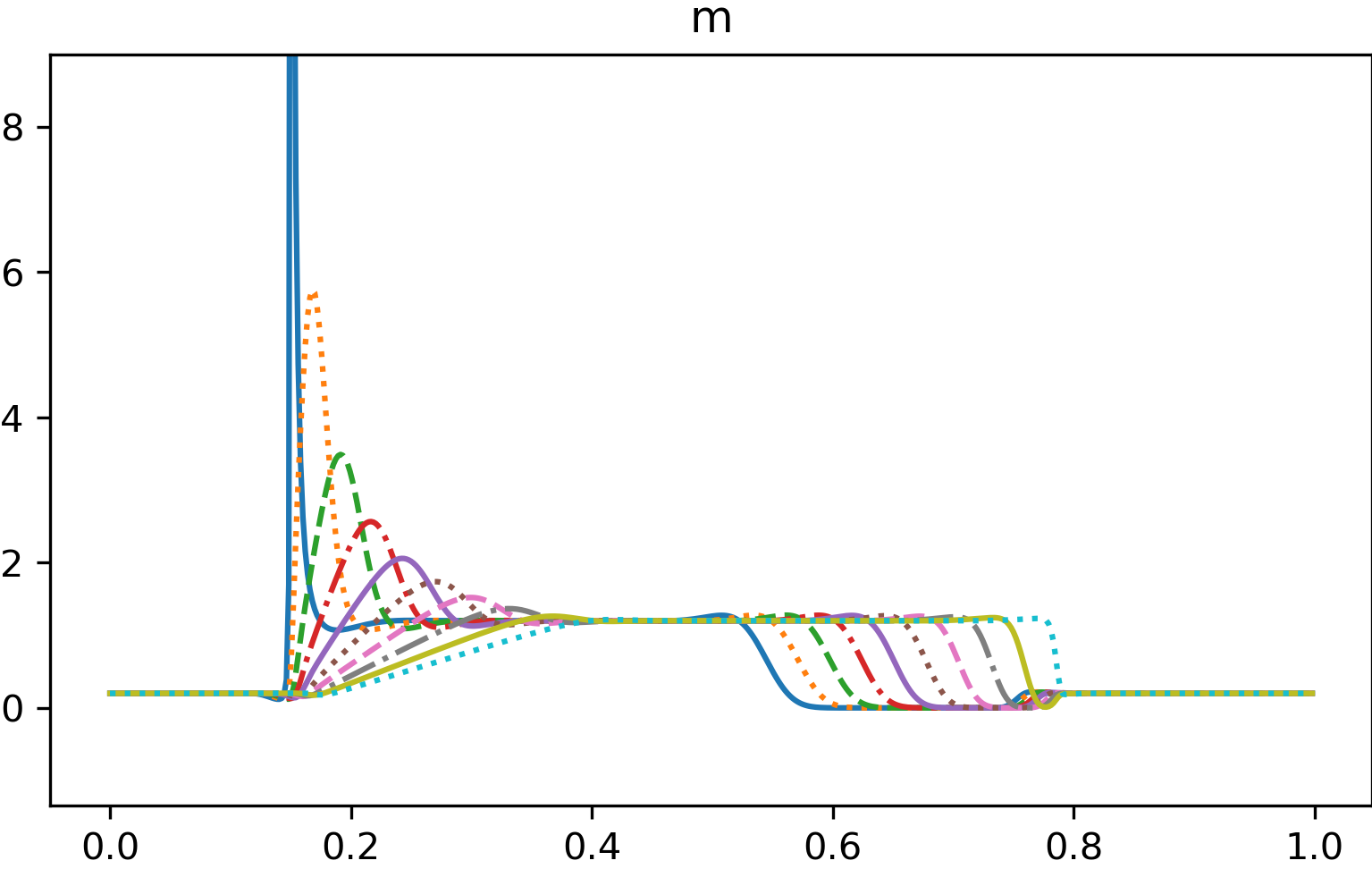}
	\includegraphics[height=\h]{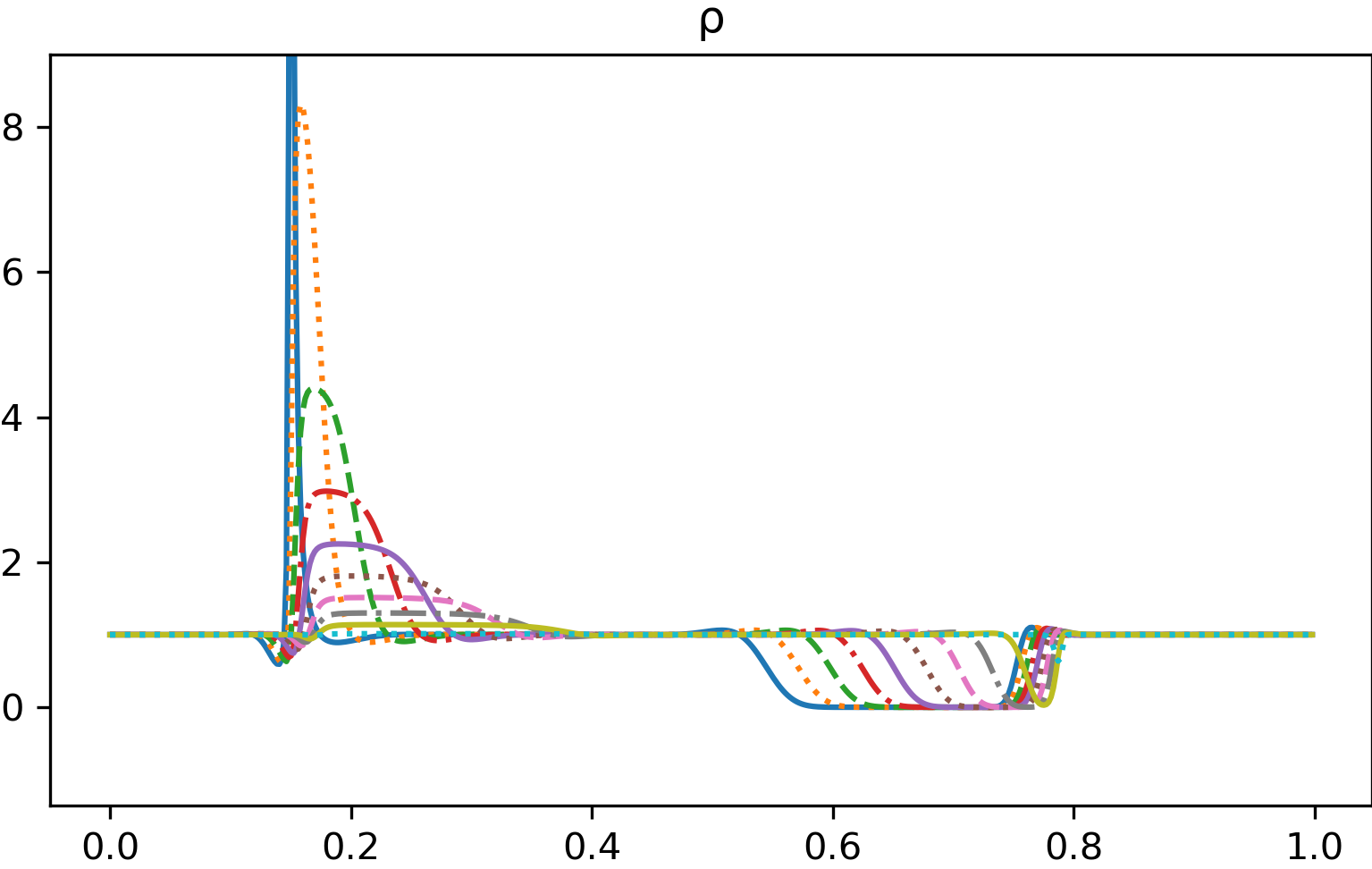}		
	\caption{
	Failed attempt at the numerical solution of Burgers inviscid equation ($\nu=0$), initial condition $u_0 = \bbone_{0.15\leq x \leq 0.65}$, evolution time $T=0.2$, $N_\tau = 256$, $N_h=512$, $N_{\prox} = 40000$.
	The numerical solution is reasonably good at the final time $t=0.2$, but is incorrect, or non-converging, at earlier times. In particular, the initial condition ($t=0.0$, dark blue) is not reproduced.
	}
	\label{fig:inviscid}
\end{figure}

\begin{figure} 
\def\h{3.7cm}
 \includegraphics[height=\h]{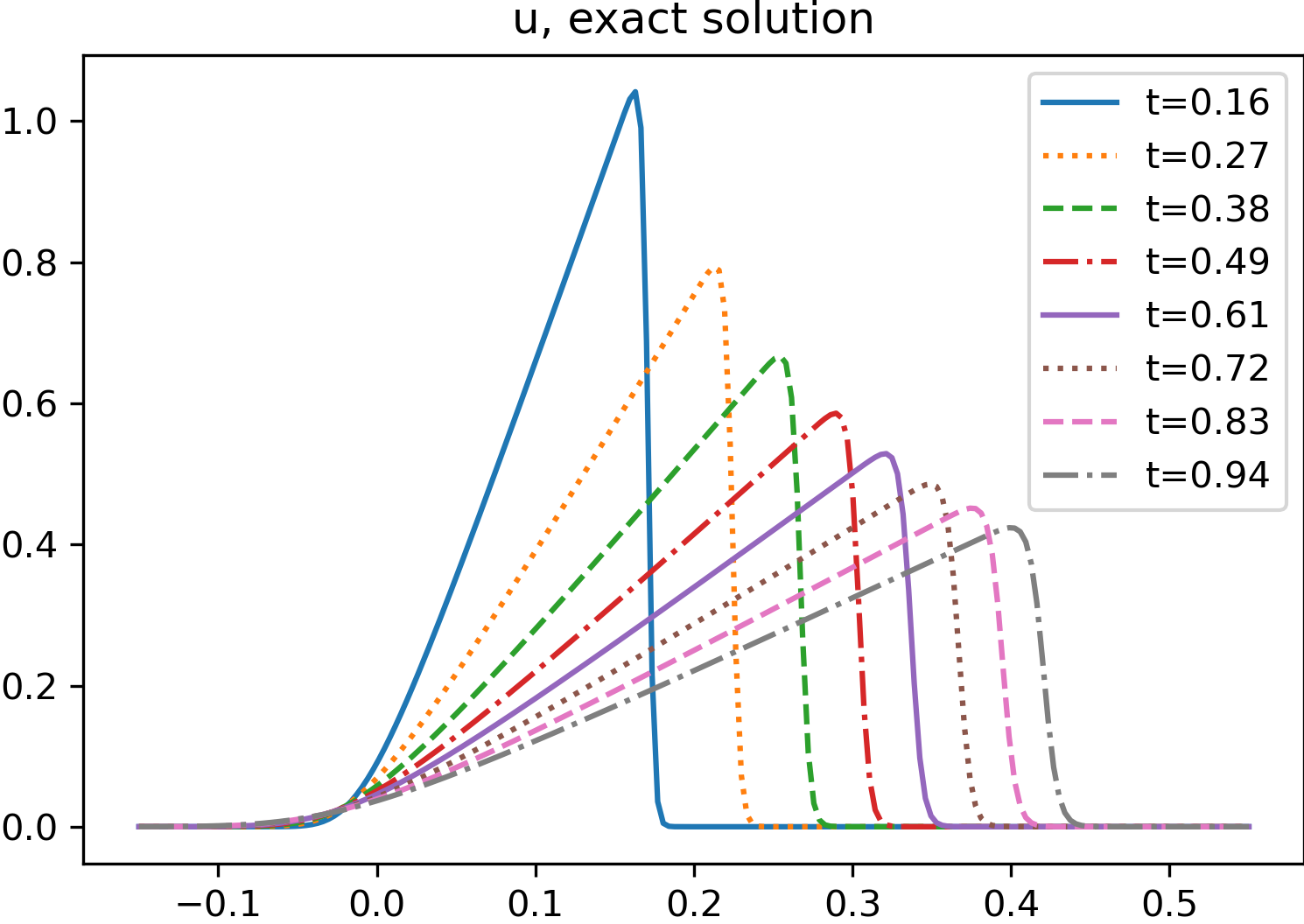}
\includegraphics[height=\h]{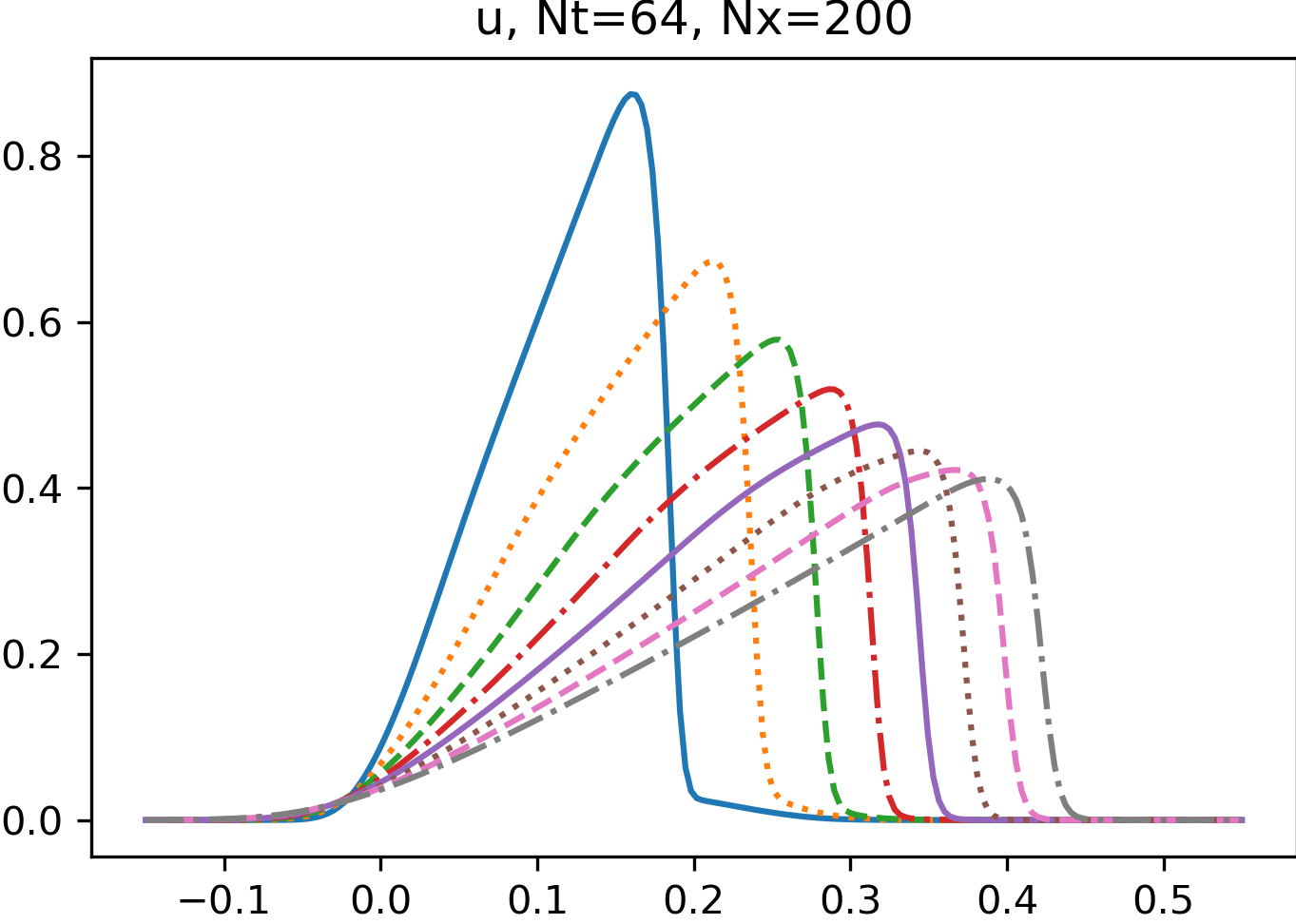}
\includegraphics[height=\h]{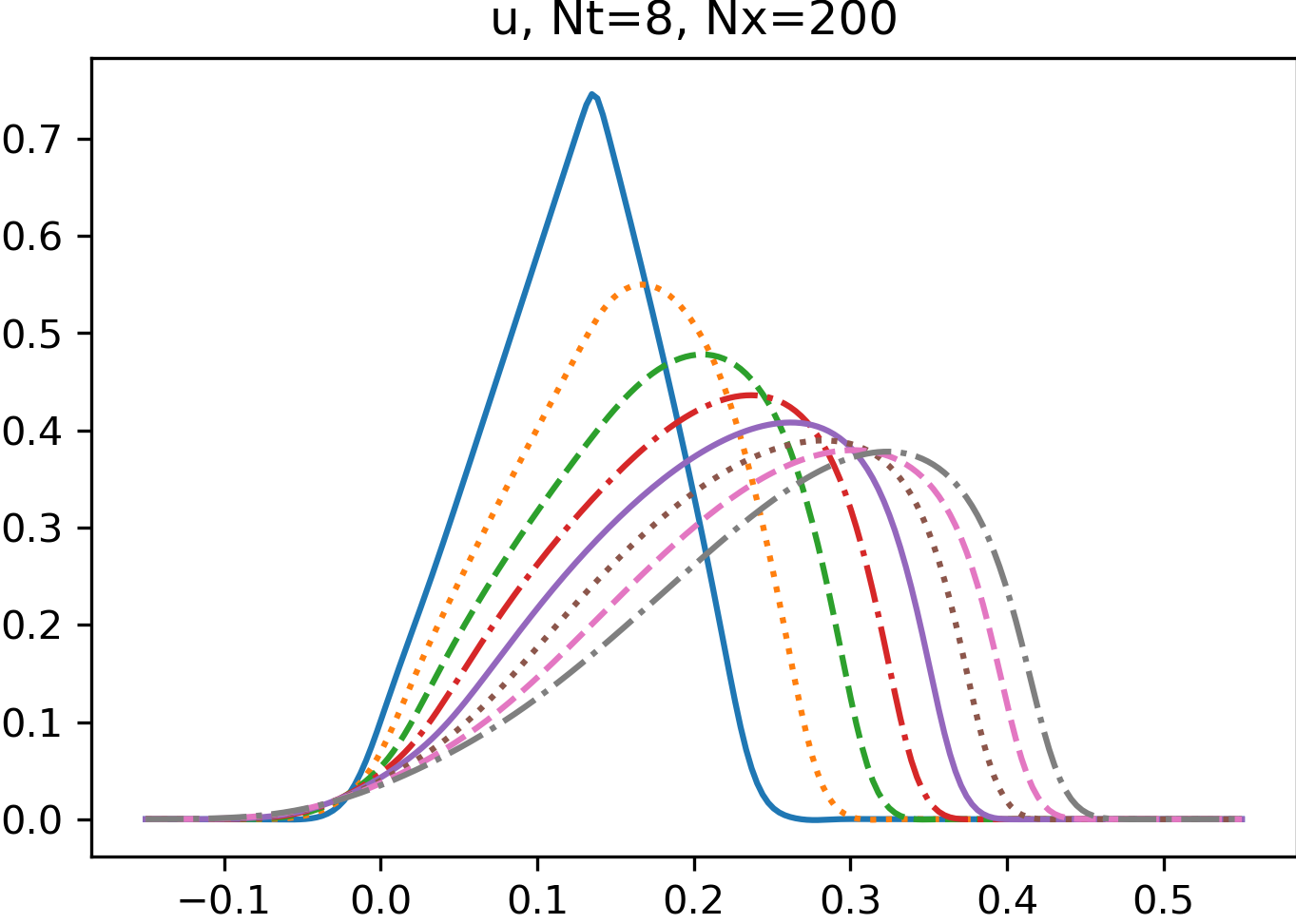}
\caption{
Left: exact solution of Burgers' equation on the time interval $[0.1,1]$ with $\nu=10^{-3}$ and $\delta = \exp(\Re)-1$ where $\Re=50$ is the Reynolds number. 
Center: numerical solution obtained with $N_\tau = 64$ timesteps and $N_h = 200$ discretization points.
Right: numerical solution obtained with only $N_\tau=8$ timesteps, and $N_h=200$, see the discussion on large timesteps in \cref{subsec:numBurgers}.
}
\label{fig:large_tau_Burgers}
\end{figure}

\paragraph{Large viscosity.}
In order to assess the accuracy of the proposed numerical method, we reproduce the known solution (\ref{eqdef:Barenblatt}, right), with viscosity $\nu=10^{-2}$ and parameter $\delta = \exp(\Re)-1$ where $\Re=5$ is regarded as a Reynolds number, from the time $T_0=0.1$ to $T=1$, and on the interval $[-0.6,0.9]$ outside of which the solution vanishes up to machine precision. We solve a sequence of small problem instances, where the number of timesteps ranges from $4$ to $400$ and the number of spatial discretization points is chosen as $N_h=5N_\tau$, and we minimize the energy functional of interest to machine precision using a Newton method, which only requires $7$ iterations here independently of the problem size.
In contrast with the assumptions of \cref{th:cvBurgers}, the problem solution is not really positively bounded below since the quadratic exponential $\exp(\frac{-x^2}{4 \nu t})$ in \eqref{eqdef:Barenblatt} falls below machine precision on the domain boundary, and in addition error is measured on the primal variable $u$ rather than the dual potential $\phi$. Nevertheless, we observe a second order convergence rate in the $L^\infty$ norm, see \cref{fig:cvPorous}, which could be expected since the solution is smooth.

\paragraph{Small viscosity.}
We illustrate on \cref{fig:burgersInterval} the numerical solution obtained with a small viscosity $\nu = 10^{-3}$, and when the initial condition is the characteristic function of a compact interval, or the average of two such characteristic functions. This creates rarefaction waves on the left, and shocks on the right, which evolve and merge as expected.
From the theoretical standpoint, the density $\rho^T$ is positively bounded below for any  final evolution time $T$ in the BBB formulation, see \cref{remark:finalT}. Let us acknowledge however that if $T$ is chosen too large, then the values of $\rho$ fall below machine precision ($32$bit floats are used for these GPU experiments), which raises numerical convergence issues, not illustrated here, see the discussion at the end of \cref{sec:nogapBV}.\\

Depending on one's perspective, the last experiments may be regarded as negative results.
\vspace{-3mm}
\paragraph{Non-viscous.}
The solution $u^T$  of the BBB formulation of Burgers' equation on the interval $[0,T]$, differs in general from the viscosity solution $u$ if $T>T_{shock}$, except at the \emph{final} time: $u^T(T,\cdot) = u(T,\cdot)$. 
More precisely, \cite[Theorem 5.5.2]{brenier2020examples} characterizes $u^T$ as the unique shock free solution of Burgers' equation on $[0,T]$ such that $u^T(T,\cdot) = u(T,\cdot)$.
In our experiment (using the proximal solver on the GPU), the numerically computed approximation of $u^T$ does match the viscosity solution at the final time $T$, but typically fails to converge for smaller values $t\in [0,T[$, likely because the density $\rho$ falls below machine precision, see \cref{fig:inviscid}. 

\def\visc{\mathrm{ent}}
\def\free{\mathrm{free}}

The failure of the proposed numerical method to produce $u^T(t,x)$ for $t \in [0,T[$ in the inviscid case is unfortunate, but let us nevertheless acknowledge that this value is mostly of theoretical interest, since the shock free solution $u^T$ differs in general from the entropic solution $u(t,x) \neq u^T(t,x)$, except at the final time $t=T$ for which the numerics do work. 
An interesting research perspective would be to understand the behavior of the discretized BBB solution $u^T_{h,\nu}$ with gridscale $h$ and viscosity $\nu$, as these two parameters tend simultaneously to zero with different asymptotic behaviors (and e.g.\ timestep $\tau=h$ for simplicity): are there viscosity scaling laws $\nu_\visc(h)$ and $\nu_\free(h)$, vanishing in the limit $h \to 0$, and such that $u^T_{h,\nu_\visc(h)}$ converges to the entropic solution $u$, and $u^T_{h,\nu_\free(h)}$ converges to the shock free solution $u^T$? If so, what is the convergence rate ?

 A strategy around the convergence issue is proposed in \cite{kouskiya2024inviscid}. In that work, the BBB formulation is successively discretized on a collection of small time sub-intervals, with some overlapping, rather than a single large interval. In addition, crucially, a well chosen \emph{base state} is introduced in the energy functional, see \cref{remark:finalT}. It may either be generated numerically, or defined a-priori e.g.\ as the solution of Burgers' equation with a small but positive viscosity parameter \cite[\S 5.4]{kouskiya2024inviscid}.

\paragraph{Large time steps.}
We consider the exact solution (\ref{eqdef:Barenblatt}, right), with smaller viscosity $\nu=10^{-3}$ and higher Reynolds number $\Re=50$, on the time interval $[0.1,1]$ and the domain $[-0.15,0.55]$, see \cref{fig:large_tau_Burgers}.
This solution is reasonably well reproduced using the proposed BBB formulation with $N_\tau=64$ timesteps, $N_h=200$ discretization points, here solved using the damped Newton method. In order to illustrate the unconditional stability of the scheme, we also solve the BBB formulation with excessively few $N_\tau=8$ timesteps. The numerical solution accuracy is degraded, as could be expected, but interestingly this degradation includes the initial condition which is not well reproduced: the `right angle' is turned into a symmetric `hat', see \cref{fig:large_tau_Burgers} (right). It is not clear to us whether this is related to the phenomenon of the previous paragraph. 


\section{Conclusion}

In this paper, we presented the first numerical convergence analysis of the ballistic Benamou-Brenier (BBB) formulation of non-linear evolution PDEs, applied to the (possibly anisotropic) quadratic porous medium equation (QPME) and the (possibly viscous) Burgers' equation. 
The obtained schemes are second order accurate w.r.t.\ the timestep and the gridscale, and are free of any Courant-Friedrichs-Levy condition relating these parameters. They take the form of an optimization problem, whose structure is reminiscent of a mean field game, and which can be solved using a proximal primal-dual method and space-time fast Fourier transforms.

Numerical experiments with the QPME show that the approach efficiently handles large timesteps, and strong anisotropies, while retaining second order accuracy and being easily parallelizable. Limitations come from memory usage, due to the use of space-time global solution arrays, and from the slow convergence of the proximal optimizer. 
The approach also applies to Burgers' equation with a (possibly small) positive viscosity, but suffers from delicate theoretical obstructions in the inviscid case \cite{brenier2018initial}. 

Numerous perspectives are open for future work. 
The PDE models may be generalized, by introducing a right hand side or non-periodic boundary conditions. Alternative discretizations may be considered, using e.g.\ triangulations instead of a Cartesian grid. The convergence analysis which is currently limited to smooth and uniformly positive solutions might be extended so as to accommodate e.g.\ the Barenblatt profile. The primal-dual optimization method, which currently limits the numerical efficiency of our approach, might be improved using e.g.\ the uniform convexity properties underlying our convergence analysis. 
The introduction of a weight in the energy functional, decaying exponentially with time following \cite{vorotnikov2025hidden}, should be investigated and the numerical method suitably adapted. 
Finally, and most importantly, the numerical BBB approach will be investigated in the case of systems of PDEs, related to e.g.\ fluid mechanics, which are a fundamental motivation of this theoretical framework \cite{brenier2020examples,vorotnikov2022partial}.

\paragraph{Acknowledgement.}
The authors would like to thank Yann Brenier and Thomas Gallouët for numerous constructive discussions. The authors also thank the anonymous referees for their constructive comments which helped at lot in improving the clarity and content of the manuscript.

\bibliographystyle{alpha}
\bibliography{Papers}
\appendix
\section{Burgers' equation}
\label{sec:Burgers}

This section is devoted to the proof \cref{th:cvBurgers}, our convergence result for the discretization of Burgers' equation. Incidentally, it shows that the proof template of \cref{sec:cv} for the QPME can accommodate variations, in this case (i) a non-linearity associated with a first order operator (rather than second order), (ii) an optional viscous linear term, and (iii) a staggered grid discretization in space, in addition to time. 


\begin{notation}
	We fix in the following the half timestep $\tau>0$ and the half gridscale $h>0$. 
	For readability and in order to avoid double sub- or super-scripts, we denote by $\cE := \cE_{\tau h}^\Burgers$ the energy of interest \eqref{eqdef:BurgersEnergy} and by $\Phi := \Phi_{\tau h}^\Burgers$ its domain, see \cref{lem:domBurgers}. Likewise we denote by $\cT:=\cT_\tau$, $\cT':=\cT'_\tau$, $\cX := \bT_h$ and $\cX' := \bT_h'$ the centered and staggered grids, in time and space, see \eqref{eqdef:DtDhe}. Finally, we denote 
	 $\|\cdot\|_1 := \|\cdot\|_{\ell^1(\cX)}$ , see \eqref{eqdef:perspective}, and $\mathring \cT := \cT\sm \{0,T\} = \{2\tau,\cdots,T-2\tau\}$.
\end{notation}

Using the exact expansion of the perspective function, presented in \cref{prop:perspective}, we  express the energy of interest as $\cE(\phi+\hat \phi) = \cE(\phi) + \cL(\hat \phi; \phi)+\cQ(\hat \phi; \phi)$ for any $\phi,\hat \phi\in \Phi$ such that $\cE(\phi)<\infty$. The linear and second order (non-quadratic) terms are
\begin{align}
\nonumber
	\cL(\hat \phi; \phi) &:= \tau h 
	\sum_{\substack{t \in \cT' \\ x \in \cX}}
	\sum_{\substack{\sigma_x = \pm\\ \sigma_t=\pm}} \Big[
	u^{\sigma_t}_{\sigma_x}(t,x) \big(\hat m(t,x)-\nu \partial_h \hat \rho(t+\sigma_t \tau,x)\big) 
	-\tfrac{1}{2} u^{\sigma_t}_{\sigma_x}(t,x)^2 \hat \rho(t+\sigma_t \tau,x+\sigma_x h) 
	\Big]
	\\[-3ex] 
\nonumber
	& \hspace{10cm}+ 2h \sum_{x \in \cX} u_0(x) \hat \phi(0,x),\\
\label{eqdef:QBurgers}
	\cQ(\hat \phi; \phi) &:= \tau h 
	\sum_{\substack{t \in \cT' \\ x \in \cX}}
	\sum_{\substack{\sigma_x = \pm\\ \sigma_t=\pm}} 
	\cP\big(\hat m(t,x) - \nu \partial_h \hat \rho(t+\sigma_t \tau,x) - \hat \rho(t+\sigma_t \tau, x+\sigma_x h)\, u^{\sigma_t}_{\sigma_x}(t,x),
	\\[-2.5ex]
\nonumber
	& \hspace{7cm}
	\rho(t+\sigma_t \tau, x+\sigma_x h)+\hat \rho(t+\sigma_t \tau, x+\sigma_x h)\big),
	\\[-5ex]
\nonumber
\end{align}
\begin{align*}
	\text{where }
	m &= \partial_\tau \phi, &
	\rho &= 1-\partial_h \phi, &
	\hat m &:= \partial_\tau \hat \phi, & 
	\hat \rho &:= -\partial_h \hat \phi, &
	u^{\sigma_t}_{\sigma_x}(t,x) := \frac{m(t,x)-\partial_h \rho(t+\sigma_t,x)}{\rho(t+\sigma_t \tau,x+\sigma_x h)},
\end{align*}
where $\sigma_t,\sigma_x \in \{+,-\}$ are signs.
The ratio defining $u^{\sigma_t}_{\sigma_x}$ should in general be understood in the sense of \mbox{
\cref{prop:perspective}}\hskip0pt
; in the proof, however, the denominator $\rho$ is positive for sufficiently small $h>0$ thanks to \mbox{
\cref{assum:Burgers,prop:consistencyBurgers} }\hskip0pt
below.
 Note that $\phi,\hat \phi, \partial_h \rho, \partial_h \hat \rho \in \bR^{\cT \times \cX}$ are centered in time and space, $m,\hat m, u^{\sigma_t}_{\sigma_x} \in \bR^{\cT'\times \cX}$ are staggered in time but centered in space, and $\rho,\hat \rho \in \bR^{\cT \times \cX'}$ are centered in time but staggered in space, since we use staggered finite difference operators $\partial_\tau$ and $\partial_h$, see \eqref{eqdef:DtDhe}.

\begin{proposition}
\label{prop:BurgersStrictCvx}
The energy $\cE$ is strictly convex over the domain $\Phi$. 	
\end{proposition}
\begin{proof}
	We proceed similarly to \cref{prop:porousStrictCvx} and note that $\cE$ is convex since it is a sum of perspective functions (composed with affine mappings) which are convex, but that \emph{strict} convexity again requires an additional argument. Consider $\phi,\hat \phi \in \Phi$ such that $\cE(\phi+s \hat \phi) = \cE(\phi) + s \cL(\hat \phi; \phi) + \cQ(s\hat \phi; \phi)$  is finitely valued and affine w.r.t.\ $s\in [0,1]$, i.e.\ $\cQ(s\hat \phi; \phi)$ is affine. We note that $\hat \phi(T,x)=0$ for all $x\in \cX$, by definition of the domain $\Phi$, and assume for induction that $\hat \phi(T-2n \tau,x) = 0$ for all $x \in \cX$, for some $0 \leq n < N_\tau = T/(2\tau)$. 
	Letting $t := T-(2n+1) \tau$ we obtain $\hat m(t,x) = \hat \phi(t-\tau,x)/(2\tau)$ and $\hat \rho(t+\tau,x+h)=0$ for all $x \in \cX$.		
	Consider the contribution to $\cQ(s\hat \phi; \phi)$ associated with the signs $\sigma_t=+$ and $\sigma_x=+$ (arbitrarily in the case of $\sigma_x$),reading 
	\begin{align*}
		&\cP( s \hat m(t,x) - \nu s \partial_h \hat \rho(t+\tau,x) - s \hat \rho(t+ \tau,x+ h) u^+_+(t,x),\, \rho(t+\tau,x+h) + s\hat \rho(t+\tau,x+h)) \\
		&= \frac{s^2}{2 \tau} \cP(\hat \phi(t-\tau,x), \rho(t+\tau,x+h)).
	\end{align*}
	By assumption, this is an affine function of $s\in [0,1]$. Therefore $0 = \hat \phi(t-\tau,x) = \hat \phi(T-2\tau(n+1),x)$, hence $\hat \phi=0$ by an immediate induction, 
	thus the announced strict convexity.
\end{proof}

The proof outline of \mbox{
\cref{th:cvPorous}}\hskip0pt
, described \mbox{
\cref{sec:cv} }\hskip0pt
for the QPME, also applies to \mbox{
\cref{th:cvBurgers} }\hskip0pt
and Burgers' equation, with natural adjustments for the intermediate results \eqref{eq:linL1} see \mbox{
\cref{subsec:LBurgers}}\hskip0pt
, and \eqref{eq:implDiff} and \eqref{eq:phiR} see \mbox{
\cref{subsec:QBurgers}}\hskip0pt
.

\subsection{Upper estimate of the linear term}
\label{subsec:LBurgers}
We estimate in this subsection the linear term $\cL(\hat \phi; \phi)$ in terms of the $\ell^1(\cX)$ norm of the perturbation $\hat \phi$, establishing in \eqref{eq:linL1Burgers} the first inequality \eqref{eq:linL1} from the proof template, under \cref{assum:Burgers}.
For that purpose we rewrite $\cL(\hat\phi; \phi)$ as finite differences of $\phi$ tested against $\hat \phi$, using discrete summations by parts in time and space, and use the consistency and the symmetry of the proposed numerical scheme. 
Recall that the finite difference operator  $\partial_h$ is skew-adjoint, in the sense that: 
\begin{equation}
\label{eq:DhAdj}
	\sum_{x \in \cX} u(x) \partial_h v(x+h) = \frac{1}{2 h} \sum_{x \in \cX} u(x) v(x+2h)-\frac{1}{2 h}\sum_{x \in \cX} u(x) v(x) = - \sum_{x \in \cX} v(x) \partial_h u(x-h),
\end{equation}
for all $u,v \in \bR^\cX$,
similarly to \cref{lem:LhAdj}. 
We next rewrite $\cL(\hat \phi; \phi)$ using \eqref{eq:DhAdj} and the discrete summation by parts w.r.t.\ time of \cref{lem:timeIPP}:  $-\cL(\hat \phi; \phi) =$ 
\begin{align}
\nonumber
	&
	h \sum_{x \in \cX} \sum_{\sigma_x=\pm} \Big[
	\frac{u^-_{\sigma_x}(\tau,x)+u^+_{\sigma_x}(\tau,x)}{2} \hat \phi(0,x) 
	+ \tau \nu u^-_{\sigma_x}(\tau,x) \partial_h \hat \rho(0,x) 
	+ \frac{\tau}{2} u^-_{\sigma_x}(\tau,x)^2 \hat \rho(0,x+\sigma_x h) 
	- u_0(x) \hat \phi(0,x)\Big]
	\\
\nonumber
	&+ \tau h 
	\sum_{\substack{t \in \mathring\cT \\ x \in \cX}}
	\sum_{\substack{\sigma_x = \pm\\ \sigma_t=\pm}} \Big[
	\partial_\tau u^{\sigma_t}_{\sigma_x} (t,x) \hat \phi(t,x) 
	+ u^{\sigma_t}_{\sigma_x} (t-\sigma_t \tau,x) \,\nu \partial_h \rho(t,x) 
	+ \frac{1}{2} u^{\sigma_t}_{\sigma_x} (t-\sigma_t \tau ,x)^2 \hat \rho(t,x+\sigma_x h)
	\Big]
	\\
\nonumber
	&
	=h \sum_{x \in \cX} \sum_{\sigma_x=\pm} \Big(
	\frac{u^-_{\sigma_x}(\tau,x)+u^+_{\sigma_x}(\tau,x)}{2}
	-\tau \nu \partial_{hh} u^-_{\sigma_x}(\tau,x)
	+ \frac{\tau}{2} \partial_h [(u_{\sigma_x}^-)^2](\tau,x-\sigma_x h) 
	- u_0(x)\Big) \hat \phi(0,x) 
	\\
\label{eq:burgersIPP}
	&+\tau h\!
	\sum_{\substack{t \in \mathring\cT \\ x \in \cX}}
	\sum_{\substack{\sigma_x = \pm\\ \sigma_t=\pm}} \!
	\Big(
	\partial_\tau u^{\sigma_t}_{\sigma_x} (t,x)
	- \nu \partial_{hh} u^{\sigma_t}_{\sigma_x}(t-\sigma_t \tau,x)  
	+ \frac{1}{2} \partial_h\big[ (u^{\sigma_t}_{\sigma_x})^2\big] (t-\sigma_t \tau ,x-\sigma_x h)
	\Big) \hat \phi(t,x).
\end{align}
In the rest of this section, we work under \cref{assum:Burgers}, and thus assume that $\phi \in C^\infty([0,T]\times \bT^d,\bR)$ is a smooth function which solves the continuous problem.
We recognize in \eqref{eq:porousIPP0} and \eqref{eq:porousIPP1} a discrete version of the primal PDE we started from, namely Burgers' equation (\ref{eq:QPME_Burgers}, right). 
The next step, achieved in \mbox{
\cref{prop:consistencyBurgers}}\hskip0pt
, is to prove that our discretization choice implies that $u^\sigma$, $\sigma=\pm$, which is defined in terms of $\phi$, solves this discrete PDE up to a residue quadratically small in $\tau$ and $h$.
We distinguish the finite difference
$m := \partial_\tau \phi$ fromthe time derivative $\bm := \partial_t \phi$, and likewise
\begin{align*}
\hspace{-2mm}
	\rho &:= 1-\partial_h \phi, &
	\brho &:= 1-\partial_x \phi, &
	u^{\sigma_t}_{\sigma_x}(t,x) &:= \frac{m(t,x)-\partial_h \rho(t+\sigma_t,x)}{\rho(t+\sigma_t \tau,x+\sigma_x h)}, &
	\bu &:= \frac{\bm-\partial_x \brho}{\brho}.
\end{align*}
Note that $m$ and $u^{\sigma_t}_{\sigma_x}$ are defined over $[\tau,T-\tau]\times \bT$ whereas the other quantities are defined over the full domain $[0,T] \times \bT$. By \cref{assum:Burgers} one has $\brho>0$ uniformly, and $\bu$ obeys Burger's equation $\partial_t \bu + \tfrac 1 2 \partial_x \bu^2 = \nu \partial_{xx} \bu$ for some non-negative viscosity coefficient $\nu \geq 0$.

\begin{proposition}
\label{prop:consistencyBurgers}
Under \cref{assum:Burgers} and with the notations $m,\rho,u^{\sigma_t}_{\sigma_x},\bm,\brho,\bu$ above:
\begin{align*}
	m &= \bm + \cO(\tau^2), &
	\rho &= \brho + \cO(h^2), &
	u^{\sigma_t}_{\sigma_x} &= \bu + \cO(\tau+h), &
	\!\!\!\sum_{\sigma_t,\sigma_x=\pm} u^{\sigma_t}_{\sigma_x} &= 4 \bu + \cO(\tau^2+h^2),
\\[-6ex]
\end{align*}
\begin{align*}
	\partial_\tau u^{\sigma_t}_{\sigma_x} &= \partial_t \bu +\cO(\tau+h),&
	\partial_{hh} u^{\sigma_t}_{\sigma_x} &= \partial_{xx} \bu +\cO(\tau+h),&
	\partial_h (u^{\sigma_t}_{\sigma_x})^2 &= \partial_x \bu^2 + \cO(\tau+h),&
\end{align*}
\begin{align*}
	\sum_{\substack{\sigma_x = \pm\\ \sigma_t=\pm}} \partial_\tau u^{\sigma_t}_{\sigma_x} &= 4 \partial_t \bu + \cO(\tau^2+h^2), &
	\sum_{\substack{\sigma_x = \pm\\ \sigma_t=\pm}} \partial_{hh} u^{\sigma_t}_{\sigma_x}(t-\sigma_t \tau,x) &= 4\partial_{xx}\bu(t,x)+\cO(\tau^2+h^2),
\end{align*}
and $
\sum^{\sigma_x=\pm}_{\sigma_t=\pm} 
\partial_h\big[ (u^{\sigma_t}_{\sigma_x})^2\big] (t-\sigma_t \tau ,x-\sigma_x h) 
= 4\partial_x\bu^2 + \cO(\tau^2+h^2)$. 
The $\cO$ notation is understood in the $L^\infty$ norm over the common domain of definition. The signs $\sigma_t=\pm$ and $\sigma_x=\pm$ are arbitrary in the above expressions unless they are bound to a sum symbol.
\end{proposition}

\begin{proof}
It is sufficient in this proof to assume that $\phi\in C^5$, rather than $\phi\in C^\infty$, hence \cref{assum:Burgers} may be slightly weakened.
We follow the proof template of \cref{prop:consistency}, to which we refer for details. 
As before, the result follows from the consistency of the finite differences schemes $\partial_\tau$ and $\partial_h$, and the second order error terms are obtained when the expressions are invariant under the reflections $\tau\mapsto -\tau$ and $h\mapsto -h$. 
The time and space shifts in the definition of $u^{\sigma_t}_{\sigma_x}$ break these symmetries, but they are restored when considering averages of over the four possible sign combinations $\sigma_t=\pm$ and $\sigma_x=\pm$. 
Note also that $\brho>0$ by assumption, hence $\rho>0$ for sufficiently small $|h|$, and thus the ratio $u^{\sigma_t}_{\sigma_x}$ is well defined.
\end{proof}

Inserting the expansions of \cref{prop:consistencyBurgers} in \eqref{eq:burgersIPP} yields 
 \begin{align}
\nonumber
	-\cL(\hat \phi; \phi) &
	= 2h \sum_{x \in \cX} \Big(\bu(\tau,x) - \nu\tau \partial_{xx} \bu(\tau,x) +\frac{\tau}{2} \partial_x \bu^2 (\tau,x)-u_0(x) + \cO(\tau^2+h^2)\Big) \hat \phi(0,x) \\
\label{eq:linL1Burgers}
	&+ 4\tau h \sum_{t \in \mathring\cT} \sum_{x \in \cX} \Big(\partial_t \bu(t,x)-\nu \partial_{xx}\bu (t,x)+\frac{1}{2} \partial_x \bu^2(t,x)+ \cO(\tau^2+h^2)\Big) \hat \phi(t,x).
\end{align}
Recalling that $\partial_t \bu + \tfrac 1 2 \partial_x \bu^2 = \nu \partial_{xx} \bu$ we obtain \eqref{eq:linL1} as desired.

\subsection{Lower estimate of the quadratic term}
\label{subsec:QBurgers}

The first part of this subsection estimates the solution of a backwards in time drift diffusion equation, discretized in an upwind and implicit manner see \eqref{eq:implLinBurgers} below, establishing the counterpart of \eqref{eq:phiR} from the proof template in Burgers' case. The second part of this subsection establishes (\ref{eq:EstimLR}, right) in Burgers's case, thus filling the gaps of the proof template and establishing our convergence result \mbox{
\cref{th:cvBurgers}}\hskip0pt

The second  order term $\cQ(\hat \phi; \phi)$ defined in \eqref{eqdef:QBurgers} is a sum of perspective functions, whose first argument reads, for all $t \in \cT'$, $x \in \cX$, and any signs $\sigma_t,\sigma_x=\pm$,
\begin{align}
\label{eq:BurgersLinScheme}
	r^{\sigma_t}_{\sigma_x}(t,x) &:= 
	\hat m (t,x)-\nu \partial_h \hat \rho(t+\sigma\tau,x) 
	- u^{\sigma_t}_{\sigma_x}(t,x) \hat \rho(t+\sigma_t \tau,\, x+\sigma_x h),\\
\label{eq:BurgersLinScheme2}
	&=\partial_\tau \hat \phi(t,x) + \nu \partial_{hh} \hat \phi(t+\sigma \tau,x) + u_{\sigma_x}^{\sigma_t} \partial_h \hat \phi(t+\sigma_t \tau, x + \sigma_x h).
\end{align}

\paragraph{Proof of the estimate \eqref{eq:phiR} of $\hat \phi$ in terms of $r_+^-$.}
 Similarly to \cref{subsec:QPorous}, we regard $r^{\sigma_t}_{\sigma_x}$ as the r.h.s.\ of the discretization a time reversed linear PDE, see \eqref{eq:implLinBurgers} below, closely related to \eqref{eq:pdePhi}. 
 The time discretization is explicit if $\sigma_t=+$, and implicit if $\sigma_t=-$. 
Define 
\begin{align}
\label{eqdef:Dhp}
	\partial^+_h u(x) &:= \frac{u(x+2h)-u(x)}{2h} = \partial_h u(x+h), &
	\partial^-_h u(x) &:= \frac{u(x)-u(x-2h)}{2h} = \partial_h u(x-h),
\end{align}
the non-centered first order finite differences operator.
We focus on the parameters $\sigma_t=-$ and $\sigma_x=+$ (implicit scheme, upwind finite difference) which yield the most stable discretization, rename for readability  $u := u^-_+$ and $r := r^-_+$, and regard \eqref{eq:BurgersLinScheme2} as Burgers' counterpart to \eqref{eq:implDiff} from the QPME proof template.
Expanding the time finite difference in \eqref{eq:BurgersLinScheme2} and $\hat \rho := -\partial_h \hat \phi$ in \eqref{eq:BurgersLinScheme} and rearranging terms we obtain (omitting the space variable $x$ for readability)

 \begin{align}	
\label{eq:implLinBurgers}
	\hat \phi(t-\tau) 
	- 2 \tau\,  \nu \partial_{hh} \hat \phi(t-\tau)
	- 2\tau\, u(t)\, \partial_h^+ \hat \phi(t-\tau) &= \hat \phi(t+\tau) - 2 \tau\, r(t), &
	\hat \phi(T)&=0.
\end{align}

\begin{proposition}
\label{prop:implBurgers}
For any $u \in [0,\infty[^\cX$ and any $\nu \geq 0$, the operator $A:=\Id -\nu \partial_{hh} - u \partial_h^+$ is inverse positive. If in addition $\|\partial_h u\|_\infty<1$, then $\|A^{-1} f \|_1 \leq (1-\|\partial_h u\|_\infty)^{-1} \|f\|_1$, $\forall f \in \bR^\cX$. 
\end{proposition}

\begin{proof}
We proceed similarly to the proof of \cref{prop:implLap}, and observe that the matrix of the finite differences operator $A := \Id -\nu \partial_{hh} - u \partial_h^+$ has non-positive off-diagonal entries, and the sum of each row is one. 
Inverse positivity therefore follows from \mbox{
\cref{lem:Minkowski}}\hskip0pt
, as announced.
Now consider $f\geq 0$, and let $g := A^{-1} f$, which is thus non-negative. 
One has 
\begin{equation*}
\!\frac{\|f\|_1}{2h} = \sum_{x \in \cX} f(x) = \sum_{x \in \cX} [g(x) -\nu \partial_{hh} g(x) - u(x) \partial_h^+ g(x)]
= \sum_{x \in \cX} g(x)(1+\partial_h^- u(x)) \geq \frac{\|g\|_1}{2h} (1-\|\partial_h u\|_\infty),
\end{equation*}

using the adjointness \eqref{eq:DhAdj} of $\partial_h^+$ and $-\partial_h^-$ ,  the fact that $\partial_{hh}$ is self-adjoint and vanishes on constants, and the observation that $\|\partial_h^+ u\|_\infty = \|\partial_h u\|_\infty$ by \eqref{eqdef:Dhp}. We conclude, similarly to \cref{prop:implLap}, by splitting an arbitrary $f \in \bR^\cX$ into a positive and a negative part and using the linearity of $A^{-1}$.
\end{proof}

Using the discrete Duhamel formula, or simply an induction on the number $N=0,\cdots,N_\tau := T/(2\tau)$ of timesteps, we obtain assuming that $u\geq 0$
\begin{equation*}
	\hat \phi(T-2\tau\, N) = -2 \tau \sum_{0 \leq n < N} \Big[\prod_{n \leq k < N} \big(1-2\tau\, \nu \partial_{hh} - 2\tau\, u(T-(2k+1)\tau) \partial_h^+\big) \Big]^{-1} r(T-(2n+1)\tau)
\end{equation*}
The estimate \eqref{eq:phiR} from the proof template follows, with the constant 
\begin{equation*}
	\prod_{t \in \cT'} (1-2\tau \|\partial_h u\|_\infty)^{-1} \leq (1-\tau K)^{-N_\tau} \leq C_\cQ := \exp(T K), \text{ with } K := 2\max_{t \in \cT'} \|\partial_h u\|_\infty,
\end{equation*}
and assuming that $\tau K \leq 1/2$, so that $(1-\tau K)^{-1} \leq \exp(2 \tau K)$, in addition to $u \geq 0$.

Let us now show that the latter two conditions hold  under \cref{assum:Burgers}, which states that $\phi \in C^\infty([0,T]\times \bT,\bR)$ is a smooth function, defined in terms of Burgers' solution $\bu>0$ itself assumed to be smooth and positive. 
Recalling that $u^{\sigma_x}_{\sigma_t} = \bu + \cO(\tau+h)$ by \cref{prop:consistencyBurgers}, we choose $\tau_*>0$ and $h_*>0$ such that $u^{\sigma_x}_{\sigma_t} \geq 0$ on $[\tau,T-\tau]\times \bT$ for all $\tau \in ]0,\tau_*]$ and $h \in ]0,h_*]$, and thus in particular  $u := u^-_+\geq 0$. 
By \cref{prop:consistencyBurgers}, one has $\partial_h u(t) = \partial_x \bu + \cO(\tau+h)$, hence $K$ is bounded independently of $\tau$ and $h$, and up to reducing the maximum timestep $\tau_*>0$ we may assume that $\tau_* K \leq 1/2$ as desired.

\paragraph{Proof of the lower bound (\ref{eq:EstimLR}, right).}
The argument is identical to \cref{subsec:QPorous}, with the sole exception of \cref{lem:int1} which needs to be replaced with the following counterpart, whose proof is trivial hence omitted.
\begin{lemma}
	Let $\phi \in \bR^\cX$ and $\rho := 1 - \partial_h \phi$. If $\rho \geq 0$ then $\|\rho\|_1=1$. 
\end{lemma}

\section{Arithmetic mean based scheme}
\label{sec:arith}
We consider in this appendix a variant $\cE^\theta_{\tau h}$, where $\theta\in [0,1]$, of the energy \eqref{eqdef:porousEnergy} associated with the discretization of the QPME. It is defined as follows, with the notations of \cref{lem:domPorous}
\begin{align}
\label{eq:arithEnergy}
	\cE^\theta_{\tau h}(m,\rho) &:= 2 \tau (2h)^d \sum_{x \in \cX} \sum_{t \in \cT'} 
	\frac{m(t,x)^2}{2\, \cI_\tau^\theta \rho(t,x)} - m(t,x) u_0(x), \\
\label{eq:timeInterpolation}
	\text{where }
	\cI_\tau^\theta \rho(t,x) &:= (1-\theta) \rho(t-\tau,x)+\theta \rho(t+\tau,x).
\end{align}
The ratio in \eqref{eq:arithEnergy} should be understood as an instance of the perspective function \eqref{eqdef:perspective}, and $\cI_\tau^\theta$ thus denotes a linear time interpolation operator.
By convention, $\cE_{\tau h}^\theta(\phi) := \cE_{\tau h}^\theta(\partial_\tau\phi,1+L_h\phi)$, for all $\phi \in \Phi^\Porous_{\tau h}$. For readability, we denote $\cE^\theta := \cE^\theta_{\tau h}$ and $\cI^\theta := \cI_\tau^\theta$ in the following, and we also adopt \cref{not:porous}: $\Phi := \Phi^\Porous_{\tau h}$, $\cT := \cT_\tau$, $\cT':=\cT_\tau'$, $\cX := \bT_h^d$, etc. 

Using the exact expansion of the perspective function, see \cref{prop:perspective}, we obtain 
$\cE^\theta(\phi+\hat \phi) = \cE^\theta(\phi) + \cL^\theta(\hat \phi; \phi) + \cQ^\theta(\hat \phi; \phi)$ for any $\phi,\hat \phi \in \Phi$ such that $\cE^\theta(\phi)<\infty$,  where 
\begin{align*}
	\cL^\theta(\hat \phi; \phi) &
	:= 2\tau (2h)^d \sum_{x\in \cX} \sum_{t \in \cT'} \Big(
	u(t,x) \hat m(t,x) - \tfrac{1}{2}u(t,x)^2\cI^\theta \hat \rho(t,x) 
	\Big)
	+ (2h)^d \sum_{x \in \cX} u_0(x) \hat \phi(0,x),\\
	\cQ^\theta(\hat \phi; \phi) &
	:= 2\tau (2h)^d \sum_{x\in \cX} \sum_{t \in \cT'} 
	\cP(\hat m(t,x) - u(t,x) \cI^\theta \hat \rho(t,x),\, \cI^\theta \rho(t,x) + \cI^\theta \hat \rho(t,x)),\\
	\text{with }&
	m := \partial_\tau \phi,\ 
	\rho := 1+L_h \phi,\ 
	\hat m := \partial_\tau \hat \phi,\ 
	\hat \rho := L_h \hat \phi, 
	\text{ and } u(t,x) := m(t,x) / \cI^\theta \rho(t,x).
\end{align*}
The ratio defining $u$ should be understood in the sense of \mbox{
\cref{prop:perspective} }\hskip0pt
in general; note nevertheless that the denominator $\cI^\theta \rho(t,x)$ is positive for sufficiently small $h>0$ under  \mbox{
\cref{assum:Burgers}}\hskip0pt
. The arithmetic interpolation operator w.r.t.\ time satisfies 
\begin{equation*}
	\sum_{t\in \cT'} \cI_\tau^\theta f(t) \, g(t) 
	= (1-\theta) f(0) g(\tau) + \theta f(T) g(T-\tau) 
	+  \sum_{t \in \mathring \cT} f(t) \, \cI_\tau^{1-\theta}\, g(t),
\end{equation*}
for any centered $f : \cT \to \bR$ and staggered $g : \cT'\to \bR$, recalling that $\mathring \cT := \cT\sm \{0,T\}$. 
In combination with the summation by parts formulas of \cref{lem:timeIPP,lem:LhAdj} we obtain
\begin{align}
\nonumber
	-\cL^\theta(\hat \phi; \phi) &= (2h)^d \sum_{x \in \cX} 
	\Big(
	u(\tau,x)\hat \phi(0,x)+ \tau (1-\theta)  \, u(\tau,x)^2 \hat \rho(0,x) - u_0(x) \hat \phi(0,x)
	\Big)\\
\nonumber
	&+ 2 \tau (2h)^d \sum_{x \in \cX} \sum_{t \in \mathring \cT} \Big(\partial_\tau u(t,x) \hat \phi(t,x) + \frac{1}{2} \hat \rho(t,x)\, \cI^{1-\theta} u(t,x)^2\Big)\\
\label{eq:arithIPP1}
	&= (2h)^d \sum_{x \in \cX} \Big(u(\tau,x) + \tau(1-\theta) L_h u^2(\tau,x) - u_0(x)\Big)\hat \phi(0,x)\\
\label{eq:arithIPP2}
	&+2\tau (2h)^d  \sum_{x \in \cX} \sum_{t \in \mathring \cT} 
	\Big(\partial_\tau u(t,x) + \frac{1}{2} \cI^{1-\theta} L_h u^2(t,x)\Big)\hat \phi(t,x).
\end{align}
The map $\phi \in \Phi$ is a critical point of $\cE^\theta$ iff all the coefficients of $\hat \phi$ vanish in \eqref{eq:arithIPP1} and \eqref{eq:arithIPP2}, i.e.
\begin{align}
\label{eq:arithScheme}
	u(\tau) + \tau(1-\theta) L_h u^2(\tau) &= u_0, &
	\partial_\tau u(t) + \frac{1}{2} \cI^{1-\theta} L_h u^2(t) &= 0,\, \forall t \in \mathring\cT.
\end{align}
This amounts, it turns out, to an implicit ($\theta=0$), explicit ($\theta=1$), or second order midpoint ($\theta=1/2$) discretization of the QPME \eqref{eq:QPME_Burgers}. The numerical solution $u$ can thus be computed one timestep at a time, and there is no point in minimizing the energy \eqref{eq:arithEnergy} using a global space-time method. In contrast, the coefficient of $\hat \phi(t,x)$ is more complex in the proposed discretization \eqref{eq:porousIPP1}, and  the corresponding criticality condition can be reformulated as 
\begin{equation}
\label{eq:harmonicODE}
	\partial_\tau \tilde u(t) 
	+ \frac{1}{4} L_h \Big( \tilde u(t-\tau)^2\, \frac{\cH \rho(t-\tau)^2}{\rho(t)^2}\Big)
	+ \frac{1}{4} L_h \Big( \tilde u(t+\tau)^2 \, \frac{\cH \rho(t+\tau)^2}{\rho(t)^2}\Big)
	= 0,\, \forall t \in \mathring \cT,
\end{equation}
where $\cH \rho(t) := 2(\rho(t-\tau)^{-1} + \rho(t+\tau)^{-1})^{-1}$ denotes the harmonic interpolation of densities w.r.t.\ time, and $\tilde u := m / \cH\rho$. 
If $\rho$ varies slowly w.r.t.\ the time variable $t$, and $\theta=1/2$, then \eqref{eq:harmonicODE} is close in spirit with (\ref{eq:arithScheme}, right).
However, in contrast with \eqref{eq:arithScheme}, it does not seem possible to solve \eqref{eq:harmonicODE} successively one timestep at a time, especially in view of the terminal boundary condition $\rho(T)=1$, and of the initial condition for $u(\tau)$ derived from \eqref{eq:porousIPP0}. 
This problem formulation is reminiscent of mean field games, see the discussion in \cite[Chapter 5]{brenier2020examples}.

Turning our attention to the second order term $\cQ^\theta$, we see that the first argument
\begin{equation*}
	r(t) := \hat m(t) - u(t) \cI^\theta \hat \rho(t) 
	= \frac{\phi(t+\tau)-\phi(t-\tau)}{2 \tau} - (1-\theta) u(t)L_h \phi(t-\tau) - \theta u(t)L_h \phi(t+\tau),
\end{equation*}
can be regarded as the r.h.s.\ of a reversed in time discretized linear PDE closely related with \eqref{eq:pdePhi}, recall the definition \eqref{eq:timeInterpolation} of the time interpolation operator $\cI^\theta$. Equivalently 
\begin{align}
\label{eq:arithPhiPDE}
	\phi(t-\tau) +2 \tau (1-\theta) u(t) L_h \phi(t-\tau) &= \phi(t+\tau) - \theta u(t) L_h \phi(t+\tau) - r(t), &
	\phi(T)&=0.
\end{align}
The scheme is implicit if $\theta=0$, explicit if $\theta=1$, and second order semi-implicit if $\theta=1/2$. 
The proof template of \cref{subsec:QPorous} requires the scheme \eqref{eq:arithPhiPDE} to be stable in the $\ell^1(\cX)$ norm.
Choosing $\theta=0$, and recalling from \cref{prop:implLap} that the implicit scheme \eqref{eq:arithPhiPDE} is indeed $\ell^1$ stable, we can obtain a convergence result for the implicit scheme \eqref{eq:arithScheme} for the QPME (convergence of the dual potential $\phi$, similarly to \cref{th:cvPorous} but with a first order rate $\cO(\tau+h^2)$). This approach unfortunately does not extend to the explicit scheme or to the midpoint scheme \eqref{eq:arithPhiPDE}, which are unstable in the $\ell^1(\cX)$ norm unless a restrictive Courant-Friedrichs-Levy condition is enforced (even though the midpoint scheme is stable in the $\ell^2(\cX)$ norm without CFL).

\section{Uniform convergence of the primal variable, by mollification}
\label{sec:conv}
Our convergence results \cref{th:cvPorous,th:cvBurgers} state that the proposed BBB numerical method produces\footnote{Precisely, the mapping $\phi_{\tau, h} : [0,T] \times \bT^d \to \bR$ can be obtained as the extension by multilinear interpolation of the minimizer $\phi^0_{\tau, h} : \cT_\tau \times \bT_h^d\to \bR$ of the discrete energy.} an approximation $\phi_{\tau, h}$ of the dual potential $\phi$, in the rather weak $L^\infty([0,T],L^1(\bT^d))$ norm, with a second order $\cO(\tau^2+h^2)$ convergence rate w.r.t.\ the half timestep $\tau$ and half gridscale $h$. 
Using a standard regularization by convolution w.r.t.\ the spatial variable, we construct in this appendix an approximation of the primal variable $u$ of the PDE of interest, the QPME or Burgers' viscous equation, which converges in the $L^\infty([0,T]\times \bT^d)$ norm, but with a slower $\cO(h^\frac{2}{d+2})$ convergence rate, when $\tau=h$. This convolution step \emph{is not implemented} in our numerical experiments, since it does not appear to be necessary in practice and only degrades the results.

\begin{definition}[Spatial convolution]
\label{def:conv}
A mollifier is a non-negative, radially symmetric mapping $\rho \in C^2(\bR^d)$ with compact support and unit integral. Spatial mollification is defined as
\begin{align*}
	[\phi \star \rho_\ve](t,x) &:= \int_{\bR^d} \phi(t,x-y) \rho_\ve (y) \diff y,&
	\text{where } \rho_\ve(y) &:= \frac{1}{\ve^d} \rho\big(\frac{y}{\ve}\big),\text{ with } \ve>0.
\end{align*}
\end{definition}
Given $\phi \in C^2([0,T]\times \bT^d)$, one obtains using the symmetry of $\rho$ and a Taylor expansion
\begin{equation*}
	|\phi\star \rho_\ve(t,x)-\phi(t,x)| \leq \int_{\bR^d} |\phi(t,x-y)-\phi(t,x)+y^\top \nabla \phi(t,x)| \rho_\ve(y) \diff y
	\leq \ve^2 C(\rho) \|\nabla^2 \phi\|_\infty,
\end{equation*}
where $\nabla$ and $\nabla^2$ denote the gradient and Hessian operators w.r.t.\ the space variable $x$, and $C(\rho) := \frac{1}{2} \int_{\bR^d} \rho(x) \|x\|^2 \diff x$.

\begin{proposition}
	Let $\phi \in C^4([0,T]\times \bT^d)$, and let $\phi_{\tau,h} \in L^\infty([0,1],L^1(\bT^d))$ be such that $\sup_{t\in[0,T]} \|\phi(t,\cdot)-\phi_{\tau,h}(t,\cdot)\|_{L^1(\bT^d)} = \cO (\tau^2+h^2)$. 	Then with the notations of \cref{def:conv}
\begin{align}
\label{eq:HessPhiError}
	\|\nabla^2\phi - \nabla^2_h (\phi_{\tau,h}\star \rho_\ve)\|_{L^\infty} 
	&= \cO\big(\ve^2 + h^2 + \frac{\tau^2+h^2}{\ve^{2+d}}\big), \\
\label{eq:GradPhiError}
	\|\nabla\phi - \nabla_h (\phi_{\tau,h}\star \rho_\ve)\|_{L^\infty} 
	&= \cO\big(\ve^2 + h^2 + \frac{\tau^2+h^2}{\ve^{1+d}}\big), \\	
\label{eq:DtPhiError}
	\|\partial_t \phi - \partial_\tau (\phi_{\tau,h}\star \rho_\ve)\|_{L^\infty} 
	&=\cO\big(\ve^2 + \tau^2 + \frac{\tau^2+h^2}{\tau\ve^{d}}\big),
\end{align}
	for all $\tau,h,\ve>0$, where the $L^\infty$ norm is over the common domain of definition, and where $\nabla_h$ and $\nabla^2_h$ denote the standard centered finite difference schemes for the gradient and Laplacian. 
\end{proposition}

\begin{proof}
Note that $\nabla_h^2 (\phi_{\tau,h} \star \rho_\ve) = (\nabla_h^2 \phi_{\tau,h}) \star \rho_\ve = \phi_{\tau,h} \star \nabla_h^2 (\rho_\ve)$, and likewise $\partial_\tau (\phi_{\tau,h} \star \rho_\ve) = (\partial_\tau \phi_{\tau,h}) \star \rho_\ve$.
	Focusing first on \eqref{eq:HessPhiError}, we estimate: $\|\nabla^2\phi - \nabla^2_h (\phi_{\tau,h}\star\rho_\ve) \|_{L^\infty}$
\begin{align*}
	& 
	\leq  
	\|\nabla^2 \phi - (\nabla^2 \phi) \star \rho_\ve\|_{L^\infty} +
	\|(\nabla^2\phi - \nabla^2_h \phi)\star \rho_\ve\|_{L^\infty} + 
	\|(\phi - \phi_{\tau,h})\star (\nabla^2_h \rho_\ve)\|_{L^\infty} \\
	& 
	\leq C\ve^2  \|\nabla^4 \phi\|_{L^\infty}  
	+ \|\nabla^2 \phi - \nabla^2_h \phi\|_{L^\infty} \|\rho_\ve\|_{L^1(\bR^d)}
	+ \|\phi-\phi_{\tau,h}\|_{L^\infty([0,T],L^1(\bT^d))} \|\nabla^2_h \rho_\ve\|_{L^\infty(\bR^d)}.
\end{align*}
We conclude the proof of \eqref{eq:HessPhiError} using the second order consistency of the centered finite differences Hessian $\nabla^2_h$, and the identities $\|\rho_\ve\|_{L^1}=\|\rho\|_{L^1}=1$ and $\|\nabla^2_h \rho_\ve\|_{L^\infty} \leq C\|\nabla^2 \rho_\ve\|_{L^\infty} = C\ve^{-(2+d)} \| \nabla^2\rho\|_\infty$, for some absolute constant $C$. The first order case \eqref{eq:GradPhiError} is similar. Considering the time finite difference \eqref{eq:DtPhiError} we compute
\begin{align*}
	\|\partial_t \phi &-\partial_\tau (\phi_{\tau,h}\star\rho_\ve)\|_{L^\infty} 
	\leq \|\partial_t \phi - \partial_\tau \phi\|_{L^\infty} 
	+ \|\partial_\tau \phi - (\partial_\tau \phi) \star \rho_\ve\|_{L^\infty} 
	+ \|[\partial_\tau (\phi-\phi_{\tau,h})] \star \rho_\ve\|_{L^\infty} 
	\\
	& \leq 
	\|\partial_t \phi - \partial_\tau \phi\|_{L^\infty} +
	\ve^2 C(\rho) \|\nabla^2 \partial_\tau \phi \|_{L^\infty} + \frac{1}{\tau}\|\phi-\phi_{\tau,h}\|_{L^\infty([0,T],L^1(\bT^d))}  \|\rho_\ve\|_{L^\infty(\bR^d)},
\end{align*}
and \eqref{eq:DtPhiError} follows using the second order consistency of the staggered finite difference $\partial_\tau$ and the identity $\|\rho_\ve\|_{L^\infty} = \ve^{-d} \|\rho\|_\infty$. 
\end{proof}

Letting $\phi_h := \phi_{\tau,h}\star \rho_\ve$ with $\tau:=h$ and $\ve:=h^{\frac{1}{2+d}}$, we obtain in the QPME case
\begin{align*}
		\|u - u_h \|_\infty &= \cO(h^{\frac{2}{d+2}}),&
		\text{with }
		u &:= \frac{\partial_t \phi}{1-\diver(\cD \nabla \phi)}, &
		\text{and }
		u_h &:= \frac{\partial_\tau \phi_h}{1-\Tr(\cD \nabla_h^2 \phi_h) - \diver(\cD)^\top \nabla_h \phi_h}.
\end{align*}
Likewise, in Burgers' case, with $\tau=h$ and $\ve = h^\frac 1 3$ (since $d=1$)
\begin{align*}
	\|u-u_h\|_\infty &= \cO(h^{\frac{2}{3}}), &
	\text{with }
	u &= \frac{\partial_t \phi + \nu \partial_{xx} \phi}{1-\partial_x \phi}, &
	u_h &= \frac{\partial_\tau \phi_h + \nu \partial_{hh}\phi_h}{1-\partial_{h} \phi_h}.
\end{align*}

\section{Closed form of the Barenblatt Potential}
\label{sec:Barenblatt}

The Barenblatt profile is a closed form solution of the QPME in a homogeneous medium, which is self-similar in time, and compactly supported in space for each time, see \eqref{eqdef:Barenblatt}. 
We describe in this section the  corresponding dual variables $m,\rho,\phi$ in the BBB formulation of the QPME, on some time interval $[T_0,T]$. 
Note that $m,\rho,\phi$ depend on the final time $T$, through the terminal constraints $\rho(T)=1$ and $\phi(T) = 0$, and are thus not self-similar in time.
Interestingly, $x\in \bR^d\mapsto \rho(t,x)$ has discontinuities across two interfaces: $|x|=R(t)$ and $|x|=R_T$, see below.

\begin{proposition}
Consider the Barenblatt profile $u(t,x) := \frac{2}{t^\alpha} \max \{0, \gamma-\frac{\beta}{4} \frac{|x|^2}{t^{2 \beta}}\}$, where $t>0$, $x\in \bR^d$, $\alpha := d/(d+2)$, $\beta := 1/(d+2)$, and $\gamma>0$ is an arbitrary parameter. 
Fix a compact time interval $[T_0,T] \ni t$, where $T>T_0>0$, and define
\begin{align*}
	R(t) &:= \nu t^\beta, &
	\text{with } \nu &:= 2\sqrt{(d+2) \gamma}.
\end{align*}
By construction, $\supp(u(t,\cdot)) = \overline B(0,R(t))$. 
Denote by $R_T := R(T)$ the radius of the support of the Barenblatt solution at the final time.
For $T_0 \leq t \leq T$ we  set 
\begin{equation}
\label{eqdef:phiBarenblatt}
\phi(t,x):=\begin{cases}
\frac{2}{T^{\alpha}} \gamma(t-T) + \frac{|x|^2}{2d}( 1-(\frac{t}{T})^\alpha ) & 
\text{if } |x| < R(t),\\
-2\gamma T^{1-\alpha}+ \frac{|x|^2}{2d} -\frac{1}{d(d+2)R_T^d} |x|^{d+2} & 
\text{if } R(t) < |x| < R_T,\\
0 & \text{if } R_T < |x|.
\end{cases}
\end{equation}
Then $\phi$ is continuous on $[T_0,T]\times \bR^d$, and is $C^\infty$ outside $\Gamma := \{(t,x)\in[T_0,T] \times \bR^d \mid |x| = R(t) \text{ or } |x|=R_T\}$. The distributional derivative $m := \partial_t \phi$ is a continuous function, whereas the distributional derivative $\rho := 1-\Delta \phi$ is a bounded and piecewise continuous function with jumps across $\Gamma$. One has the closed form expressions
\begin{align}
\label{eq:mrhoBarenblatt}
m(t,x) &= \frac{2}{T^\alpha} \max \big\{0,\gamma-\frac{\beta}{4} \frac{|x|^2}{t^{2 \beta}}\big\},&
\rho(t,x)&=\begin{cases}
    (\frac{t}{T})^\alpha , & \text{if $|x|<R(t)$},\\
     2(\frac{|x|}{R_T})^d, & \text{if $R(t)<|x|<R_T$},\\
     1, & \text{ otherwise}.
\end{cases}
\end{align}
Moreover one has $\partial_t \rho = \Delta m$  in the sense of distributions, $\rho(T,\cdot) = 0$, and denoting $u_0 := u(T_0,\cdot)$ one has 
\begin{equation}
\label{eq:nogapBarenblatt}
	\int_{[T_0,T]\times \bR^d} 
	\frac{u^2}{2} = 
	\int_{[T_0,T] \times \bR^d} 
	\Big(- \frac{m^2}{2\rho} + m u_0\Big),
\end{equation}
showing that there is no duality gap in \eqref{eq:gap}, and therefore $(\rho,m)$ solves the ballistic Benamou-Brenier formulation of the QPME. Finally, we observe that $m=\rho u$, as expected from \eqref{eq:pdePhi}.
\end{proposition}

\begin{proof}
Since $\phi$ has a radially symmetric profile, we express it as $\phi(t,x) = \vp(t,|x|)$, where $\vp : [T_0,T] \times [0,\infty[ \to \bR$ is piecewise polynomial.
Our first objective is to show that $\vp$ and $\partial_r \vp$ are continuous across the interfaces $\{r = R(t)\}$ and $\{r=R_T\}$; once this is proved, it follows that the distributional derivatives $m(t,x) = \partial_t \vp(t,|x|)$ and $\rho(t,x) = 1-r^{-(d-1)} \partial_r (r^{d-1} \partial_r \vp) (t,|x|)$ have no singular part (i.e.\ jump term), hence can be evaluated piecewise by direct differentiation, which yields \eqref{eq:mrhoBarenblatt}.
Denoting by $[\vp(t,\cdot)]_r := \lim_{\ve \to 0^+} \vp(t,r+\ve)-\vp(t,r-\ve)$ the jump (if any) at $r$ we compute
\begin{align*}
	-[\vp(t,\cdot)]_{R(t)} &= 
	\frac{2}{T^\alpha} \gamma t - \frac{R(t)^2}{2d} \big(\frac{t}{T}\big)^\alpha + \frac{1}{d(d+2) R_T^d} R(t)^{d+2}
	= \frac{2 \gamma t}{T^\alpha}- \frac{\nu^2 t}{2 d T^\alpha} + \frac{\nu^2 t}{d (d+2) T^\alpha}
	= 0,\\
	-[\partial_r \vp(t,\cdot)]_{R(t)} &= 
	-\frac{R(t)}{d} \big(\frac{t}{T}\big)^\alpha 
	+ \frac{R(t)^{d+1}}{d R_T^d} 
	= \frac{R(t)}{d} \Big(\big(\frac{R(t)}{R_T}\big)^d - \big(\frac{t}{T}\big)^\alpha\Big) = 0,
	\\
	[\vp(t,\cdot)]_{R_T} &= 
	2 \gamma T^{1-\alpha}-\frac{R_T^2}{2d} + \frac{R_T^{d+2}}{d(d+2) R_T^d} = T^{1-\alpha}\Big(2 \gamma - \frac{4 \gamma(d+2)}{2d} + \frac{4 \gamma (d+2)}{d (d+2)}\Big) = 0,
\end{align*}
and finally $[\partial_r \vp(t,\cdot)]_{R_T} = -R_T/d + R_T^{d+1} / (dR_T^d) = 0$. We simply substituted the definition of $R(t)$ and $\nu$, and used the relations $d\beta=\alpha$ and $1-\alpha=2 \beta$ relating the exponents. The identities $\phi(T,\cdot) = 0$ and $m=\rho u$ are clear by construction, and $\partial_t \rho = - \partial_t \Delta \phi = -\Delta m$.

We now turn to the second part of the proof, and obtain for any $x \in \bR^d$, using $m = \rho u = \partial_t \phi$
\begin{equation}	
	\int_{T_0}^T \frac{m(t,x)^2}{\rho(t,x)} - m(t,x) u_0(x)\diff t
	=  \int_{T_0}^T m(t,x) (u(t,x)-u_0(x)) \diff t 
	= - \int_{T_0}^T \phi(t,x) \partial_t u(t,x) \diff t.
\label{eq:dualgap1}
\end{equation}
The integration by parts is justified since $\partial_t \phi = m$ is continuous by \eqref{eq:mrhoBarenblatt}, and $\partial_t u(\cdot,x) \in L^\infty_{\mathrm{loc}}$ is locally bounded on $[T_0,T]$ in view of the explicit expression of the Barenblatt profile \eqref{eqdef:Barenblatt}.

Likewise, for any fixed $t \in [0,T]$, we obtain omitting for readability the arguments $(t,x)$
\begin{align*}
	2\int_{\bR^d}\phi\, \partial_t u 
	= \int_{B(t)} \phi\, \Delta (u^2) = 
	- \oint_{\partial B(t)} u^2 \<\nabla \phi, \mathrm{n}\> + \int_{B(t)} u^2 \Delta \phi 
	= \int_{B(t)} u^2(1-\rho)
	= \int_{\bR^d} \Big(u^2 - \frac{m^2}{\rho}\Big),
\end{align*}
where the open unit ball $B(t) := B(0,R(t))$ is the interior of the support of $u(t,\cdot)$, and $\mathrm{n}(x) := x/\|x\|$ denotes the outward unit vector.
We used the QPME $2\partial_t u = \Delta (u^2)$, the fact that $u$ vanishes on $\partial B(0,R(t))$, the defining identity $1-\rho = \Delta \phi$, and the equality $m = \rho u$.
The integration by parts is justified since $u(t,\cdot)$ and $\phi(t,\cdot)$ are $C^\infty$ on the open ball $B(0,R(t))$, and $u$ and $\nabla \phi(t,x) = \partial_r \vp(t,|x|) \mathrm{n}(x)$ are continuous on the closed ball $\overline B(0,R(t))$, see above. 

Integrating the identity \eqref{eq:dualgap1} w.r.t.\ $x \in \bR^d$, and the last equation w.r.t.\ $t \in [T_0,T]$, we obtain \eqref{eq:nogapBarenblatt} by linear combination, which concludes the proof.
\end{proof}

\section{Validity of the BBB formulation of Burgers' viscous equation}
\label{sec:nogapBV}

We show that the BBB formulation of Burgers' viscous equation is well posed on any time interval $[0,T]$, by a straightforward adaptation of \cite[Proposition 5.2.1]{brenier2020examples} which is devoted to the closely related Hamilton-Jacobi equation $\partial_t u + \frac{1}{2} |\nabla u|^2 = \nu \Delta u$.
In numerical applications one may nevertheless need to restrict the time $T$, for reasons discussed below.

\begin{proposition}
\label{prop:nogapVB}
	Given a smooth initial condition $u_0 : \bT \to \bR$, and a \emph{positive} viscosity coefficient $\nu >0$, consider the solution $u : [0,T] \times \bT \to \bR$ to the viscous Burgers' equation
	$\partial_t u + \partial_x \frac{u^2}{2} = \nu \Delta u$. 
The BBB formulation \eqref{eq:pdeMRho} of this PDE admits a solution $(\rho,m)$, obeying 
	\begin{equation}
		\label{eq:nogapVB}
		\int_{[0,T]\times \bT} 
		\frac{u^2}{2} = 
		\int_{[0,T] \times \bT} 
		\Big(- \frac{(m-\nu \partial_x \rho)^2}{2\rho} + m u_0\Big),
	\end{equation}
	showing that there is no duality gap in \eqref{eq:gap}.  Moreover,  $m=\rho u + \nu \partial_x \rho$ as expected from \eqref{eq:pdePhi}.
\end{proposition}

\begin{proof}
Let us define $\rho$ as the solution of the following backward in time Cauchy problem, involving the solution $u$ of Burgers' viscous equation as a coefficient:
\begin{align}
\label{eq:transportDiffusion}
	\partial_t \rho + \partial_x( \rho u) &= -\nu \Delta \rho, &
	\rho(T,\cdot)=1,
\end{align}
and note that $\rho$ is  \emph{positive} by general properties of transport diffusion equations .
We then define  $m :=\rho u + \nu \partial_x \rho$ and $\phi(t,x) := - \int_t^T m(s,x) \diff s$, for all $t \in [0,T]$.
By construction, there holds 
\begin{align*}
\partial_t \rho + \partial_x m&=0, &
u&=\frac{m-\nu \partial_x \rho}{\rho}, &
\partial_x \phi &= 1-\rho, &
\partial_t \phi &= m.
\end{align*}
By integration by parts, successively in time and space, we obtain
\begin{equation*}
\int_{[0,T]\times \bT} m (u-u_0) = - \int_{[0,T]\times \bT} \phi \partial_t u 
=\int_{[0,T]\times \bT} \big[\phi \partial_x \frac{u^2}{2} - \nu \phi \partial_{xx} u \big]
=\int_{[0,T]\times \bT} \big[(\rho-1) \frac{u^2}{2} + \nu  u \partial_x \rho \big]
\end{equation*}
Using the defining expression of $m$, and the above identity, we conclude that
\begin{equation*}
		\int_{[0,T]\times \bT} \rho u^2 
		= \int_{[0,T]\times \bT} (m-\nu \partial_x \rho) u
		= \int_{[0,T]\times \bT} [m u_0 + (\rho-1) \frac{u^2}{2}],
\end{equation*}
which is equivalent to \eqref{eq:nogapVB}, given that $\rho u^2 = (m-\nu \partial_x \rho)^2/\rho$.
This shows that there is no duality gap in \eqref{eq:gap}, and therefore that $(\rho,m)$ solves the ballistic Benamou-Brenier formulation.
\end{proof}

The positivity of the solution $\rho$ of the transport diffusion equation \eqref{eq:transportDiffusion} is the key ingredient of the proof of \cref{prop:nogapVB}. 
We describe below a heuristic estimate of $\min(\rho)$, which is interesting for numerical considerations, since our BBB implementation is not expected to reliably work if this quantity falls below machine floating point precision. Denoting by $x(t) \in \bT$ the minimizer of $\rho(t,\cdot)$, in such way that $\partial_x \rho(t,x(t))=0$ and $\partial_{xx} \rho(t,x(t))\geq 0$, we obtain formally
\begin{equation*}
	\frac{d}{dt} [\rho(t,x(t))] =
	\partial_t \rho(t,x(t)) =
	- \partial_x (\rho u) (t,x(t)) - \nu \Delta \rho(t,x(t)) \leq 
	- \rho(t,x(t)) \partial_x u(t,x(t)),
\end{equation*}
using the envelope theorem for the first identity, and \eqref{eq:transportDiffusion} for the second one. Therefore, $\rho(t,x(t)) \geq \exp(\int_t^T \partial_x u(s,x(s))\diff s)$ in view of the terminal boundary condition $\rho(T)=1$. From a dimensional analysis, one typically expects that 
$\|\partial_x u\|_\infty \lesssim (\max(u)-\min(u))^2 / (8\nu)$, 
where the factor $1/8$ is introduced to match the steadily propagating wave example $u(t,x)=2c / (1+\exp(c (x-ct)/\nu))$ with speed $c>0$. Note also that $\max(u)=\max(u_0)$, and $\min(u)=\min(u_0)$.  Therefore, we heuristically expect 
$\min(\rho) \succsim \exp(- T/T_*)$ with $T_*:= 8\nu/(\max(u_0)-\min(u_0))^2$.

\end{document}